\documentclass[12pt]{amsart}
\usepackage{a4}
\usepackage{amsmath}
\usepackage{amssymb}
\usepackage{color}
\usepackage{ifthen}
\usepackage{amscd}

\newtheorem{propo}{Proposition}[section]
\newtheorem{corol}[propo]{Corollary}
\newtheorem{theor}[propo]{Theorem}
\newtheorem{lemma}[propo]{Lemma}
\theoremstyle{definition}
\newtheorem{defin}[propo]{Definition}

\theoremstyle{remark}
\newtheorem{remar}[propo]{Remark}

\numberwithin{equation}{section}

\newcommand{\ad }{\mathrm{ad}}
\newcommand{\adR }[2]{#2 \triangleleft #1}
\newcommand{\al }{\alpha }
\newcommand{\Ar }[1]{\Hom (#1)}               % the arrows of a category
\newcommand{\Aut }{\mathrm{Aut}}              % automorphism group

\newcommand{\cC }{\mathcal{C}}

\newcommand{\cG }{\mathcal{G}}
\newcommand{\cha }{\mathrm{char}}
\newcommand{\cK }[1]{\mathcal{K}(#1)}

\newcommand{\Cm }{A}
\newcommand{\cm }{a}

\newcommand{\co }{\mathrm{co}\,}              % coinvariants
\newcommand{\coa }{\delta }                   % coaction

\newcommand{\copr }{\varDelta }               % coproduct
\newcommand{\cou }{\varepsilon }              % counit
\newcommand{\cV }{\mathcal{V}}
\newcommand{\cX }{\mathcal{X}}

\newcommand{\End }{\mathrm{End}}
\newcommand{\eqn}[1]{\overset{ \makebox{\tiny #1} }{=}}
\newcommand{\eqns}[1]{\overset{ \makebox[0pt]{\tiny #1} }{=}}
\newcommand{\fd }{finite-dimensional}
\newcommand{\ffdi }{\mathcal{F}_\theta }      % family of finite-dimensional
\newcommand{\fie }{\Bbbk }
\newcommand{\Fin }{admits all reflections}
\newcommand{\fiso }{\mathcal{X}_\theta }      % family of isoclasses

\newcommand{\gfrak }{\mathfrak{g}}

\newcommand{\Hilb}{\mathcal{H}}
\newcommand{\Hom }{\mathrm{Hom}}
\newcommand{\Homsto }[2][\Wg (M)]{\mathrm{Hom}(#1,#2)}
\newcommand{\Ib }{\mathbb{I}}

\newcommand{\id }{\mathrm{id}}

\newcommand{\lact }{\cdot }

\newcommand{\NA }{\mathcal{B}}
\newcommand{\ndN }{\mathbb{N}}
\newcommand{\ndZ }{\mathbb{Z}}

\newcommand{\Ob }{\mathrm{Ob}}
\newcommand{\op }{\mathrm{op}}
\newcommand{\ot }{\otimes }

\newcommand{\pr }{\mathrm{pr}}

\newcommand{\qfact }[2]{[#1]^!_{#2}}
\newcommand{\qnum }[2]{[#1]_{#2}}
\newcommand{\rcs }{E}

\newcommand{\rf }{r}
\newcommand{\Rf }{R}
\newcommand{\rsys }[1]{\boldsymbol{\Delta }{}^{#1}}
\newcommand{\rersys }[1]{\boldsymbol{\Delta }{}^{#1\,\mathrm{re}}}

\newcommand{\Rmat }{\mathfrak{R}}
\newcommand{\Rmatb }{\overline{\mathfrak{R}}}
\newcommand{\Rwg }{\mathcal{R}}         % root system of a Weyl groupoid
\newcommand{\ls }{\mathrm{span}_\fie }  % linear span

\newcommand{\Wg }{\mathcal{W}}
\newcommand{\yd }{\mathcal{YD}}

\newcommand{\ydH }{ {}^H_H\mathcal{YD}}
\newcommand{\ydU }{ {}^{U^0}_{U^0}\mathcal{YD}}

\renewcommand{\_}[1]{_{(#1)}}
\renewcommand{\^}[1]{^{(#1)}}

\title[Right coideal subalgebras of Nichols algebras]{
Right coideal subalgebras of Nichols algebras and the Duflo order on the Weyl
groupoid}

\author{I. Heckenberger}
\address{Istv\'an Heckenberger,
Mathematisches Institut,
Universit\"at K\"oln,
Weyertal 86-90,
D-50931 K\"oln, Germany}
\email{i.heckenberger@googlemail.com}
\author{H.-J. Schneider}
\address{
Hans-J\"urgen Schneider,
Mathematisches Institut,
Universit\"at M\"unchen,
Theresienstr. 39,
D-80333 Munich, Germany}
\email{Hans-Juergen.Schneider@mathematik.uni-muenchen.de}
\thanks{The work of I.H. is supported by DFG within a Heisenberg fellowship}

\begin{document}

\begin{abstract}
  We study graded right coideal subalgebras
  of Nichols algebras of semisimple Yetter-Drinfeld modules.
  Assuming that the Yetter-Drinfeld module admits all reflections and
  the Nichols algebra is decomposable, we construct an injective order
  preserving and order reflecting map between
  morphisms of the Weyl groupoid and graded right coideal subalgebras of the
  Nichols algebra. Here morphisms are ordered with respect to right Duflo
  order and right coideal subalgebras are ordered with respect to
  inclusion. If the Weyl groupoid is finite, then we prove that the Nichols
  algebra is decomposable and the above map is bijective.
  In the special case of
  the Borel part of quantized enveloping algebras
  our result implies a conjecture of Kharchenko.
\end{abstract}


\maketitle

\section*{Introduction}

It is well-known that quantum groups do not have ``enough'' Hopf subalgebras.
Instead the larger class of right (or left) coideal subalgebras should be
studied. A right coideal subalgebra $E \subset A$ of a Hopf algebra $A$ with
comultiplication $\copr $ is a subalgebra of $A$ with $\copr (E) \subset E
\ot A$.

\vspace{.3\baselineskip}

\textbf{1. Right coideal subalgebras of quantized enveloping algebras
$\boldsymbol{U^{\ge 0}}$.}
Let $\gfrak $ be a 
semisimple complex Lie algebra, 
$\Pi $ a basis of its root system with respect to a fixed Cartan
subalgebra, and
$U=U_q(\gfrak )$
the quantized enveloping algebra of $\gfrak $
in the sense of \cite[Ch.\,4]{b-Jantzen96}.
We assume that $q$ is not a root of unity.
%As an algebra, $U$ is generated by elements
%$K_\al , K_\al ^{-1}, E_\al , F_\al $,
%where $\al \in \Pi $, and relations given in \cite[4.3]{b-Jantzen96}.
Let $U^+$ and $U^{0}$
be the subalgebras of $U$
generated by the sets $\{E_\al \,|\,\al \in \Pi \}$ and $\{ K_\al ,K_\al
^{-1}\,|\,\al \in \Pi \}$, respectively, and let $U^{\ge 0}=U^+U^0$.
For any element $w$ of the Weyl group $W$ of $\gfrak $
let $U^+[w]\subset U^{+}$ be the subspace defined in \cite[8.24]{b-Jantzen96}
in terms of root vectors constructed via Lusztig's automorphisms.
We prove in Thm.~\ref{th:rcsU}, see also Cor.~\ref{details},
the following:

\vspace{.3\baselineskip}

\noindent{\em The map $w \mapsto U^+[w]U^0$ defines an order preserving bijection between $W$ and the set of all right coideal subalgebras of
  $U^{\ge 0}$ containing $U^0$, where right coideal subalgebras are ordered by inclusion and $W$ is ordered by the Duflo order.
If $E_1 \subset E_2$ are right coideal subalgebras of $U^{\ge 0}$ containing $U^0$, then $E_2$ is free over $E_1$ as a right module.}

\vspace{.3\baselineskip}

Recall that
%the (right) Duflo order $\leq_D$
%on the Weyl group $W$ is defined as follows.
if $w_1,w_2$ are elements in $W$, then $w_1 \leq_D w_2$
in the (right) Duflo order
if and only if any
reduced expression of $w_1$ can be extended to a reduced expression of $w_2$
beginning with $w_1$.

In particular, the number of right coideal subalgebras of $U^{\ge 0}$
containing $U^0$ is equal to the order of the Weyl group $W$. This last
statement was conjectured by Kharchenko in
\cite{p-Khar09} for
simple Lie algebras $\gfrak $. The conjecture was proven
for $\gfrak $ of type $A_n$ \cite{a-KharSaga08},
$B_n$ \cite{p-Khar09} and $G_2$ \cite{a-Pogor09}
by combinatorial calculations using Lyndon words. In these papers right
coideal subalgebras are classified in terms of certain
subsets of positive roots.
%In his approach the Weyl group
%itself does not appear. It only turns out in the end that the order of the
%Weyl group coincides with the number of coideal subalgebras.

The subspaces $U^+[w]\subset U^+$ are familiar objects in quantum groups.
Among others, they are used by Lusztig
\cite{b-Lusztig93} to establish a PBW basis for $U^+$, by De Concini, Kac and
Procesi \cite{inp-dCKP95}
to introduce quantum Schubert calculus,
and are identified by Yakimov \cite{p-Yakim09}
as quotients of quantized Bruhat cell translates
\cite{b-Joseph,a-Gorel00}.
It was essentially well-known that $U^+[w]U^0$ is a right coideal subalgebra of
$U^{\ge 0}$: proofs and indications in this direction can be found in
\cite{a-LevSoi90}, \cite[2.2]{inp-dCKP95}, \cite[9.3]{inb-dCP92},
\cite{a-AJS94}. The arguments often use case by case considerations and
reduction to the rank two case, and sometimes they work only in the $h$-adic
setting.
The algebras $U^+[w]$ are known to depend only on $w$
and not on the chosen reduced expression of $w$,
see e.\,g.~\cite[40.2.1]{b-Lusztig93} and \cite[8.21]{b-Jantzen96}.
With our systematic approach to graded right coideal subalgebras
we offer a new way to study $U^+$
without the usual case by case considerations, and
intrinsically
characterize
the algebras
$U^+[w]$ and their ordering with respect to inclusion.
With the necessary modifications,
our results also apply to the small quantum groups
of semisimple Lie algebras where $q$ is a root of unity, and to multiparameter
versions of $U$, see Cor.~\ref{co:WKfin} and Rem.~\ref{re:small}.

\vspace{.3\baselineskip}

\textbf{2. Right coideal subalgebras of Nichols algebras.}
The paper is written in the very general context of Nichols algebras $\NA(M)$
of semisimple Yetter-Drinfeld modules $M \in \ydH$, where $H$ is an arbitrary
Hopf algebra with bijective antipode.
Nichols algebras, also called quantum symmetric algebras,
see \cite{a-Rosso98},
% over group algebras
appear as fundamental objects
in the classification theory of
% pointed
Hopf algebras \cite{a-AndrSchn98,inp-AndrSchn02,a-AndrSchn05p},
in particular of Hopf algebras which are generated by
group-like and skew-primitive elements.
%For an introduction to the subject
%we refer to \cite{a-AndrSchn98,inp-AndrSchn02}.
For example, in the setting of quantized enveloping algebras,
$M = \oplus_{\al \in \Pi} \fie E_{\al}$ is a Yetter-Drinfeld
module over $U^0$, and $\NA(M) = U^+$.
%Thus in the special case of the quantum groups
%$U_q(\mathfrak{g})$ our proofs do not depend on case by case considerations or
%reduction to rank 2.
Finite-dimensional Nichols algebras of diagonal type
are classified in \cite{a-Heck09a}.
%under the assumption that the underlying
%Yetter-Drinfeld module is of diagonal type.
Recently, much progress in the understanding of finiteness properties of
Nichols algebras of nonabelian group type
has been achieved, see e.\,g.~\cite{p-AFGV09,p-AFGV09b},
\cite{p-HeckSchn08a} and references therein.

In rather general situations
(if $M$ admits all reflections, see Sect.~\ref{sec:Constr})
one can associate a Weyl groupoid $\Wg (M)$ to $M$, see \cite{p-AHS08},
\cite{p-HeckSchn08a}.
In case of the Borel part of a quantized Kac-Moody algebra $\gfrak $,
$\Wg (M)$ is essentially the Weyl group of $\gfrak $.
Under the assumption that $\Wg (M)$ is finite,
we prove in Cor.~\ref{details} a PBW-theorem
for the Nichols algebra $\NA(M)$ and its right
coideal subalgebras, where the subalgebra generated by a root vector in the
quantum group case is replaced by the Nichols algebra of a finite-dimensional
irreducible Yetter-Drinfeld module. As a consequence we can show that the
real roots associated with $\Wg(M)$ satisfy the axioms of a root system
in the sense of \cite{a-HeckYam08}, see also \cite{p-HeckSchn08a}.
In Thm.~\ref{th:comm}
we provide generalizations of results of
Levendorskii and Soibelman \cite{a-LevSoi90,a-LevSoi91} on
coproducts and commutators of root vectors.
Our proofs are new even for $U^+$,
since they are free of case by case calculations, and do not 
use the braid relations for Lusztig's automorphisms.

We note that a PBW-theorem for right coideal subalgebras of character Hopf
algebras (where the braiding is diagonal)
is obtained by Kharchenko in terms of Lyndon words, see
\cite{a-Khar08}. For Nichols algebras of diagonal type a PBW theorem in the
spirit of Lusztig was proven by the first author and Yamane
\cite{p-HeckYam08}.

The main results in this paper rely
on the crucial coproduct formula in Thm.~\ref{th:Omega}. For
quantum groups this formula amounts to an explicit computation of the braided
coproduct of $U^+$ in the image of $T_{\al}(U^+)$ as a subalgebra of $U$.
Our formula has the advantage to involve only algebra maps, and hence it is
well-suited to study coideal subalgebras.

To provide more details, let $\theta \in \ndN$,
let $M_1 ,\dots, M_{\theta}$ be finite-dimensional irreducible objects in
$\ydH$, and $M = (M_1,\dots,M_{\theta})$. The goal is to understand the
Nichols algebra
$$\NA(M) = \NA(M_1 \oplus \cdots \oplus M_{\theta})$$
as a Hopf algebra in the braided category $\ydH$.
 Let $\ndZ ^\theta $ be the free abelian group of rank $\theta$ with
 standard basis $\al_1,\dots,\al_{\theta}$. The Nichols algebra $\NA(M)$ is
 $\ndZ ^\theta $-graded where $\deg(M_i) = \al_i$ for all $1 \leq i \leq
 \theta$.

First we define reflection operators $\Rf _i$,
$1 \leq i \leq \theta$. Assume that for all $j \neq i$,
$$a_{ij}^M = -\max \{m\,|\,(\ad \,M_i)^m(M_j)\not=0\} < \infty.$$
Define $a_{ii}^M =2$. Then $(a^M_{ij})_{i,j\in \{1,\dots,\theta \}}$
is a generalized Cartan matrix. Let $s_i^{M}\in \Aut (\ndZ ^\theta )$ be
the corresponding reflection.
Define $\Rf _i(M)_i=M_i^*$, and
  $$ \Rf _i(M)_j=(\ad\,M_i)^{-a^M_{ij}}(M_j) \text{ for all }j \neq i,$$
   and let
  $\Rf _i(M)= (\Rf _i(M)_1,\dots,\Rf _i(M)_{\theta})$.
Finally let $K_i^M=\NA (M)^{\co \NA (M_i)},$
 where the coinvariant elements are defined with respect to the projection of $\NA(M)$ onto $\NA(M_i)$.
By \cite[Thm.\,3.12]{p-AHS08} there is an algebra isomorphism
$$\Omega _i^M:K_i^M\#\NA (M_i^*)\to \NA (\Rf _i(M))$$
which is the identity on all $\Rf_j(M)_j\subset
K_i^M\#\NA (M_i^*)$.
By the coproduct formula in
Thm.~\ref{th:Omega}, $\Omega _i^M$ becomes an
isomorphism of $\ndZ ^\theta $-graded braided Hopf algebras.

In Prop.~\ref{pr:JantzenT} we show that in the quantum group case the
inverse of $\Omega _i^M$ can be identified with Lusztig's automorphism
$T_{\al_i}$ restricted to $U^+$.

Assume that all iterations of the construction $M \mapsto \Rf_j(M)$ are
well-defined. In \cite{p-AHS08}, \cite[Thm. 6.10]{p-HeckSchn08a} the Weyl
groupoid $\Wg(M)$ of $M$ is defined. The objects of $\Wg(M)$ are sequences of
isomorphism classes $[N]=([N_1],\dots,[N_{\theta}])$ where the sequence of
Yetter-Drinfeld modules $(N_1,\dots,N_{\theta})$ is obtained from $M$ by
iterating the operations $\Rf_j$. The morphisms are generated by elementary
reflections $s_i^N : \Rf_i(N) \to N$.
%We introduce the notation
%\begin{align*}
%  \Homsto {[M]}=\mathop{\cup }_{N\in \Ob (\Wg )}\Hom (N,M) \quad
%  \text{(disjoint union).}
%  \label{eq:Homsto}
%\end{align*}
Then our main result
on right coideal subalgebras
in the general case,
see Thm.~\ref{th:finite} and Cor.~\ref{co:WKfin},
says the following.

\medskip

{\em Assume that the Weyl groupoid of $M$ is finite.
  Then there exists an order preserving bijection $\varkappa ^M$
  between the set of morphisms
  of $\Wg (M)$ with target $[M]$
  and the set of $\ndN _0 ^\theta $-graded right coideal subalgebras of $\NA
      (M)\#H$ containing $H$,
 where right coideal subalgebras are ordered with respect to inclusion and
 the morphisms are ordered by the Duflo order.}

\medskip

The map $\varkappa ^M$ also exists for non-finite $\Wg (M)$,
if we assume that $\NA (M)$ is decomposable, see Def.~\ref{de:decomp}.
Then $\varkappa ^M$
is always injective, order preserving and order reflecting by
Thm.~\ref{th:bijective}.

\medskip

\textbf{Acknowledgement.}
The first author would like to thank S.~Kolb for interesting discussions on
coideal subalgebras of $U_q(\gfrak)$.

\section{Weyl groupoids and the Duflo order}

Recall the definition of the Weyl groupoid of a root system from
\cite[Sect.~2]{a-CH09a}, see also \cite[Sect.~5]{p-HeckSchn08a}.

Let $I$ be a non-empty finite set and $(\al _i)_{i\in I}$
the standard basis of $\ndZ ^I$.
%
%By \cite[\S 1.1]{b-Kac90},
%a generalized Cartan matrix $\Cm =(\cm _{ij})_{i,j\in I}$
%is a matrix in $\ndZ ^{I\times I}$ such that
%  \begin{enumerate}
%    \item[(M1)] $\cm _{ii}=2$ and $\cm _{jk}\le 0$ for all $i,j,k\in I$ with
%      $j\not=k$,
%    \item[(M2)] if $i,j\in I$ and $\cm _{ij}=0$, then $\cm _{ji}=0$.
%  \end{enumerate}
%
Let $\cX $ be a non-empty set, and for all $i\in I$ and $X\in \cX $ let
$\rf _i : \cX \to \cX $ be a map and $\Cm ^X=(\cm ^X_{jk})_{j,k \in I}$
a generalized Cartan matrix.
The quadruple
\[\cC = \cC (I,\cX ,(\rf _i)_{i \in I},
(\Cm ^X)_{X \in \cX }),\]
is called a \textit{Cartan scheme} if
\begin{enumerate}
\item[(C1)] $\rf _i^2 = \id$ for all $i \in I$,
\item[(C2)] $\cm ^X_{ij} = \cm ^{\rf _i(X)}_{ij}$ for all $X \in \cX $
  and $i,j\in I$.
\end{enumerate}

Let $\cC = \cC (I,\cX ,(\rf _i)_{i \in I}, (\Cm ^X)_{X \in \cX })$ be a
Cartan scheme. For all $i \in I$, $X\in \cX $ define $s_i^X \in
\Aut(\ndZ ^I)$ by
\[s_i^X(\al _j) = \al _j - \cm _{ij}^X \al _i \qquad \text{ for all } j \in
I.\]

Recall that a groupoid is a category where all morphisms are isomorphisms.
The {\em Weyl groupoid of $\cC $}  is the groupoid
$\Wg (\cC )$ with $\Ob (\Wg (\cC ))=\cX $, where the morphisms are
generated by all $s_i^X$ (considered as morphism in $\Hom (X,r_i(X))$
with $i \in I$, $X \in \cX $.
%Formally, for $X,Y\in \cX $ the set $\Hom (X,Y)$ consists of the triples
%$(Y,s,X)$, such that
%\[
%s=s_{i_n}^{\rf _{i_{n-1}}\cdots \rf _{i_1}(X)}\cdots
%s_{i_2}^{\rf _{i_1}(X)}s_{i_1}^X
%\]
%and $\rf _{i_n}\cdots \rf _{i_2}\rf _{i_1}(X)=Y$
%for some $n\in \ndN _0$ and $i_1,\ldots ,i_n\in I$.
%The composition of morphisms is induced by the group structure of
%$\Aut (\ndZ ^I)$:
%\[ (Q,g,Y)\circ (Y,f,X) = (Q,gf, X)\]
%for all $(Q,g,Y),(Y,f,X)\in \Ar{\Wg (\cC )}$. If
% $w=(Y,f,X)\in \Ar{\Wg (\cC )}$ and $\al \in \ndZ ^\theta $,
% then we define $w(\al )=f(\al )$.
Then  $s_i^{r_i(X)}s_i^X=\id _X$ in $\Hom (X,X)$. We will write $s_i$ instead
of $s_i^X$ if $X$ is uniquely determined by the context.

For any groupoid $\cG $ and any $X\in \Ob (\cG )$ let
\begin{align*}
  \Homsto[\cG ]{X}=\mathop{\cup }_{Y\in \Ob (\cG )}\Hom (Y,X)\quad
  \text{(disjoint union).}
\end{align*}

Let $\cC $ be a Cartan scheme and let $X\in \cX $.
Following \cite[\S 5.1]{b-Kac90} we say that
\begin{align}
  \rersys X =\{w (\al _i)\,|\,i\in I, w\in \Homsto[\Wg (\cC )]X\}
  \label{eq:reroots}
\end{align}
is the set of \textit{real roots} (of $X$), where $w\in \Homsto[\Wg (\cC )]X$
is
interpreted as an element in $\Aut (\ndZ ^I)$.
A real root $\al \in \rersys X$ is called \textit{positive}, if $\al \in \ndN
_0^I$. The set of positive real roots is denoted by $\rersys X_+$.

\begin{remar}
  Weyl groupoids associated to Nichols algebras satisfy
  additional properties which do not follow from the axioms of Cartan schemes,
  see Thm.~\ref{th:finite}.
%  Let $\cC $ be a Cartan scheme and let $X\in \cX $.
%  Assume that $\rersys X$ is finite.
%  In the proof of \cite[Thm.\,6.1]{a-CH09a} it was observed that
%  $\al \in \rersys X$ does not imply that $\al \in \rersys X_+$ or
%  $-\al \in \rersys X_+$.
  For example,
  cf.~\cite[Pf.\,of\,Thm.\,6.1]{a-CH09a},
  let
  $\cX =\{X_1,X_2,X_3\}$, $I=\{1,2\}$, $r_1(X_i)=X_{\sigma (i)}$,
  $r_2(X_i)=X_{\tau (i)}$, where $\sigma =(1\,2)$, $\tau =(2\,3)$.
  Let
  \[ A^{X_1}=
  \begin{pmatrix}
    2 & -1 \\ -3 & 2
  \end{pmatrix}, \quad
  A^{X_2}=
  \begin{pmatrix}
    2 & -1 \\ -4 & 2
  \end{pmatrix}, \quad
  A^{X_3}=
  \begin{pmatrix}
    2 & -1 \\ -4 & 2
  \end{pmatrix}. \]
  Then $\cC (I,\cX ,(r_i)_{i\in I},(A^X)_{X\in \cX })$ is a Cartan scheme
  with finitely many real roots
  \begin{align*}
    \rersys {X_1}=&\{ \pm 1,\pm 2,\pm 12,\pm 12^2,\pm 12^3,\pm 1^22^3,\\
    &\quad \pm 1^32^4,\pm 1^32^5,\pm 1^42^5,\pm 1^42^7,\pm 1^52^7,\pm 1^52^8 \},\\
    \rersys {X_2}=&\{ \pm 1,\pm 2,\pm 12,\pm 12^2,\pm 12^3,\pm 1^22^3,\\
    &\quad \pm 12^4,\pm 12^5,\pm 1^22^5,\pm 1^22^7,\pm 1^32^7,\pm 1^32^8 \},\\
    \rersys {X_3}=&\{ \pm 12^{-1},\pm 1,\pm 2,\pm 12,\pm 1^22,\pm 12^2,\\
    &\quad \pm 12^3,\pm 1^22^3,\pm 12^4,\pm 1^32^4,\pm 1^22^5,\pm 1^32^5\},
  \end{align*}
  where $k\al _1+l\al _2$ is abbreviated by $1^k2^l$ for all $k,l\in \ndZ $.
  Observe that $\rersys {X_3}$ contains the real root $\al _1-\al _2$, and
  hence
  $\cC (I,\cX ,(r_i)_{i\in I},(A^X)_{X\in \cX })$
  does not admit a root system in the sense of the following
  definition.
\end{remar}

We say that
\[\Rwg = \Rwg (\cC , (\rsys X)_{X \in \cX })\]
is a {\em root system of type} $\cC $
if $\cC = \cC (I,\cX ,(\rf _i)_{i \in I}, (\Cm ^X)_{X \in \cX })$
is a Cartan scheme and $\rsys X \subset \ndZ ^I$, where $X \in \cX $,
are subsets such that
\begin{enumerate}
\item[(R1)] $\rsys X=(\rsys X\cap \ndN _0^I)\cup -(\rsys X\cap \ndN _0^I)$
  for all $X \in \cX $,
\item[(R2)] $\rsys X\cap \ndZ \al _i=\{\al _i,-\al _i\}$ for all $i\in I$,
  $X\in \cX $,
\item[(R3)] $s_i^X(\rsys X)= \rsys {r_i(X)}$ for all $i \in I$, $X\in \cX $,
\item[(R4)] $(r_i r_j)^{m_{i,j}^X}(X) = X$ for all $i,j \in I$ and
  $X\in \cX $ such that $i\not=j$ and
  $m_{i,j}^X:= \# (\rsys X\cap (\ndN _0\al _i + \ndN _0 \al_j))$ is finite.
\end{enumerate}
%We note that Axiom~(M2) is redundant for root systems by
%\cite[Lemma\,2.5]{a-CH09a}.

If $\Rwg (\cC , (\rsys X)_{X \in \cX })$ is a root system of type $\cC $,
then $\Wg (\Rwg ) := \Wg (\cC )$ is called the {\em
Weyl groupoid of} $\Rwg $.
The elements of $\rsys X_+:=\rsys X\cap \ndN _0^I$ and
$\rsys X_-:=-\rsys X_+$ are called \textit{positive} and
\textit{negative roots}, respectively.
Note that (R3) implies that $w(\rsys Y)=\rsys X$
for all $X,Y\in \cX $ and $w\in \Hom (Y,X)$.

Recall that a groupoid $\cG $ is \textit{connected},
if for all $X,Y\in \Ob (\cG )$
the set $\Hom (Y,X)$ is non-empty. It is \textit{finite},
if $\Ar{\cG }$ is finite.

The following claim was proven in \cite[Lemma\,2.11]{a-CH09a}.

\begin{lemma} \label{le:finrs}
  Let $\cC $ be a Cartan scheme and $\Rwg $ a root system of type $\cC $.
  Assume that $\Wg (\Rwg )$ is connected.
  Then the following are equivalent.
  \begin{enumerate}
    \item $\rsys X$ is finite for at least one
      $X\in \Ob (\Wg (\Rwg ))$.
%    \item $\rsys X$ is finite for at least one
%      $X\in \Ob (\Wg (\Rwg ))$.
    \item $\rersys X$ is finite for at least one
      $X\in \Ob (\Wg (\Rwg ))$.
%    \item $\Homsto [\Wg (\Rwg )]X$ is finite for all $X\in \Ob (\Wg (\Rwg ))$.
    \item $\Homsto [\Wg (\Rwg )]X$ is finite for at least one
      $X\in \Ob (\Wg (\Rwg ))$.
    \item The groupoid $\Wg (\Rwg )$ is finite.
  \end{enumerate}
  Further,
%  since $\Wg (\Rwg )$ is connected,
  (1)--(3) hold for one $X\in \Ob (\Wg (\Rwg))$
  if and only if they hold for all $X\in \Ob (\Wg (\Rwg))$.
\end{lemma}

\begin{remar}
  The equivalence of (1), (2), and (4) was stated and proven in
  \cite[Lemma\,2.11]{a-CH09a}. Clearly, (4) implies (3).
  For the proof of the implication (4)$\Rightarrow $(1) in
  \cite[Lemma\,2.11]{a-CH09a} one needs only to assume (3), and hence
  all claims of Lemma~\ref{le:finrs} are equivalent.
\end{remar}

Let $\cC $ be a Cartan scheme and $\Rwg $ a root system of type $\cC $.
%Assume that $\Wg (\Rwg )$ is connected.
Then $\Rwg $ is called \textit{finite}, see \cite[Def.\,2.20]{a-CH09a},
if $\rsys X$ is finite for all $X\in
\Ob (\Wg (\Rwg ))$. If $\Wg (\Rwg )$ is connected, then this is equivalent
to the conditions in Lemma~\ref{le:finrs}.

Let $\ell $ denote the
length function on Weyl groupoids of root systems: for each $X\in \cX $
and each $w\in \Homsto[\Wg (\Rwg )] X$ let
\[ \ell (w)=
  \min \{m\in \ndN _0\,|\,\text{there exist $i_1,\dots,i_m\in I$
  with $w=\id _X s_{i_1}\cdots s_{i_m}$}\}.
\]
\begin{propo} \cite[Prop.\,2.12]{a-CH09a} \label{pr:posroots}
  Let $\cC $ be a Cartan scheme and $\Rwg $ a root system of type $\cC $.
  Let $X\in \cX $, $m\in \ndN _0$, and $i_1,\dots,i_m\in I$ such that
  $\ell (w)=m$ for $w=s_{i_1}\cdots s_{i_m}\in \Homsto [\Wg (\Rwg )] X$.
  Then the roots
  \[ \beta _k=\id _X s _{i_1}\cdots s_{i_{k-1}}(\al _{i_k}),
  \qquad k\in \{1,2,\dots,m\}, \]
  are positive and pairwise distinct. If
  $\Rwg $ is finite
  and $w\in \Homsto [\Wg (\Rwg )]X$ is the unique longest element,
  then $\{\beta _k\,|\,1\le k\le \ell (w)\}=\rsys X_+$.
\end{propo}

Prop.~\ref{pr:posroots} implies in particular that the set of roots of a
finite root system is uniquely determined by its Cartan scheme and coincides
with the set of real roots.

%In general the Cartan matrices $A^X$ are not of finite type if the Weyl
%groupoid $\Wg (\Rwg )$ is finite \cite[Prop.\,5.1(2)]{a-CH09a}.
%But in the special case, when the Cartan scheme $\cC $ is {\em standard},
%i.\,e. $\Cm ^X = \Cm^M$ for all $X,M \in \Ob(\Wg (\cC ))$, the following
%holds (see \cite[Thm.\,3.3]{a-CH09a} for a slightly different proof).
%
%\begin{corol}\label{co:standard}
%Let $\cC $ be a standard Cartan scheme with generalized Cartan matrix $\Cm
%=\Cm ^X$ for all $X \in \Ob(\Wg (\cC ))$. Let $\Rwg $ be a root system of type
%$\cC $.
%Then the following are equivalent.
%\begin{enumerate}
%\item $\Wg (\Rwg )$ is finite.
%\item $\Cm $ is a Cartan matrix of finite type.
%\end{enumerate}
%\end{corol}

\begin{defin}\label{de:Cox}
  Let $\cX $ and $I$ be non-empty sets, $(r_i)_{i\in I}$ a family of
  maps $r_i:\cX \to \cX $, and $(m_{i,j}^X)_{i,j\in I,X\in \cX}$
  a family of numbers in $\ndN \cup \{\infty \}$
  such that $m_{i,i}^X=1$ and
  $(r_ir_j)^{m_{i,j}^X}(X)=X$ for all $X\in \cX $ and $i,j\in I$ with $m_{i,j}
  ^X<\infty $.
  \footnote{This slight extension of notation is compatible with (R4) and
  (C1).}
  Let $\cG $ be a groupoid with $\Ob (\cG )=\cX $, and let
  $(s_i^X)_{i\in I,X\in \cX }$ be a family of morphisms $s_i^X\in \Hom
  (X,r_i(X))$. We say that $(\cG ,(s_i^X)_{i\in I,X\in \cX })$
  \textit{satisfies the Coxeter relations} if
  \begin{align} \label{eq:Coxrel}
    \underbrace{s_i^{r_j(r_ir_j)^{m_{i,j}^X -1}(X)}
    s_j^{(r_ir_j)^{m_{i,j}^X -1}(X)}
    \dots
    s_j^{r_ir_j(X)}s_i^{r_j(X)}s_j^X} _{\text{$2m_{i,j}^X$ factors}}=\id _X
  \end{align}
  for all $X\in \cX $ and $i,j\in I$ with $m_{i,j}^X<\infty $.
  In particular, Eq.~\eqref{eq:Coxrel}
  means for $i=j$ that $s_i^{r_i(X)}s_i^X=\id _X$
  for all $X\in \cX $ and $i\in I$.

  Let $\Wg $ be a groupoid and $(s_i^X)_{i\in I,X\in \cX }$ a family of
  morphisms as above. We say that $(\Wg ,(s_i^X)_{i\in I,X\in \cX })$
  is a \textit{Coxeter groupoid}, if
  \begin{enumerate}
    \item $(\Wg ,(s_i^X)_{i\in I,X\in \cX })$ satisfies the Coxeter relations,
      and
    \item
      for each pair $(\cG ,(t_i^X)_{i\in
      I,X\in \cX })$ satisfying the Coxeter relations (with the same $\cX $,
      $I$, $(r_i)$ and $(m_{i,j}^X)$ as for $\Wg $) there is a unique functor
      $F:\Wg \to \cG $ such that $F$ is the identity on
      $\cX =\Ob (\Wg )=\Ob (\cG )$ and
      $F(s_i^X)=t_i^X$ for all $i\in I$, $X\in \cX $.
  \end{enumerate}
\end{defin}

The universal property of a Coxeter groupoid
$(\Wg ,(s_i^X)_{i\in I,X\in \cX })$
implies that $\Ar{\Wg }$ is generated by the morphisms $s_i^X\in \Hom
(X,r_i(X))$, where $i\in I$ and $X\in \cX $.

For the rest of this section
let $\cC =\cC (I,\cX ,(r_i)_{i \in I},
(\Cm ^X)_{X \in \cX })$ be a Cartan scheme and
let $\Rwg =\Rwg (\cC ,(\rsys X)_{X\in \cX })$ be a root system of type $\cC
$.

\begin{theor} \cite[Thm.~1]{a-HeckYam08} \label{th:Cox}
  For all $i,j\in I$ and $X\in \cX $ let
  \[ m_{i,j}^X= \# (\rsys X\cap (\ndN _0\al _i + \ndN _0 \al_j)). \]
  Then $(\Wg (\Rwg ), (s_i^X)_{i\in I,X\in \cX })$ is a Coxeter groupoid
  with respect to $(m_{i,j}^X)$.
\end{theor}

\begin{defin}\label{de:Lambda}
  For all $X\in \cX $,
  $m\in \ndN $ and $(i_1,\dots,i_m)\in I^m$ let
  $\Lambda _+^X()=\emptyset $ and
  \begin{align*}
    &\beta _k=s_{i_1}^{r_{i_1}(X)}s_{i_2}^{r_{i_2}r_{i_1}(X)}
    \cdots s_{i_{k-1}}^{r_{i_{k-1}}\cdots r_{i_2}r_{i_1}(X)}(\al _{i_k})\in
    \rsys X, \quad k\in \{1,2,\dots,m\},\\
    &\Lambda ^X(i_1,\dots,i_m)=(\beta _k)_{k\in \{1,2,\dots ,m\}},\\
    &\Lambda _+^X(i_1,\dots,i_m)=\big\{ \lambda \in \rsys X_+ \,|\,
    \#\{ k\in \ndZ \,|\,1\le k\le m,\lambda =\pm \beta _k\}
    \text{ is odd}\big\}.
  \end{align*}
\end{defin}

\begin{lemma}
  \label{le:Lambda1}
  Let $m\in \ndN $, $i_1,\dots,i_m\in I$, $X\in \cX $, and $Y=r_{i_1}(X)$.
%  If $\ell (1_Ms_{i_1}\cdots s_{i_m})=m$ then
%  \[ \Lambda _+^M(i_1,\dots,i_m)=
%  \big\{s_{i_1}^{r_{i_1}(M)}s_{i_2}^{r_{i_2}r_{i_1}(M)}
%  \cdots s_{i_{k-1}}^{r_{i_{k-1}}\cdots r_{i_2}r_{i_1}}(\al _{i_k})\,|\,
%  1\le k\le m\} \]
%  and $\# \Lambda ^M_+(i_1,\dots,i_m)=m$.
  Then
  \begin{align*}
    \Lambda _+^X(i_1,\dots,i_m)=
    \begin{cases}
      s_{i_1}^Y(\Lambda _+^{Y}(i_2,\dots,i_m))
      \cup \{\al _{i_1}\}
      & \text{if $\al _{i_1}\notin \Lambda _+^Y(i_2,\dots,i_m)$,}\\
      s_{i_1}^Y\big(\Lambda _+^Y(i_2,\dots,i_m)\setminus
      \{\al _{i_1}\}\big)
      & \text{if $\al _{i_1}\in \Lambda _+^Y(i_2,\dots,i_m)$.}
    \end{cases}
  \end{align*}
\end{lemma}

\begin{proof}
%  The entries of $\Lambda ^M(i_1,\dots,i_m)$
%  are pairwise distinct positive roots in
%  $\rsys M$ by \cite[Prop.\,2.12]{a-CH09a}.
  The claim follows from Def.~\ref{de:Lambda} and basic properties of
  the map $s_{i_1}^Y$.
\end{proof}

The sets $\Lambda _+^X(i_1,\dots,i_m)$ ultimately describe the
elements of $\Homsto[\Wg (\Rwg )]X$.
%To elaborate this, we will need the following known fact.

\begin{propo}
  \label{pr:Lambda2}
  Let $X\in \cX $,
  $m,m'\in \ndN _0$ and $i_1,\dots,i_m,i'_1,\dots,i'_{m'}\in I$. The
  following are equivalent.
  \begin{enumerate}
    \item $\Lambda _+^X(i_1,\dots,i_m)=\Lambda _+^X(i'_1,\dots,{i'}_{m'})$,
    \item $s_{i_1}\cdots s_{i_m}=s_{i'_1}\cdots s_{i'_{m'}}$
      in $\Homsto [\Wg (\Rwg )]X$.
  \end{enumerate}
  Moreover, $\#\Lambda _+^X(i_1,\dots,i_m)=
  \ell (\id _Xs_{i_1}\cdots s_{i_m})$.
\end{propo}

\begin{proof}
  Let $\cV $ denote the category with $\Ob (\cV )=\cX $ and morphisms
  \begin{align*}
    \Hom (Y,Z)=\big\{&\big(\Lambda ^Z_+(j_1,\dots,j_n),
    \id _Z s_{j_1}\cdots s_{j_n}\big)\,|\\
    & n\in \ndN _0,\, j_1,\dots,j_n\in I,
    r_{j_1}\cdots r_{j_n}(Y)=Z \big\}
  \end{align*}
  for all $Y,Z\in \cX $. Composition of morphisms
  is defined via concatenation:
  \begin{align*}
    \big(\Lambda ^Z_+(j_1,\dots,j_n), \id _Z s_{j_1}\cdots s_{j_n}\big)
    \circ
    \big(\Lambda ^Y_+(k_1,\dots,k_p), \id _Y s_{k_1}\cdots s_{k_p}\big)
    \qquad &\\
    =\big(\Lambda ^Z_+(j_1,\dots,j_n,k_1,\dots ,k_p),
    \id _Z s_{j_1}\cdots s_{j_n}s_{k_1}\cdots s_{k_p}\big)&
  \end{align*}
  for all $Y,Z\in \cX $,
  $n,p\in \ndN _0$ and $j_1,\dots,j_n,k_1,\dots,k_p\in I$
  with $Z=r_{j_1}\cdots r_{j_n}(Y)$. First we prove that $\cV $ is indeed a
  category.

  Let  $n,n',p,p'\in \ndN _0$, $j_1,\dots,j_n$, $j'_1,\dots,j'_{n'}$,
  $k_1,\dots, k_p$, $k'_1,\dots ,k'_{p'}\in I$
  and $Z\in \cX $ such that
  \[ \Lambda ^Z_+(j_1,\dots ,j_n)=\Lambda ^Z_+(j'_1,\dots ,j'_{n'}),\quad
  \Lambda ^Y_+(k_1,\dots ,k_p)=\Lambda ^Y_+(k'_1,\dots ,k'_{p'}),
  \]
  and $\id _Z s_{j_1}\cdots s_{j_n}=\id _Z s_{j'_1}\cdots s_{j'_{n'}}$,
  where $Y=r_{j_n}\cdots r_{j_1}(Z)$.
  The definition of $\Lambda ^Z_+$ implies that
  \begin{align}
    \Lambda ^Z_+(j_1,\dots,j_n,k_1,\dots ,k_p)=&
    \Lambda ^Z_+(j_1,\dots,j_n,k'_1,\dots ,k'_{p'}),\\
    \Lambda ^Z_+(j_1,\dots,j_n,k_1,\dots ,k_p)=&
    \Lambda ^Z_+(j'_1,\dots,j'_{n'},k_1,\dots ,k_p),
  \end{align}
  and hence the composition is independent of the choice of representatives.
  Clearly, the composition is associative, and $(\Lambda ^Z_+(),\id _Z)$ is the
  identity for any $Z\in \cX $, and hence $\cV $ is a category.
  Moreover, $\Ar{\cV }$ is generated by the morphisms $(\Lambda
  ^Z_+(i),s_i^{r_i(Z)})$ with $i\in I$ and $Z\in \cX $,
  which are invertible since $\Lambda ^Z_+(i,i)=\emptyset $.
  Thus $\cV $ is a groupoid.
%  and $r_{j_m}\cdots r_{j_1}(Z)=r_{j'_{m'}}\cdots r_{j'_1}(Z)$.

  Let $Z\in \cX $ and $i,j\in I$. Assume that
  $m_{i,j}^Z<\infty $. Then
  \begin{align}\label{eq:Lambda+Nii}
    \Lambda _+^Z(\underbrace{i,j,i,j,\dots ,i,j}_{2m_{i,j}^Z \text{ entries}})
    =& \emptyset ,
  \end{align}
  since the entries $\beta _k$ of $\Lambda ^Z(i,j,i,j,\dots ,i,j)$
  are the elements of
  $(\ndZ \al _i+\ndZ \al _j)\cap \rsys Z$, each appearing
  with multiplicity one, see \cite[Lemma\,6]{a-HeckYam08}.
  In the special case, where $i=j$, we have $\Lambda ^Z(i,i)=(\al _i,-\al _i)$.
  Hence the pair
  \[ (\cV , ((\Lambda ^Z_+(i),s_i^{r_i(Z)}))_{i\in I,Z\in \cX }) \]
  satisfies the Coxeter relations in the sense of Def.~\ref{de:Cox}.
  By Thm.~\ref{th:Cox}, there is a functor from $\Wg (\Rwg )$ to $\cV $
  which maps $s_i^{r_i(Z)}$ to $(\Lambda ^Z_+(i),s_i^{r_i(Z)})$
  for each $Z\in \cX $ and
  $i\in I$. This proves the implication
  (2)$\Rightarrow $(1) in the claim of the lemma.

  We prove (1)$\Rightarrow $(2). Assume that
  $\Lambda _+^X(i_1,\dots,i_m)=\Lambda _+^X(i'_1,\dots,{i'}_{m'})$.
  Using that
  \[ \Lambda ^Z_+(i,i,j_1,\dots,j_n)=\Lambda ^Z_+(j_1,\dots,j_n),
  \quad \id _Z s_is_is_{j_1}\cdots s_{j_n}=\id _Z s_{j_1}\cdots s_{j_n} \]
  for all $Z\in \cX $ and $n\in \ndN _0$, $i,j_1,\dots,j_n\in I$,
  we may assume that $m'=0$.
  By the first part of the lemma we may restrict to
  the case when $\id _Xs_{i_1}\cdots s_{i_m}$ is a reduced expression, that is,
  $\ell (\id _Xs_{i_1}\cdots s_{i_m})=m$.
  Then we have to show that $m=0$.
  The roots
  \[ s_{i_1}^{r_{i_1}(X)}\cdots s_{i_{k-1}}^{r_{i_{k-1}}\cdots r_{i_1}(X)}
  (\al _{i_k})\in \rsys X, \]
  where $1\le k\le m$,
  are pairwise distinct and positive by
  Prop.~\ref{pr:posroots}. Hence the assumption
  $\Lambda _+^X(i_1,\dots,i_m)=\emptyset $ implies the desired claim $m=0$.
  The last assertion of the lemma follows from the first one and from
  Prop.~\ref{pr:posroots}.
\end{proof}

Let $X,Y\in \cX $, $w\in \Hom (Y,X)$, $m\in \ndN _0$, and
$i_1,\dots,i_m\in I$ such that $w=s_{i_1}\cdots s_{i_m}$.
Let $\Lambda ^X_+(w)=\Lambda ^X_+(i_1,\dots,i_m)$. This notation is
justified by Prop.~\ref{pr:Lambda2}.

\begin{corol} \label{co:aliinL}
  Let $X,Y\in \cX $, $w\in \Hom (Y,X)$, and $i\in I$. Then
  \begin{align*}
    \ell (s_i^Xw)=
    \begin{cases}
      \ell (w)+1 & \text{if $\al _i\notin \Lambda ^X_+(w)$,}\\
      \ell (w)-1 & \text{if $\al _i\in \Lambda ^X_+(w)$.}
    \end{cases}
  \end{align*}
%  $\al _i\notin \Lambda ^M_+(w)$ if and only if
%  $\al _i\in \Lambda ^{r_i(M)}_+(s_i^Mw)$ if and only if
%  $\ell (s_i^Mw)=\ell (w)+1$.
\end{corol}

\begin{proof}
  By Prop.~\ref{pr:Lambda2}, $\ell (s_i^Xw)=\ell (w)+1$ if and only if
  $\#\Lambda ^{r_i(X)}_+(s_i^X w)=\# \Lambda ^X_+(w)+1$. Then the claim holds
  by Lemma~\ref{le:Lambda1}.
\end{proof}

\begin{defin}\label{de:Duflo}
  Let $X\in \cX $.
  For all $Y,Z\in \cX $ and $x\in \Hom (Y,X)$, $y\in \Hom (Z,Y)$
  we write $x\le _D xy$ (in $\Homsto [\Wg (\Rwg )]X)$ if and only
  if $\ell (xy)=\ell (x)+\ell(y)$. Then $\le _D$ is a partial order on
  $\Homsto [\Wg (\Rwg )]X$ which is called the \textit{(right) Duflo order}.
\end{defin}

The definition stems from the corresponding notion for Weyl groups of
semisimple Lie algebras, see \cite{a-Melni04}, \cite[A\,1.2]{b-Joseph}.
The Duflo order is also known as the \textit{weak order} \cite{b-BjoeBre05}.

\begin{remar} \label{re:Duflo}
  As for the right Duflo order for Weyl groups, $\le _D$ is the weakest
  partial order on $\Wg (\Rwg )$ such that
  $x\le _Dxs_i$ for all $x\in \Ar{\Wg (\Rwg )}$ and $i\in I$ with
  $\ell (x)<\ell (xs_i)$.
\end{remar}

The following theorem gives a characterization of the right Duflo order.

\begin{theor} \label{th:Dorder}
  Let $X\in \cX $ and let $w_1,w_2\in \Homsto [\Wg (\Rwg )]X$.
  Then $w_1\le _D w_2$ if and only if
  $\Lambda ^X_+(w_1)\subset \Lambda ^X_+(w_2)$,
\end{theor}

\begin{proof}
  We proceed by induction on $\ell (w_1)$. If $w_1=\id _X$, then
  $\Lambda ^X_+(w_1)=\emptyset $ and $w_1\le _Dw_2$,
  and hence the claim holds.
  If $\ell (w_1)=1$, then $w_1=s_i^{r_i(X)}$ for
  some $i\in I$, and hence $\Lambda ^X_+(w_1)=\al _i$.
  By definition, $w_1\le _Dw_2$ if and only if
  $\ell (w_2)=1+\ell (s_i^Xw_2)$. Hence
  the claim holds by Cor.~\ref{co:aliinL}.

  Assume now that $\ell (w_1)>1$. Let $i\in I$
  with $\ell (w_1)=\ell (w)+1$ for $w=s_i^Xw_1$.
  Then
  \begin{align}
    \al _i\in \Lambda ^X_+(w_1)
    \label{eq:aliinLw1a}
  \end{align}
  by Cor.~\ref{co:aliinL}. Thus Lemma~\ref{le:Lambda1} implies that
  $\Lambda ^X_+(w_1)\subset \Lambda ^X_+(w_2)$ if and only if
  \begin{align} \label{eq:istep2a}
    \al _i\in \Lambda ^X_+(w_2) \quad \text{and} \quad
    \Lambda ^{r_i(X)}_+(s_i^Xw_1)\subset \Lambda ^{r_i(X)}_+(s_i^Xw_2).
  \end{align}
  Induction hypothesis, \eqref{eq:aliinLw1a} and Cor.~\ref{co:aliinL}
  imply that the relations in
  \eqref{eq:istep2a} are equivalent to
  \begin{align} \label{eq:istep3a}
    \ell (s_i^Xw_2)=\ell (w_2)-1,\quad
    \ell (s_i^Xw_2)=\ell (s_i^Xw_1)+\ell (w_1^{-1}w_2).
  \end{align}
  Since \eqref{eq:istep3a}
  implies that $w_1\le _D w_2$, the if part of the claim holds.
  Further, $w_1\le _Dw_2$ implies that
  \[ \ell (s_i^Xw_2)\le \ell (s_i^Xw_1)+\ell (w_1^{-1}w_2)=\ell (w_1)-1
  +\ell (w_1^{-1}w_2) =\ell (w_2)-1, \]
  that is, that \eqref{eq:istep3a} holds.
  Therefore the only if part of the claim holds as well.
\end{proof}

\section{Braided Hopf algebras and Nichols algebras}
\label{sec:NA}

Let $\fie $ be a field and let $H$ be a Hopf algebra over $\fie $ with
bijective antipode. Let $\ydH $ denote the category of Yetter-Drinfeld modules
over $H$.

Let $R$ be a bialgebra in $\ydH $. We use the Sweedler notation for the
coaction $\coa: R\to H\ot R$ and the coproduct $\copr _R:R\to R\ot R$ in the
following form: $\coa (r)=r\_{-1}\ot r\_0$, $\copr _R(r)=r\^1\ot r\^2$ for all
$r\in R$.
%Let $R^\cop $ be the bialgebra in $\ydH $ with $R^\cop =R$ as
%algebras and comultiplication $\copr _{R^\cop }=c^{-1}\copr _R$.
%If no confusion is possible, we just write $\copr ^\cop _R$ or $\copr ^\cop $
%for this coproduct.
%We use the Sweedler notation for $\copr _R^\cop $ in
%the following form: $\copr _R^\cop (r)=r^{[1]}\ot r^{[2]}$ for $r\in R^\cop $.

Let $R$ be a Hopf algebra in $\ydH $ with antipode $S_R$.
Let $R\#H$ be the bosonization of $R$,
see e.\,g.~\cite[Sect.~1.4]{p-AHS08}.
Recall that $R\#H$ is a Hopf algebra with projection $\pi _H:R\#H\to H$, and
\begin{gather}
  \label{eq:bosalg}
  (r\#h)(r'\#h')=r(h\_1\lact r')\#h\_2h',\\
  r\_{-1}\ot r\_0=\pi _H(r\_1)\ot r\_2,\quad
  r\_1\ot r\_2=r\^1 r\^2{}\_{-1}\ot r\^2{}\_0
  \label{eq:coal}
\end{gather}
for all $r,r'\in R$, $h,h'\in H$,
where $\copr (a)=a\_1\ot a\_2$ for all $a\in R\#H$.

Let $S$ denote the antipode of the Hopf algebra $R\#H$. Then
\begin{align}
  S_R(r)=r\_{-1}S(r\_0),\quad S(r)=S(r\_{-1})S_R(r\_0)\quad \text{for all
  $r\in R$,}
  \label{eq:antipode}
\end{align}
and $S_R\in \End (R)$ is a morphism in $\ydH $ satisfying
\begin{align}
  \label{eq:SRprop1}
  S_R(rs)=&S_R(r\_{-1}\lact s)S_R(r\_0),\\
  \copr _R(S_R(r))=&S_R(r\^1{}\_{-1}\lact r\^2)\ot S_R(r\^1{}\_0)
  \label{eq:SRprop2}
\end{align}
for all $r,s\in R$.
If $S$ is bijective, then the map $S_R^{-1}:R\to R$,
\begin{align}
  \label{eq:SRi}
  S_R^{-1}(r)=S^{-1}&(r\_0)r\_{-1} \quad \text{for all $r\in R$}
\end{align}
is a morphism in $\ydH $ and is inverse to $S_R$. Moreover,
\begin{align}
  \label{eq:SRi1}
  S^{-1}(r)=S_R^{-1}(r\_0)S^{-1}(r\_{-1}) \quad \text{for all $r\in R$.}
\end{align}
In this case, Eqs.~\eqref{eq:SRprop1}, \eqref{eq:SRprop2} are equivalent to
\begin{align}
  \label{eq:SRiprop1}
  S_R^{-1}(rs)=&S_R^{-1}(s\_0)S_R^{-1}(S^{-1}(s\_{-1})\lact r),\\
  \copr _R(S_R^{-1}(r))=&S_R^{-1}(r\^2{}\_0)\ot S_R^{-1}(S^{-1}(r\^2{}\_{-1})
  \lact r\^1)
  \label{eq:SRiprop2}
\end{align}
for all $r,s\in R$.

%Since $S_R$ is a morphism in $\ydH $, it follows from Eq.~\eqref{eq:antipode}
%that
%\[ S^2(r)=S(S_R(r\_0))S^2(r\_{-1})
%=S(r\_{-1})S_R^2(r\_0)S^2(r\_{-2}) \] for all $r\in R$, and hence
%\begin{align}
%  S^2(r)=S(r\_{-1})\lact S_R^2(r\_0)\quad \text{for all $r\in R$.}
%  \label{eq:antip2}
%\end{align}
%Recall that $S^2$ is a Hopf algebra map in $R\#H$.
%This and
%Eq.~\eqref{eq:antip2} imply that
%\begin{align}
%  S^2(h\lact r)=&S^2(h)\lact S^2(r),
%  \label{eq:hactS2}\\
%  \coa(S^2(r))=&S^2(r\_{-1})\ot S^2(r\_0),
%  \label{eq:coaS2}\\
%  S^2S_R(r)=&S_RS^2(r)
%  \label{eq:S2SR}
%\end{align}
%for all $h\in H$ and $r\in R$.

\begin{remar}\label{re:bijant}
  (i) Let $A$ be a Hopf algebra. Then $A^\op $ is a Hopf algebra if and
  only if the antipode $S$ of $A$ is bijective. In this case $S^{-1}$ is the
  antipode of $A^\op $.

  (ii) Let $B$ be a bialgebra in $\ydH $ with a coalgebra filtration
  \begin{align*}
    \fie 1=B_0\subset B_1\subset B_2\subset \cdots \subset B,
    \quad \cup _{n=0}^\infty B_n=B.
  \end{align*}
  Then $B$ is a Hopf algebra in $\ydH $, and the antipodes of $B$ and $B\#H$
  are bijective. Indeed,
  \begin{align*}
    H\subset B_1\#H \subset B_2\#H\subset \cdots \subset B\#H
  \end{align*}
  is a coalgebra filtration of $B\#H$ and
  \begin{align*}
    H^\op \subset (B_1\#H)^\op \subset (B_2\#H)^\op \subset \cdots
    \subset (B\#H)^\op
  \end{align*}
  is a coalgebra filtration of $(B\#H)^\op $. Since $H$ and $H^\op $ are Hopf
  algebras, $B\#H$ and $(B\#H)^\op $ are Hopf algebras by
  \cite[Lemma~5.2.10]{b-Montg93}. By Part~(i) the antipode of $B\#H$ is
  bijective. Hence $S_B$ is bijective with inverse given in
  Eq.~\eqref{eq:SRi} (for $R=B$).
  \qed
\end{remar}

Let $\rcs\subset R$ be a subspace. We say that $E$ is a \textit{right coideal
subalgebra of $R$ in} $\ydH $ if $\rcs \subset R$ is a subobject in $\ydH $
and a subalgebra (containing $1$)
with $\copr _R(\rcs )\subset \rcs \ot R$. If $H$ is the trivial $1$-dimensional
Hopf algebra, we follow the traditional terminology and call
$\rcs $ a right coideal subalgebra of $R$.

Let $G$ be a group.
We say that a right coideal subalgebra $E$ of a (braided) Hopf algebra $R$
is $G$-\textit{graded}, if $R=\oplus _{g\in G}R_g$ is a $G$-graded algebra
and $E=\oplus _{g\in G}(E\cap R_g)$.
For any $G$-graded algebras $A,B$, and any algebra map $f:A\to B$ we say
that $f$ is a \textit{homomorphism of $G$-graded algebras}, if
$f(A_g)\subset B_g$ for all $g\in G$.

\begin{lemma}\label{le:EcapR}
  Let $B\subset R$ be a Hopf subalgebra in $\ydH $. Assume that there exists
  a morphism $\pi :R\to B$ of Hopf algebras in $\ydH $
  such that $\pi |_B=\id _B$.
  Let $R^{\co B}=\{r\in R\,|\,r\^1\ot \pi (r\^2)=r\ot 1\}$.
  Let $\rcs \subset R$ be a right coideal subalgebra in $\ydH $
  such that $B\subset \rcs $.  Then the multiplication map
  $(R^{\co B}\cap \rcs )\ot B\to \rcs $ is an isomorphism.
\end{lemma}

\begin{proof}
  The inverse of the multiplication map is given by
  \[ E\to (R^{\co B}\cap \rcs )\ot B, \quad
  r\mapsto r\^1 S_R(\pi (r\^2))\ot \pi (r\^3) \]
  for all $r\in E$.
\end{proof}

\begin{propo} \label{pr:rcs}
  Let $R$ be a Hopf algebra in $\ydH $.

  (i) Let $E\subset R$.
  If $E$ is a right coideal subalgebra of $R$ in $\ydH $, then
  $E\#H$ is a right coideal subalgebra of $R\#H$.

  (ii) Let $E'$ be a right coideal subalgebra of $R\#H$ with $H\subset E'$.
  Then
  $E={E'}^{\,\co H}$ is a right coideal subalgebra of $R$ in $\ydH $
  with $E'=E\#H$.

  (iii) Let $G$ be a group, and assume that the algebra
  $R=\oplus _{g\in G}R_g$ is $G$-graded and $R_g\in \ydH $ for all $g\in G$.
  Then $R\#H=\oplus _{g\in G}(R\#H)_g$ is $G$-graded
  with $(R\#H)_g=R_g\#H$. Let $E\subset R$ be a subobject in $\ydH $.
  Then
  $E$ is a $G$-graded subalgebra of $R$ if and only if $E\#H$ is a
  $G$-graded subalgebra of $R\#H$.
\end{propo}

\begin{proof}
  (i) and (iii) follow from Eqs.~\eqref{eq:bosalg}, \eqref{eq:coal}.
  (ii) is a special case of Lemma~\ref{le:EcapR} with $H=\fie 1$ and $B=H$.
\end{proof}

Let $V\in \ydH $. Assume that $\dim _\fie V<\infty $.
Then $V^*\in \ydH $ with the following properties:
\begin{align}
  \langle h\lact f,v\rangle =&\langle f, S(h)\lact v\rangle,
  \label{eq:dual1}\\
  f\_{-1}\langle f\_0,v\rangle =&S^{-1}(v\_{-1})\langle f,v\_0\rangle,
  \label{eq:dual2}
\end{align}
for all $h\in H$, $v\in V$, $f\in V^*$, see e.g. \cite[Sect.~1.2]{p-AHS08}.
Let $\NA (V)$ and $\NA (V^*)$ denote the Nichols algebra of $V$ and $V^*$,
respectively. These are $\ndN _0$-graded braided Hopf algebras in $\ydH $ with
degree $1$ parts $\NA ^1(V)\simeq V$, $\NA ^1(V^*)\simeq V^*$ and with $\fie $
as degree $0$ part.
Since the antipode of $H$ is bijective, the antipodes of
$\NA (V)$, $\NA (V^*)$, $\NA (V)\#H$ and $\NA (V^*)\#H$ are bijective by
Remark~\ref{re:bijant}(ii).

The evaluation map between $V^*$ and $V$ induces a bilinear form
\begin{align}
  \langle\cdot,\cdot\rangle :\NA (V^*)\times \NA (V)\to \fie ,
  \label{eq:pairing}
\end{align}
see \cite[Sect.~1.5]{p-AHS08} for the origins.
This pairing is non-degenerate, and it satisfies the equations
\begin{align}
  \label{eq:pairing1}
  \langle 1,1\rangle =1, \quad
  \langle f,v\rangle =&0 \quad \text{for all $f\in \NA ^k(V^*)$, $v\in \NA
  ^l(V)$, $k\not=l$, and}
\end{align}
\begin{align}
  \label{eq:pairing2}
  \langle h\lact f,v\rangle =&\langle f,S(h)\lact v\rangle ,\\
  \label{eq:pairing3}
  f\_{-1}\langle f\_0,v\rangle =&S^{-1}(v\_{-1})\langle f,v\_0\rangle ,\\
  \label{eq:pairing4}
  \langle f,vw\rangle =&\langle f\^1,w\rangle\langle f\^2,v\rangle ,\\
  \label{eq:pairing5}
  \langle fg,v\rangle =&\langle g,v\^2\rangle\langle f,v\^1\rangle
\end{align}
for all $f,g\in \NA (V^*)$, $v,w\in \NA (V)$, $h\in H$.

Let $\{b^\al \}$ and $\{b_\al \}$ be $\ndN _0$-graded
dual bases of $\NA (V^*)$ and $\NA (V)$,
respectively. Eqs.~\eqref{eq:pairing2}-\eqref{eq:pairing5} imply the
following for all $h\in H$, $v\in \NA (V)$, $f\in \NA (V^*)$.
{\allowdisplaybreaks
\begin{align}
  \label{eq:bb1}
  \cou (h)\sum _\al b_\al \ot b^\al
  =&\sum _\al h\_1\lact b_\al \ot h\_2\lact b^\al ,\\
  \label{eq:bb1a}
  \sum _\al S(h)\lact b_\al \ot b^\al =&\sum _\al b_\al \ot h\lact b^\al ,\\
  \label{eq:bb2}
  \sum _\al 1\ot b_\al \ot b^\al =&
  \sum _\al b_\al {}\_{-1}b^\al {}\_{-1}\ot b_\al {}\_0\ot b^\al {}\_0,\\
  \label{eq:bb2a}
  \sum _\al b_\al {}\_{-1}\ot b_\al {}\_0 \ot b^\al =&
  \sum _\al S(b^\al {}\_{-1})\ot b_\al \ot b^\al {}\_0,\\
  \label{eq:bb3}
  \sum _\al vb_\al \ot b^\al =&
  \sum _\al b_\al \ot \langle b^\al {}\^2,v\rangle b^\al {}\^1,\\
  \label{eq:bb4}
  \sum _\al b_\al v\ot b^\al =&
  \sum _\al b_\al \ot \langle b^\al {}\^1,v\rangle b^\al {}\^2,\\
  \label{eq:bb5}
  \sum _\al b_\al \ot b^\al f=&
  \sum _\al \langle f,b_\al {}\^1\rangle b_\al {}\^2 \ot b^\al ,\\
  \label{eq:bb6}
  \sum _\al b_\al \ot fb^\al =&
  \sum _\al \langle f,b_\al {}\^2\rangle b_\al {}\^1 \ot b^\al ,\\
  \label{eq:bb7}
  \sum _{\al ,\beta } b_\al \ot b_\beta \ot b^\beta b^\al =&
  \sum _\gamma \copr _{\NA (V)}b_\gamma \ot b^\gamma .
\end{align}
}

Let $K\in {}_{\NA (V)\#H}^{\NA (V)\#H}\yd $. We write
\begin{align} \label{eq:coaBVH}
  \coa _{\NA (V)\#H}(x)=x_{[-1]}\ot x_{[0]}
\end{align}
for the left coaction of $\NA (V)\#H$ on $x\in K$.
Then $K\in \ydH $, where the action is the restriction of the action
of $\NA (V)\#H$ to $H$, and the coaction is
\[ \coa =(\pi _H\ot \id )\coa _{\NA (V)\#H}, \]
and $\pi _H:\NA (V)\#H\to H$ is the canonical projection.
Let $\coa _{\NA (V)}:K\to \NA (V)\ot K$,
\begin{align}
  \coa _{\NA (V)}(x)=x_{[-1]}S(x_{[0]}{}\_{-1})\ot x_{[0]}{}\_0 \quad
  \text{for all $x\in K$.}
  \label{eq:BVcoaK}
\end{align}
We use modified Sweedler notation for $\coa _{\NA (V)}$ in the form
\begin{align} \label{eq:coaBV}
  \coa _{\NA (V)}(x)=x\^{-1}\ot x\^0 \quad \text{for all $x\in K$.}
\end{align}
The map $\coa _{\NA (V)}$ is $H$-linear and $H$-colinear
via diagonal action and coaction.
We are going to study the right action of $\NA (V)\#H$ on $K$ defined by
\begin{align}
  \adR{v}{x}=S^{-1}(v)\lact x \quad
  \text{for all $v\in \NA (V)\#H$, $x\in K$.}
  \label{eq:adR}
\end{align}

\begin{lemma}
  \label{le:adRprop}
  Let $K\in {}_{\NA (V)\#H}^{\NA (V)\#H}\yd $.

  (i)
  Let $x\in K$, $v,w\in \NA (V)$ and $h\in H$. Then
  \begin{align}
    h\lact (\adR{v}{x})=&\adR{(h\_2\lact v)}{(h\_1\lact x)},
    \label{eq:adR1}
    \\
    \coa (\adR{v}{x})=&x\_{-1}v\_{-1}\ot \adR{v\_0}{x\_0},
    \label{eq:adR2}
    \\
    \adR{w}{(\adR{v}{x})}=&\adR{vw}{x}.
    \label{eq:adR3}
  \end{align}

  (ii) For all $f\in V^*$ and $x\in K$ let
  $\partial ^L_f(x)=\langle f,x\^{-1}\rangle x\^{0}$.
%  where $\coa _{\NA (V)}(x)=x\^{-1}\ot x\^{0}$.
  Then $\partial ^L_f(x)\in K$ and
  \begin{align}
    \partial ^L_f(\adR{v}{x}) =\adR{v}{\partial ^L_f(x)}
    +\langle S^{-1}(x\_{-1})\lact f,v\^1\rangle \adR{v\^2}{x\_0}
    -\langle f,v\_2\rangle \adR{v\_1}{x}
    \label{eq:derfxb}
  \end{align}
  for all $f\in V^*$, $x\in K$, $v\in \NA (V)$.

  (iii) Suppose that $K$ is a $\NA (V)\#H$-module algebra. Then
  \begin{align}
    \adR{v}{(xy)}=(\adR{v\^2{}\_0}{x})(S^{-1}(v\^2{}\_{-1})\lact
    (\adR{v\^1}{y}))
    \label{eq:adR4}
  \end{align}
  for all $v\in \NA (V)$ and $x,y\in K$.
\end{lemma}

\begin{proof}
  (i) By the relations of $\NA (V)\#H$,
  \begin{align*}
    h\lact (\adR{v}{x})=&hS^{-1}(v)\lact x=S^{-1}(vS(h))\lact x
    =S^{-1}(S(h\_1)(h\_2\lact v))\lact x\\
    =&S^{-1}(h\_2\lact v)\lact (h\_1 \lact x)
    =\adR{(h\_2\lact v)}{(h\_1\lact x)}.
  \end{align*}
  Using the Yetter-Drinfeld structure of $K$ we obtain that
  \begin{align*}
    \coa (\adR{v}{x})=&\coa (S^{-1}(v)\lact x)
    =\pi _H
    (S^{-1}(v)\_1x_{[-1]}S(S^{-1}(v)\_3))\ot S^{-1}(v)\_2\lact x_{[0]}\\
    =&\pi _H(S^{-1}(v\_3)x_{[-1]}v\_1)\ot S^{-1}(v\_2)\lact x_{[0]}\\
    =&\pi _H(x_{[-1]}v\_1)\ot S^{-1}(v\_2)\lact x_{[0]}
    =x\_{-1}v\_{-1}\ot \adR{v\_0}{x\_0},
  \end{align*}
  where the fourth relation follows from $v\in \NA (V)$, and the last one from
  the definitions of $\coa $ and $\adR {}{}$. Eq.~\eqref{eq:adR3} holds since
  $S^{-1}$ is an algebra antiautomorphism of $\NA (V)\#H$.

  (ii) Let $f\in V^*$, $v\in \NA (V)$, and $x\in K$.
  Let $\vartheta :\NA (V^*)\#H \to \NA (V^*)$
  be the linear map with $\vartheta (w\#h)=w\cou (h)$ for all
  $w\in \NA (V^*)$, $h\in H$. It is well-known that $\vartheta (w)=w\_1 \pi
  _H(S(w\_2))$ for all $w\in \NA (V^*)\#H$. We get
  \begin{align*}
    &\partial ^L_f(\adR{v}{x})=\langle f, (\adR{v}{x})\^{-1}\rangle
    (\adR{v}{x})\^0 \eqn{\eqref{eq:BVcoaK}}
    \langle f,\vartheta (S^{-1}(v\_3)x_{[-1]}v\_1)\rangle
    \adR{v\_2}{x_{[0]}}\\
    &\quad =\langle f,\vartheta (S^{-1}(v\_3)\pi _H(x_{[-1]}v\_1))\rangle
    \adR{v\_2}{x_{[0]}}\\
    &\qquad
    +\langle f,\vartheta (\pi _H(S^{-1}(v\_3))x_{[-1]}\pi _H(v\_1))\rangle
    \adR{v\_2}{x_{[0]}}\\
    &\qquad
    +\langle f,\vartheta (\pi _H(S^{-1}(v\_3)x_{[-1]})v\_1)\rangle
    \adR{v\_2}{x_{[0]}}\\
    &\quad =\langle f,\vartheta (S^{-1}(v\_2))\rangle \adR{v\_1}{x}
    +\langle f,\vartheta (x_{[-1]})\rangle
    \adR{v}{x_{[0]}}+\langle f,\vartheta (x\_{-1}v\_1)\rangle
    \adR{v\_2}{x\_{0}}\\
    &\quad =\langle f,S^{-1}(v\_3)\pi _H(v\_2)\rangle \adR{v\_1}{x}
    +\adR{v}{\partial ^L_f(x)}+\langle f,x\_{-1}\lact \vartheta (v\_1)\rangle
    \adR{v\_2}{x\_{0}}\\
    &\quad \eqns{\eqref{eq:SRi}}
    \langle f,S_{\NA (V)}^{-1}(v\_2)\rangle \adR{v\_1}{x}
    +\adR{v}{\partial ^L_f(x)}+\langle f,x\_{-1}\lact v\^1\rangle
    \adR{v\^2}{x\_{0}}\\
    &\quad =-\langle f,v\_2\rangle \adR{v\_1}{x}
    +\adR{v}{\partial ^L_f(x)}+\langle S^{-1}(x\_{-1})\lact f,v\^1\rangle
    \adR{v\^2}{x\_{0}}
  \end{align*}
  which proves (ii). In the latter transformations,
  the third equation follows from the fact that $f$ is of degree $1$ in
  $\NA (V^*)$, and hence the pairing of a homogeneous product with $f$
  vanishes if the sum
  of the degrees of the factors is not $1$. Further, the fourth equation
  is obtained from $(\id \ot \pi _H)\copr (v)=v\ot 1$ and from the definition
  of $\vartheta $. In the fifth equation, the definition of $\vartheta $
  and of $\partial ^L_f(x)$ is used. The last equation comes from
  Eq.~\eqref{eq:pairing2} and the facts that $f$ has degree $1$ and each
  element of $V$ is primitive.

  (iii) Let $v\in \NA (V)$ and $x,y\in K$. Then
  \begin{align*}
    \adR{v}{(xy)}=&S^{-1}(v)\lact (xy)
    =(S^{-1}(v\_2)\lact x)(S^{-1}(v\_1)\lact y)\\
    =&(S^{-1}(v\^2{}\_0)\lact x)(S^{-1}(v\^1v\^2{}\_{-1})\lact y)\\
    =&(\adR{v\^2{}\_0}{x})(S^{-1}(v\^2{}\_{-1})\lact (\adR{v\^1}{y}))
  \end{align*}
  since $K$ is a $\NA (V)\#H$-module algebra.
\end{proof}

We extend the definition of $\partial ^L_f$ given in Lemma~\ref{le:adRprop}.
Note that $\NA (V)$ is a coalgebra.
For any left $\NA (V)$-comodule $L$ and any $f\in V^*$ let $\partial
^L_f:L\to L$ be the map defined by
\begin{align}
  \label{eq:partialL}
  \partial ^L_f(x)=\langle f,x\^{-1}\rangle x\^0 \quad \text{for all $x\in
  L$,}
\end{align}
where $\coa _{\NA (V)}(x)=x\^{-1}\ot x\^0$.

\begin{lemma}
  Let $L$ be a left $\NA (V)$-comodule.
  Let $L_0\subset L$ and let $\langle L_0\rangle $ be the smallest left
  $\NA (V)$-subcomodule of $L$ containing
  $L_0$. Then $\langle L_0\rangle $ is the smallest subspace of $L$ which
  contains $L_0$ and is stable under the maps $\partial ^L_f$, $f\in V^*$.
  \label{le:NAcomod}
\end{lemma}

\begin{proof}
  By assumption,
  \[ \langle L_0\rangle =\ls \{ g(x\^{-1})x\^0 \,|\,x\in L_0,g\in \NA
  (V)^*\}. \]
  The non-degeneracy of the pairing $\langle \cdot, \cdot \rangle $ between
  $\NA (V^*)$ and $\NA (V)$, see \eqref{eq:pairing},
  implies that
  \[ \langle L_0\rangle =\ls \{ \langle g,x\^{-1}\rangle
  x\^0 \,|\,x\in L_0,g\in \NA (V^*)\}. \]
  Eq.~\eqref{eq:pairing5} further implies that
  \begin{align*}
    \langle L_0\rangle =&\ls \{ \langle f_1\cdots f_k,x\^{-1}\rangle
    x\^0 \,|\,x\in L_0,k\in \ndN _0, f_1,\dots ,f_k\in V^*\}\\
    =&\ls \{ \partial ^L_{f_1}\cdots \partial ^L_{f_k}(x)\,|\,
    x\in L_0,k\in \ndN _0, f_1,\dots ,f_k\in V^*\}.
  \end{align*}
  This proves the lemma.
\end{proof}

\section{The algebra map $\Rmat $}
\label{sec:Rmat}

Let $V,V'\in \ydH $ and $W=V\oplus V'$ with $\dim _\fie W<\infty $.
Let $\pi :\NA (W)\#H \to \NA (V)\#H$ denote the Hopf algebra
projection corresponding to the decomposition $W=V\oplus V'$, and let
\begin{align} \label{eq:K}
  K=&\NA (W)^{\co \NA (V)\# H}=\{x\in \NA (W)\,|\,(\id \ot \pi )\copr
  _{\NA (W)\#H}(x)=x\ot 1\}
% & K'=&\NA (M')^{\co \NA (M'_i)\#H}.
\end{align}
be the algebra of right coinvariants with respect to the right coaction
of $\NA (V)\#H$.
Then $K$ is a braided Hopf
algebra in ${}_{\NA (V)\#H}^{\NA (V)\#H}\yd $, and hence the results in the
previous section apply.

Recall from \cite[Def.~2.5, Remark 2.7]{p-AHS08}
that $K\#\NA (V^*)$ is an algebra
in $\ydH $ such that the multiplication map $K\ot \NA (V^*)\to K\# \NA (V^*)$
is an isomorphism of left $K$-modules and right $\NA (V^*)$-modules,
and
\begin{align}
  (1\#f)(x\#1)=\langle f\^2,\pi (x\^1)\rangle x\^2{}\_0 \#
  S^{-1}(x\^2{}\_{-1})\lact f\^1
  \label{eq:KB(V*)}
\end{align}
for all $f\in \NA (V^*)$, $x\in K$. For brevity we will often write $xf$
instead of
$x\#f$ for all $x\in K$ and $f\in \NA (V^*)$.
Note that if $f\in V^*$ and $x\in K$ then
\begin{equation}
  \label{eq:fx}
  \begin{aligned}
    (1\#f)(x\#1)=&\langle f,x\^1\rangle x\^2\#1 + x\_0\#S^{-1}(x\_{-1})\lact f
    \\
    =&\partial ^L_f(x)\#1 + x\_0\#S^{-1}(x\_{-1})\lact f
  \end{aligned}
\end{equation}
by Eqs.~\eqref{eq:KB(V*)}, \eqref{eq:partialL}.

Let $\ad $ denote the left adjoint action of $\NA (W)\#H$ on itself.
Assume that $\dim (\ad \,\NA (V)\#H)(x)<\infty $ for all $x\in V'$.
Since $K$ is generated as an algebra by $(\ad \,\NA (V)\#H )(V')$,
see \cite[Prop.\,3.6]{p-AHS08},
the left adjoint action of $\NA (V)\#H$ on $K$ is locally finite.

Let $\Rmat: K\ot K\to K\ot (K\#\NA (V^*))$ be the linear map with
\begin{align}
  \Rmat (x\ot y)=\sum _{\al }\adR{b_\al }{x}\ot b^\al y
  \label{eq:Rmat}
\end{align}
for all $x,y\in K$, where $\{b^\al \}$ and $\{b_\al \}$ are dual bases of the
$\ndN _0$-graded algebras $\NA (V^*)$ and $\NA (V)$,
respectively,
and $\adR{}{}$ is defined in Eq.~\eqref{eq:adR}.
By the local finiteness of the left adjoint action of $\NA (V)\#H$ on $K$, the
sum in Eq.~\eqref{eq:Rmat} is finite for all $x,y\in K$.
Since $\adR{v}{1}=\cou (v)1$
for all $v\in \NA (V)\#H$, we conclude that
\begin{align}
  \Rmat (1\ot y)=1\ot y \quad \text{for all $y\in K$.}
  \label{eq:Rmat1}
\end{align}

Let $\Rmatb: K\ot K\#\NA (V^*)\to K\ot K\#\NA (V^*)$ be the linear map with
\begin{align}
  \Rmatb (x\ot y)=\sum _{\al }\adR{S_{\NA (W)}(b_\al )}{x}\ot b^\al y
  \label{eq:Rmatb}
\end{align}
for all $x\in K$ and $y\in K\#\NA (V^*)$. Then
\begin{align} \label{eq:RmatbRmat}
  \Rmatb \Rmat (x\ot y)=x\ot y \quad \text{for all $x,y\in K$.}
\end{align}
Indeed,
\begin{align*}
  \Rmatb \, \Rmat (x\ot y)=&\Rmatb \Big(\sum _{\al }\adR{b_\al }{x}\ot b^\al
  y\Big)
  =\sum _{\al ,\beta }\adR{b_\al S_{\NA (W)}(b_\beta )}{x}\ot b^\beta b^\al
  y\\
  =&\sum _{\gamma }\adR{b_\gamma \^1 S_{\NA (W)}(b_\gamma \^2)}{x}
  \ot b^\gamma y =\sum _\gamma x\ot \cou (b_\gamma )b^\gamma y=x\ot y,
\end{align*}
where the third equation follows from Eq.~\eqref{eq:bb7}.

\begin{lemma}
  \label{le:RisHlin}
  For all $x,y\in K$ and $h\in H$,
  \begin{align}
    h\lact \Rmat (x\ot y)=&\Rmat (h\_1\lact x \ot h\_2\lact y),
    \label{eq:Rlin}\\
    \coa \Rmat (x\ot y)=&x\_{-1}y\_{-1}\ot \Rmat (x\_0\ot y\_0).
    \label{eq:Rcolin}
  \end{align}
\end{lemma}

\begin{proof}
  By Eqs.~\eqref{eq:Rmat}, \eqref{eq:bb1} and \eqref{eq:adR1},
  \begin{align*}
    h\lact \Rmat (x\ot & y)
    =\sum _{\al }h\_1\lact (\adR{b_\al }{x})
    \ot (h\_2\lact b^\al )(h\_3\lact y)\\
    =&\sum _{\al }\adR{(h\_2 \lact b_\al )}{(h\_1\lact x)}
    \ot (h\_3\lact b^\al )(h\_4\lact y)\\
    =&\sum _{\al }\adR{b_\al }{(h\_1\lact x)}
    \ot b^\al (h\_2\lact y)
    =\Rmat (h\_1\lact x \ot h\_2\lact y)
  \end{align*}
  for all $h\in H$ and $x,y\in K$.
  Similarly, Eqs.~\eqref{eq:Rmat}, \eqref{eq:adR2} and \eqref{eq:bb2}
  imply that
  \begin{align*}
    \coa \Rmat (x\ot & y)
    =\coa \Big( \sum _\al \adR{b_\al }{x}\ot b^\al y\Big)\\
    =&\sum _\al x\_{-1}b_\al {}\_{-1}b^\al{}\_{-1}y\_{-1} \ot
    \adR{b_\al {}\_0}{x\_0}\ot b^\al {}\_0y\_0\\
    =&\sum _\al x\_{-1}y\_{-1} \ot
    \adR{b_\al {}}{x\_0}\ot b^\al y\_0
    =x\_{-1}y\_{-1}\ot \Rmat (x\_0 \ot y\_0)
  \end{align*}
for all $x,y\in K$.
\end{proof}

The vector space $K\ot K$ is an algebra in $\ydH $ with product
\begin{align}
  (x\ot y)(z\ot w)=x(\ad \,\pi (y\_1))(z)\ot y\_2 w
  \label{eq:KKprod}
\end{align}
for all $x,y,z,w\in K$. Note that this is the usual product of $K\ot K$ in
${}_{\NA (V)\#H}^{\NA (V)\#H}\yd $. Similarly, $K\#B(V^*) \ot K\#\NA (V^*)$
is an algebra in $\ydH $ with product
\begin{align}
  (x\ot y)(z\ot w)=x(y\_{-1}\lact z)\ot y\_0 w
  \label{eq:KBKBprod}
\end{align}
for all $x,y,z,w\in K\#\NA (V^*)$.

\begin{theor} \label{th:Ralg}
  The map $\Rmat :K\ot K\to K\#\NA (V^*)\ot K\#\NA (V^*)$
  is an algebra map in $\ydH $.
\end{theor}

\begin{proof}
  The map $\Rmat $ is a morphism in $\ydH $ by Lemma~\ref{le:RisHlin}.
  Further,
  \begin{align*}
    &\Rmat ( (x\ot y)(1\ot z))=\Rmat (x\ot yz)\\
    &\quad =\sum _\al \adR{b_\al }{x}\ot b^\al yz
    =\Rmat (x\ot y)(1\ot z)=\Rmat (x\ot y)\Rmat(1\ot z)
  \end{align*}
  for all $x,y,z\in K$ by Eq.~\eqref{eq:Rmat1}. Hence it suffices to show that
  \begin{align}
    \Rmat ((x\ot 1)(y\ot 1))=&\Rmat (x\ot 1)\Rmat (y\ot 1),
    \label{eq:Rmathom1}\\
    \Rmat ((1\ot x)(y\ot 1))=&\Rmat (1\ot x)\Rmat (y\ot 1)
    \label{eq:Rmathom2}
  \end{align}
  for all $x,y\in K$. Let $x,y\in K$. Then
  \begin{align*}
    \Rmat (x\ot 1)&\Rmat (y\ot 1)
    \eqn{\eqref{eq:Rmat}} \Big(\sum _\al \adR{b_\al }{x}\ot b^\al
    \Big) \Big(\sum _\beta \adR{b_\beta}{y}\ot b^\beta \Big)\\
    \eqn{\eqref{eq:KBKBprod}}&
    \sum _{\al ,\beta }(\adR{b_\al }{x})(b^\al {}\_{-1}\lact
    (\adR{b_\beta }{y}))\ot b^\al {}\_0b^\beta \\
    \eqn{\eqref{eq:bb2a}}&
    \sum _{\al ,\beta }(\adR{b_\al {}\_0}{x})(S^{-1}(b_\al {}\_{-1})\lact
    (\adR{b_\beta }{y}))\ot b^\al b^\beta \\
    \eqn{\eqref{eq:bb7}}&
    \sum _\gamma (\adR{b_\gamma \^2 {}\_0}{x})(S^{-1}
    (b_\gamma \^2{}\_{-1})\lact
    (\adR{b_\gamma \^1}{y}))\ot b^\gamma \\
    \eqn{\eqref{eq:adR4}}&
    \sum _\gamma \adR{b_\gamma }{(xy)}\ot b^\gamma
    \eqn{\eqref{eq:Rmat}}
    \Rmat (xy\ot 1).
  \end{align*}
  This proves Eq.~\eqref{eq:Rmathom1}. Further,
%  (1\#f)(x\#1)=\langle f\^2,\pi (x\^1)\rangle x\^2{}\_0 \#
%  S^{-1}(x\^2{}\_{-1})\lact f\^1
  \begin{align*}
    \Rmat( (1\ot &x)(y\ot 1))
    \eqn{\eqref{eq:KKprod}}
    \Rmat ( (x_{[-1]}\lact y)\ot x_{[0]})\\
    \eqn{\eqref{eq:Rmat}}&
    \sum _\al \adR{b_\al }{(x_{[-1]}\lact y)}\ot b^\al x_{[0]}
    \eqn{\eqref{eq:adR}}
    \sum _\al \adR{(S(x_{[-1]})b_\al )}{y}\ot b^\al x_{[0]}\\
    \eqn{\eqref{eq:KB(V*)}}&
    \sum _\al \adR{(S(x_{[-1]})b_\al )}{y}\ot
    \langle b^\al {}\^2,
    \pi (x_{[0]}{}\^1)\rangle x_{[0]}{}\^2{}\_0
    (S^{-1}(x_{[0]}{}\^2{}\_{-1})\lact b^\al {}\^1)\\
    \eqn{\eqref{eq:bb3}}&
    \sum _\al \adR{(S(x_{[-1]})\pi (x_{[0]}{}\^1)b_\al )}{y}\ot
    x_{[0]}{}\^2{}\_0 (S^{-1}(x_{[0]}{}\^2{}\_{-1})\lact b^\al )\\
    =&\sum _\al \adR{(S\pi (x\_1)\pi (x\_2)S\pi _H(x\_3)b_\al )}{y}\ot
    x\_5(S^{-1}\pi _H(x\_4)\lact b^\al )\\
    \eqn{\eqref{eq:bb1a}}&
    \sum _\al \adR{(S\pi _H(x\_1)(\pi _H(x\_2)\lact b_\al ))}{y}\ot
    x\_3b^\al \\
    =&\sum _\al \adR{(b_\al S(x\_{-1}))}{y}\ot x\_0b^\al
    \eqn{\eqref{eq:adR}}
    \sum _\al x\_{-1}\lact (\adR{b_\al }{y})\ot x\_0b^\al \\
    \eqn{\eqref{eq:KBKBprod}}&
    (1\ot x)\Big(\sum _\al \adR{b_\al }{y}\ot b^\al \Big)
    \eqn{\eqref{eq:Rmat},\eqref{eq:Rmat1}}
    \Rmat(1\ot x)\Rmat(y\ot 1).
  \end{align*}
  This proves Eq.~\eqref{eq:Rmathom2}, and the proof of the theorem is
  completed.
\end{proof}

\section{Reflections of Nichols algebras}
\label{sec:refl}

%Now let us consider a more specific situation.
Let $\theta \in \ndN $ and $\Ib =\{1,2,\dots\theta \}$.
Let $\ffdi $ denote the set of $\theta $-tuples of finite-dimensional
irreducible objects in $\ydH $,
% where morphisms are the componentwise isomorphisms:
% for $M,N\in \ffdi $ let $f\in \Hom (M,N)$ if and only if
% $f=(f_1,\dots,f_\theta )$,
% where $f_i:M_i\to N_i$ is an isomorphism of objects
% in $\ydH $. Composition of morphisms is componentwise composition of
% isomorphisms of Yetter-Drinfeld modules.
and let
$\fiso $ denote the set of $\theta $-tuples of
isomorphism classes of finite-dimensional irreducible objects in $\ydH $.
For each $M=(M_1,\dots,M_\theta )\in \ffdi $
let $[M]=([M_1],\dots,[M_\theta ])\in \fiso $ denote the corresponding
$\theta $-tuple of isomorphism classes.

Let $\{\al _1,\dots,\al _\theta \}$ be the standard basis of $\ndZ ^\theta $.
For all $M=(M_1,\dots,M_\theta )\in \ffdi $
define an algebra grading by $\ndN _0^\theta $
on the Nichols algebra $\NA (M):=\NA (M_1\oplus \cdots \oplus M_\theta )$
by $\deg M_j=\al _j$ for all $j\in \Ib $, see \cite[Rem.~2.8]{p-AHS08}.
We call this the \textit{standard $\ndN _0^\theta $-grading of} $\NA (M)$.

Let $i\in \Ib $ and $M\in \ffdi $. Let
\begin{align}
  K_i^M=&\NA (M)^{\co \NA (M_i)\# H}.
% & K'=&\NA (M')^{\co \NA (M'_i)\#H}.
\end{align}
Clearly, $M_j\subset K_i^M$ for all $j\in \Ib \setminus \{i\}$.
As in \cite[Def.\,6.4]{p-HeckSchn08a}
we say that $M$ is \textit{$i$-finite} if
$(\ad\,\NA (M_i)\#H)(M_j)$ is finite-dimensional for all
$j\in \Ib \setminus \{i\}$. Note that if $N\in \ffdi $ with $[M]=[N]$ then
$M$ is $i$-finite if and only if $N$ is $i$-finite.
% morphisms of $\ffdi $ preserve (and reflect) $i$-finiteness.

\begin{propo} \cite[Prop.\,3.6]{p-AHS08}
  Let $i\in \Ib $ and $M\in \ffdi $. Assume that $M$ is $i$-finite.

  (i) The algebra $K_i^M$ is generated by $\oplus _{j\in \Ib \setminus \{i\}}
  (\ad \,\NA (M_i))(M_j)$.

  (ii) The left adjoint action of $\NA (M_i)$ on $K_i^M$ is locally finite.
  \label{pr:Klocfin}
\end{propo}

Let us recall some crucial definitions introduced
in \cite[Sect.\,3.4]{p-AHS08}.
% Define $\Omega _i$ as a natural isomorphism, and the
% natural isomorphism $\varphi _i$: $\id \simeq \Rf _i^2$.
Let $i\in \Ib $ and $M\in \ffdi $.
If $M$ is not $i$-finite, let $\Rf _i(M)=M$.
Otherwise let
%$s_i^M\in \Aut (\ndZ ^\theta )$,
$a^M_{ij}\in \ndZ $, where $j\in \Ib $,
and $M'=(M'_1,\dots,M'_\theta )\in (\ydH )^\theta $ be given by
\begin{gather}
  a^M_{ij}=\begin{cases}
    2 & \text{if $j=i$},\\
    -\max \{m\,|\,(\ad \,M_i)^m(M_j)\not=0\} & \text{if $j\not=i$,}
  \end{cases}
  \label{eq:aij}\\
  M'_i=M_i^*,\quad
  M'_j=(\ad\,M_i)^{-a^M_{ij}}(M_j)\,
  \text{for all $j\in \Ib \setminus \{i\}$}
  \label{eq:M'_j}.
\end{gather}
and let
\begin{align}
  s_i^{M}\in \Aut (\ndZ ^\theta ),\quad
  s_i^{M}(\al _j)=\al _j-a^{M}_{ij}\al _i
  \text{ for all $j\in \Ib $.}
  \label{eq:siM}
\end{align}
By \cite[Thm.\,3.8]{p-AHS08},
$M'_j$ is finite-dimensional and irreducible for all $j\in \Ib $,
and hence $M'\in \ffdi $.
Let $\Rf _i(M)=M'$ in this case. Note that $[\Rf _i(M)]=[\Rf _i(N)]$ in
$\fiso $ for all
$N\in \ffdi $ with $[N]=[M]$. Thus we may define
\[ \rf _i([M])=[\Rf _i(M)], \]
and these definitions provide us with maps $\Rf _j:\ffdi \to \ffdi $, $\rf
_j:\fiso \to \fiso $ for all $j\in \Ib $. Further, if $N\in \ffdi $ with
$[M]=[N]$ and $M$ is $i$-finite, then $a_{ij}^M=a_{ij}^N$ for all $j\in \Ib $,
and $s_i^M=s_i^N$. Thus we may write $a_{ij}^{[M]}$ and $s_i^{[M]}$ instead of
$a_{ij}^M$ and $s_i^M$ if needed.

%where $\NA (f):\NA (M)\to \NA (N)$ is the unique isomorphism of braided Hopf
%algebras in $\ydH $ determined by $\NA (f)|_{M_j}=f|_{M_j}$ for all
%$j\in \Ib $.

%This definition of $\rf _i([M])$ makes sense, since the isomorphism classes
%$[M'_j]$, where $j\in \Ib $, depend only on the isomorphism classes
%$[M_j]$, $j\in \Ib $.
%The formulas in \eqref{eq:riM}

%together with the above definitions in the case when $M$ is not
%$i$-finite, yield maps
%$\rf _j:\fiso \to \fiso $ for each $j\in \Ib $.
%and $r_i(M)=r_i(M)_1\oplus \cdots \oplus r_i(M)_\theta $.
%Let $N=\Rf _i(\Rf _i(M))$. By \cite[Thm.\,3.12]{p-AHS08}, $[N_j]=[M_j]$ for
%all $j\in \Ib $ and $s_i^{[M]}=s_i^{[\Rf _i(M)]}$. Thus $\Rf _j$ is
%invertible and $\rf _j^2=\id $ for all $j\in \Ib $.

Let $i\in \Ib $ and $M\in \ffdi $. Assume that $M$ is $i$-finite.
Let
\begin{align}
  \Omega _i^M:K_i^M\#\NA (M_i^*)\to \NA (\Rf _i(M))
\end{align}
be the unique algebra map which is the identity on all $M'_j\subset
K_i^M\#\NA (M_i^*)$, where
$j\in \Ib $ --- see \cite[Prop.\,3.14]{p-AHS08}.
%In particular, $\Omega _i^M |_{\NA (r_i(M)_i)}=\id $.
Then $\Omega _i^M$ is a map in $\ydH $ and
\begin{align}
  \Omega _i^M(x_\al )\in \NA (\Rf _i(M))_{s_i^M(\al )}\quad
  \label{eq:Omgrad}
\end{align}
for all $x_\al \in K_i^M\#\NA (M_i^*)$ of degree $\al \in \ndZ ^\theta $,
where $K_i^M\subset \NA (M)$ is graded by the standard grading of $\NA (M)$
and $\deg M_i^*=-\al _i$. Further, $\Omega _i^M$ is bijective
and
\begin{align}
  S_{\NA (\Rf _i(M))}\Omega _i^M(K_i^M)=K_i^{\Rf _i(M)}
  \label{eq:OmK}
\end{align}
by the last paragraph of the proof of \cite[Thm.~3.12]{p-AHS08}.
By \cite[Eq.~(3.37)]{p-AHS08} the restriction of
$S_{\NA (\Rf _i(M))}\Omega _i^M$ to $M_j$
defines an isomorphism
\[ \varphi _{ij}^M:M_j\to \Rf _i^2(M)_j,\quad j\in \Ib \setminus \{j\}, \]
in $\ydH $, and there is a canonical isomorphism $\varphi _{ii}^M:
M_i\to \Rf _i^2(M)_i=M_i^{**}$ in $\ydH $, see \cite[Rem.\,1.4]{p-AHS08}. Let
\begin{align} \label{eq:varphi}
  \varphi _i^M=(\varphi _{ij}^M)_{j\in \Ib }:M\to \Rf ^2_i(M)
\end{align}
be the family of these isomorphisms.

The following property of $\Omega _i^M $  will be one of the main ingredients to
characterize right coideal subalgebras of Nichols algebras.

\begin{theor}\label{th:Omega}
  Let $M\in \ffdi $ and $i\in \Ib $. Assume that $M$ is $i$-finite. Then
  the following diagram is commutative:
  \begin{equation}
    \begin{CD}
      K_i^M & @>\copr _{K_i^M}>> {K_i^M}\ot {K_i^M}
      \overset{\Rmat }{\longrightarrow }
      &{K_i^M}\ot ( {K_i^M}\#\NA (M_i^*)) \\
      @V\Omega _i^M VV & & @VV\Omega _i^M \ot \Omega _i^M V\\
      \NA (\Rf _i(M)) & @>\copr _{\NA (\Rf _i(M))}>>
      &\NA (\Rf _i(M))\ot \NA (\Rf _i(M))
    \end{CD}
    \label{eq:Ompropdiag}
  \end{equation}
  that is,
  \begin{align}
    \copr _{\NA (\Rf _i(M))}\Omega _i^M (x) =
    (\Omega _i^M \ot \Omega _i^M )\Rmat \copr _{K_i^M}(x)
    \label{eq:Omprop}
  \end{align}
  for all $x\in {K_i^M}$.
\end{theor}

%\begin{proof}
%  Eq.~\eqref{eq:coprL} holds by \cite[Eq.~(3.11)]{p-AHS08}.
%  Let $x'=S_{\NA (r_i(M))}\Omega _i^M (x)$ and $L'_j=S_{\NA (r_i(M))}\Omega _i^M (L_j)$.
%  Then
%  \[ \copr _{\NA (r_i(M))}(x')-x'\ot 1\in \NA (M_i^*)\ot L'_j. \]
%  by \cite[Eqs.~(3.11),(3.34)]{p-AHS08}.
%  Since $\copr _{\NA (r_i(M))}(\Omega _i^M (x))=
%  \copr _{\NA (r_i(M))}S^{-1}_{\NA (r_i(M))}(x')$,
%  Eq.~\eqref{eq:SRiprop2} gives the claim.
%\end{proof}

\begin{proof}
  For all $j\in \Ib \setminus \{i\}$ and $k\in \ndN $
  let
  \begin{align}
    L_j=(\ad \,\NA (M_i))(M_j),\quad L_j^k=(\ad \,M_i)^{k-1}(M_j).
    \label{eq:Lj}
  \end{align}
  Then
  \begin{align}
    \copr _{\NA (M)}(x)-x\ot 1 \in &\NA (M_i)\ot L_j
    \label{eq:coprL}
%    \\
%    \copr _{\NA (r_i(M))}(\Omega _i^M (x))-1\ot \Omega _i^M (x) \in &
%    \Omega _i^M (L_j)\ot \NA (M_i^*).
%    \label{eq:coprOmL}
  \end{align}
  for all $j\in \Ib \setminus \{i\}$ and $x\in L_j$
  by \cite[Eq.~(3.11)]{p-AHS08}.

  By \cite[Prop.\,3.14]{p-AHS08},
  $\Omega _i^M :{K_i^M}\#\NA (M_i^*)\to \NA (\Rf _i(M))$ is an algebra map.
  By Thm.~\ref{th:Ralg}, $\Rmat $ is an algebra map.
  By Prop.~\ref{pr:Klocfin}(i), ${K_i^M}$ is generated as an algebra
  by $\cup _{j\not=i}L_j$.
  Hence it suffices to prove that Eq.~\eqref{eq:Omprop} holds
%  \begin{align}
%    (\Omega _i^M \ot \Omega _i^M )\Rmat \copr _{K_i^M}(x)=\copr _{\NA (r_i(M))}\Omega _i^M (x)
%    \quad \text{for all $x\in L_j$, $j\in \Ib \setminus
%    \{i\}$.}
%    \label{eq:Ompropgen}
%  \end{align}
for all $x\in L_j$, $j\in \Ib \setminus \{i\}$.
  Let $j\in \Ib \setminus \{i\}$ and let $x\in L_j$.
  Eq.~\eqref{eq:coprL} implies that $\copr _{K_i^M}(x)=1\ot x+x\ot 1$.
  Hence
  \begin{align}
    (\Omega _i^M \ot \Omega _i^M )\Rmat \copr _{K_i^M}(x)=1\ot \Omega _i^M (x)
    +\sum _\al \Omega _i^M (\adR{b_\al }{x})\ot b^\al
    \label{eq:Rx}
  \end{align}
  by Eqs.~\eqref{eq:Rmat}, \eqref{eq:Rmat1} and since $\Omega _i^M |_{\NA
  (M_i^*)}=\id $.
%  On the other hand, let $\{b_\al \,|\,\al \in J\}$ and
%  $\{b^\al \,|\,\al \in J\}$ be $\ndN _0$-graded dual bases of $\NA (M_i)$
%  and $\NA (M_i^*)$, respectively, where $J$ is an appropriate index set.
%  Then Eq.~\eqref{eq:coprOmL} implies that
%  there exists a family $\{x_\al \in L_j\,|\,\al \in J\}$ such that $x_\al =0$
%  for almost all $\al $ and
%  \begin{align}
%    \copr _{\NA (r_i(M))}(\Omega _i^M (x))-1\ot \Omega _i^M (x) = \sum _\al \Omega _i^M (x_\al
%    ) \ot b^\al .
%    \label{eq:coprOmx}
%  \end{align}
%  Further, if $x\in (\ad \,\NA ^k(M_i))(M_j)$, then
%  \begin{align}
%    \sum _\al x_\al \ot b^\al \in
%%    \sum _{l=0}^\infty (\ad \,\NA ^{k+l}(M_i))(M_j)\ot
%    \NA ^l(M_i^*)
%    \label{eq:}
%  \end{align}
%  since $\Omega _i^M $ is $\ndZ ^\theta $-graded. We use this fact to prove
%  Eq.~\eqref{eq:Rx} by induction on $-a_{ij}^M-k$, where
%  $x\in (\ad \,\NA ^k(M_i))(M_j)$.

  Since $M$ is $i$-finite, Prop.~\ref{pr:Klocfin}(ii) tells that
  the left adjoint action of $\NA (M_i)$ on ${K_i^M}$ is locally
  finite.  Then \cite[Thm.\,3.8]{p-AHS08} can be applied, that is,
  $L_j^{1-a_{ij}^M}$ generates
  $L_j$ as a $\NA (M_i)$-comodule. By Lemma~\ref{le:NAcomod}
  it suffices to show that the set of solutions of
  Eq.~\eqref{eq:Omprop} contains $L_j^{1-a_{ij}^M}$ and is stable under the
  maps $\partial ^L_f$ for all $f\in M_i^*$.

  Suppose first that $x\in L_j^{1-a_{ij}^M}$. Then
  \begin{align*}
    (\Omega _i^M \ot \Omega _i^M )\Rmat \copr _{K_i^M}(x)=
    1\ot \Omega _i^M (x)+\Omega _i^M (x)\ot 1
    =\copr _{\NA (\Rf _i(M))}(\Omega _i^M (x))
  \end{align*}
  by Eq.~\eqref{eq:Rx} and since $\Omega _i^M (x)\in \Rf _i(M)_j$.
  Thus the set of solutions of
  Eq.~\eqref{eq:Omprop} contains $L_j^{1-a_{ij}^M}$.

  Let $k\in \ndN $, $k\le 1-a_{ij}^M$. Assume that
  Eq.~\eqref{eq:Omprop} holds for all $x\in L_j^k$. That is,
  \begin{align}
    \copr _{\NA (\Rf _i(M))}\Omega _i^M (x)=1\ot \Omega _i^M (x)
    +\sum _\al \Omega _i^M (\adR{b_\al }{x})\ot b^\al  \quad \text{for all $x\in
    L_j^k$.}
    \label{eq:indhyp}
  \end{align}
  Let $y\in L_j^k$ and $f\in M_i^*$.
  We have to prove that Eq.~\eqref{eq:indhyp} holds for $x=\partial ^L_f(y)$.
  Since $\partial ^L_f(y)=fy-y\_0(S^{-1}(y\_{-1})\lact f)$
  in ${K_i^M}\#\NA (M_i^*)$ by Eq.~\eqref{eq:fx},
  and since $\Omega _i^M $ is an algebra map in $\ydH $
  with $\Omega _i^M |_{\NA (M_i^*)}=\id $, we obtain that
  \begin{align*}
    &\copr _{\NA (\Rf _i(M))}\Omega _i^M (\partial ^L_f(y))
    =\copr _{\NA (\Rf _i(M))}(f \Omega _i^M (y)
    -\Omega _i^M (y\_0)(S^{-1}(y\_{-1})\lact f))\\
    &\quad =(f\ot 1+1\ot f)\Big(1\ot \Omega _i^M (y)+\sum _\al
    \Omega _i^M (\adR{b_\al }{y})\ot b^\al \Big)-\Big(1\ot \Omega _i^M (y\_0)
     \\
    &\quad \quad +\sum _\al
    \Omega _i^M (\adR{b_\al }{y\_0})\ot b^\al \Big)(1\ot S^{-1}(y\_{-1})\lact f
    +S^{-1}(y\_{-1})\lact f\ot 1).
  \end{align*}
  By applying the product rule in $\NA (\Rf _i(M))\ot \NA (\Rf _i(M))$
  this becomes
  \begin{align*}
    &\copr _{\NA (\Rf _i(M))}\Omega _i^M (\partial ^L_f(y))
    =1\ot (f \Omega _i^M (y)-\Omega _i^M (y\_0)(S^{-1}(y\_{-1})\lact f ))
     \\
    &\quad \quad +f\ot \Omega _i^M (y)
    -y\_{-1}S^{-1}(y\_{-2})\lact f\ot \Omega _i^M (y\_0) \\
    &\quad \quad +\sum _\al (f\Omega _i^M (\adR{b_\al }{y})\ot b^\al
    +f\_{-1}\lact \Omega _i^M (\adR{b_\al }{y})\ot f\_0b^\al ) \\
    &\quad \quad -\sum _\al \Omega _i^M (\adR{b_\al }{y\_0})\ot b^\al
    (S^{-1}(y\_{-1})\lact f) \\
    &\quad \quad -\sum _\al \Omega _i^M (\adR{b_\al }{y\_0})
    (b^\al{}\_{-1}S^{-1}(y\_{-1})\lact f)\ot b^\al {}\_0.
  \end{align*}
  In the last expression, the first line equals $1\ot \Omega _i^M (\partial
  ^L_f(y))$ and the second is zero. We rewrite all other terms such that
  the second tensor factors contain only $b^\al $.

  Eqs.~\eqref{eq:bb6} and \eqref{eq:pairing3} and the definition of
  $\adR{}{}$ yield that
  \begin{align*}
    \sum _\al & f\_{-1}\lact \Omega _i^M (\adR{b_\al }{y})\ot f\_0b^\al
    =\sum _\al f\_{-1}\lact (\langle f\_0 ,b_\al {}\^2\rangle
    \Omega _i^M (\adR{b_\al{}\^1 }{y}))\ot b^\al \\
    =&\sum _\al \langle f ,b_\al {}\^2{}\_0\rangle
    S^{-1}(b_\al {}\^2{}\_{-1})\lact \Omega _i^M (\adR{b_\al{}\^1 }{y})\ot b^\al \\
    =&\sum _\al \langle f ,b_\al {}\^2{}\_0\rangle
    \Omega _i^M (\adR{b_\al{}\^1 b_\al {}\^2{}\_{-1}}{y})\ot b^\al
    =\sum _\al \langle f ,b_\al {}\_2\rangle
    \Omega _i^M (\adR{b_\al{}\_1}{y})\ot b^\al ,
  \end{align*}
  Eq.~\eqref{eq:bb5} implies that
  \begin{align*}
    \sum _\al \Omega _i^M (\adR{b_\al }{y\_0}){\ot }b^\al
    (S^{-1}(y\_{-1})\lact f)
    =&\sum _\al \langle S^{-1}(y\_{-1})\lact f,b_\al {}\^1\rangle
    \Omega _i^M (\adR{b_\al {}\^2}{y\_0}) {\ot }b^\al ,
  \end{align*}
  and from Eq.~\eqref{eq:bb2a} we conclude that
  \begin{align*}
    &\sum _\al \Omega _i^M (\adR{b_\al }{y\_0})
    (b^\al{}\_{-1}S^{-1}(y\_{-1})\lact f)\ot b^\al {}\_0 \\
    & \quad =\sum _\al \Omega _i^M (\adR{b_\al {}\_0}{y\_0})
    (S^{-1}(y\_{-1}b_\al{}\_{-1})\lact f)\ot b^\al .
  \end{align*}
  On the other hand,
  \begin{align*}
    &(\Omega _i^M \ot \Omega _i^M )\Rmat \copr _{K_i^M}(\partial ^L_f(y))=
    (\Omega _i^M \ot \Omega _i^M )\Rmat (1\ot \partial ^L_f(y)+\partial ^L_f(y)\ot 1)\\
    &\quad =1\ot \Omega _i^M (\partial ^L_f(y))
    +\sum _\al \Omega _i^M (\adR{b_\al }{\partial ^L_f(y)})\ot b^\al .
  \end{align*}
  Comparing coefficients in front of $b^\al $ we conclude that
  Eq.~\eqref{eq:indhyp}  holds for $x=\partial ^L_f(y)$ if equation
  \begin{equation}
  \begin{aligned}
    f(\adR{b}{y}) +\langle f,b\_2\rangle (\adR{b\_1}{y})
    -\langle S^{-1}(y\_{-1})\lact f,b\^1\rangle \adR{b\^2}{y\_0}& \\
    -(\adR{b\_0}{y\_0})(S^{-1}(y\_{-1}b\_{-1})\lact f)
    =&\adR{b}{\partial ^L_f(y)}
  \end{aligned}
    \label{eq:Omproppart}
  \end{equation}
  holds in ${K_i^M}\#\NA (M_i^*)$ for all $b\in \NA (M_i)$.
  Using Eq.~\eqref{eq:KB(V*)} with $f\in M_i^*$,
  Eq.~\eqref{eq:Omproppart} becomes equivalent to
  \begin{align*}
    \partial ^L_f(\adR{b}{y}) =-\langle f,b\_2\rangle \adR{b\_1}{y}
    +\langle S^{-1}(y\_{-1})\lact f,b\^1\rangle \adR{b\^2}{y\_0}
    +\adR{b}{\partial ^L_f(y)}.
  \end{align*}
  The latter is true by Lemma~\ref{le:adRprop}(ii). Thus
  Eq.~\eqref{eq:indhyp} (and hence Eq.~\eqref{eq:Omprop})
  holds for $x=\partial ^L_f(y)$ and hence for all $x\in
  L_j$. This finishes the proof of the theorem.
\end{proof}

\section{Right coideal subalgebras of Nichols algebras}
\label{sec:Rcs}

%We write $\copr _{\NA (V)}$ and $\copr $ for the coproduct
%of the braided Hopf algebra $\NA (V)$ and the Hopf algebra $\NA (V)\# H$,
%respectively.
%\begin{propo}
%  Let $V_1,V_2\in \ydH $. Let $W=V_1\oplus V_2$ and
%  ${K_i^M}=\NA (W)^{\co \NA (V_1)\# H}$.
%  Then the map $\Phi $ from the class of subsets of $\NA (W)\#H$ to the class
%  of subsets of ${K_i^M}$, defined by $\Phi (B)=B\cap {K_i^M}$,
%  yields a bijection between the following subclasses:
%\begin{enumerate}
%  \item right coideal subalgebras of $\NA (W)\#H$ containing $V_1\oplus H$,
%  \item right coideal subalgebras $\rcs $ of ${K_i^M}$ such that
%    $(\ad\,\NA (V_1)\#H)(\rcs )\subset \rcs $ and
%    $\coa _H(\rcs )\subset \rcs $.
%\end{enumerate}
%\end{propo}

Let $\theta \in \ndN $ and $M=(M_1,\dots, M_\theta )\in \ffdi $.
Let
$\cK M $ denote the set of all $\ndN _0^\theta $-graded
right coideal subalgebras of $\NA (M)$ in $\ydH $, where $\NA (M)$ is graded
by the standard $\ndN _0^\theta $-grading, see Sect.~\ref{sec:refl}.
%Let $\cK _i(\NA (N))$ be the subset of $\cK (\NA (N))$
%containing those right coideal subalgebras $\rcs $ with $N_i\subset \rcs $,
%and let $\cK _{-i}(\NA (N))=
%\cK (\NA (N))\setminus \cK _i(\NA (N))$.

For all $\al \in
\ndZ ^\theta $ let $t^\al =t_1^{n_1}\cdots t_\theta ^{n_\theta }\in \ndN
_0[t_1^{\pm 1},\dots,t_\theta ^{\pm 1}]$,
where $\al =\sum _{i\in \Ib }n_i\al _i$.
For any $N\in \ffdi $ and any $\ndN _0^\theta $-graded
object $X=\oplus _{\al \in \ndN _0^\theta }X_\al \subset \NA (N)$
in $\ydH $ let
\begin{align}
  \Hilb _X(t)
%  = \sum _{\makebox[0pt]{\tiny $n_1,\dots,n_\theta \in \ndN _0$}}
%  (\dim X_{\sum _{i\in \Ib }n_i\al _i})t_1^{n_1}\cdots t_\theta ^{n_\theta }
  =\sum _{\al \in \ndN _0^\theta }(\dim X_\al )t^\al
  \in \ndN _0[ [ t_1,\dots, t_\theta ]]
  \label{eq:HS}
\end{align}
be the \textit{(multivariate) Hilbert series} of $X$.

There is a $\ndZ $-linear action
of $\mathrm{GL}(\theta ,\ndZ )$ on $\ndZ [ t_1^{\pm 1},
\dots ,t_\theta ^{\pm 1} ]$ via $gt^\al =t^{g(\al )}$ for all $g\in
\mathrm{GL}(\theta ,\ndZ )$, $\al \in \ndZ ^\theta $.
This extends to a partially defined $\ndZ $-linear action
of $\mathrm{GL}(\theta ,\ndZ )$ on $\ndZ [ [t_1,\dots,t_\theta ] ]$:
the action of each
$g\in \mathrm{GL}(\theta ,\ndZ )$ is well-defined on the subring of
$\ndZ [ [t_1,\dots,t_\theta ]]$ consisting of those formal power series
$\sum _{\al \in \ndN _0^\theta }a_\al t^\al $, where $a_\al \in \ndZ $ for
all $\al $ and $a_\al =0$ if $g(\al )\notin \ndN _0^\theta $.

We start our considerations of right coideal subalgebras
with general lemmata.

\begin{lemma}
  Let $M\in \ffdi $ and let $\rcs $ be an $\ndN _0^\theta $-graded
  right coideal of $\NA (M)$ in $\ydH $. If $\rcs \not=\fie 1$,
  then there exists $i\in \Ib $
  such that $M_i\subset \rcs $.
  \label{le:rcshasMi}
\end{lemma}

\begin{proof}
  Let $\pr _1:\NA (M)\to M_1\oplus \cdots \oplus M_\theta $
  be the $\ndN _0$-graded projection
  to the homogeneous elements of degree $1$. Recall that the map
  \[ (\pr _1\ot \id )\copr _{\NA (M)}: \oplus _{n=1}^\infty \NA ^n(M)
  \to (M_1\oplus \cdots \oplus M_\theta )\ot \NA (M) \]
  is injective.
  By assumption, $\copr _{\NA (M)}(E)\subset E\ot \NA (M)$,
  and $E\not=\fie 1$. Thus
  \[ 0\not=(\pr _1\ot \id )\copr _{\NA (M)}(E)
  \subset \big( E\cap (M_1\oplus \cdots \oplus M_\theta )\big)\ot \NA (M) \]
  and hence $E\cap (M_1\oplus \cdots \oplus M_\theta )\not=0$.
  Since $E$ is $\ndN _0^\theta $-graded, there exists $i\in\Ib $
  such that $E\cap M_i\not=0$. Since $E\in \ydH $ and $M_i\in \ydH $ is
  irreducible, this implies that $M_i\subset E$.
\end{proof}

\begin{lemma}
  Let $M\in \ffdi $, $i\in \Ib $, and $E\in \cK M$.
  Then $M_i\not \subset E$ if and only if $S_{\NA (M)}(\rcs )\subset
  K_i^M$.
  \label{le:MnotinE}
\end{lemma}

\begin{proof}
  The homogeneous subspace of $K ^M_i$ of degree $\al _i$ is $0$,
  hence $S_{\NA (M)}(\rcs )\subset K_i^M$ implies that $M_i\not \subset
  E$. Conversely, assume that $M_i\not \subset E$ and let $\pi : \NA (M)\to
  \NA (M_i)$ be the canonical map.
  Then $M_i\not \subset \pi (E)$ since $E$ is $\ndN _0^\theta $-homogeneous.
  Since $\pi (E)$ is a right coideal of $\NA (M_i)$, we conclude that
  $\pi (E)=\fie 1$ by Lemma~\ref{le:rcshasMi}
  and hence
  \begin{align}
    \pi (S_{\NA (M)}(E))=\fie 1.
    \label{eq:piSE}
  \end{align}
  On the other hand, $S_{\NA (M)}(E)$ is a left coideal subalgebra of $\NA
  (M)$, and by Eq.~\eqref{eq:piSE} it is contained in $\NA (M)^{\co \NA
  (M_i)}=K_i^M$ which proves the lemma.
\end{proof}

\begin{corol} \label{co:MnotinE}
  Let $M\in \ffdi $, $i\in \Ib $, and $E\in \cK {\Rf _i(M)}$. Assume that $M$
  is $i$-finite and $\Rf _i(M)_i\not \subset E$.
  Then $(\Omega _i^M)^{-1}(E) \subset K_i^M$.
\end{corol}

\begin{proof}
%  Since $M$ is $i$-finite, \cite[Thm.~3.12(2)]{p-AHS08} gives that $\Rf _i(M)$
%  is $i$-finite, and hence
%  the map $\Omega _i^{\Rf _i(M)}$ is bijective, see Sect.~\ref{sec:refl}.
%  Since $M_i\not \subset E$, we obtain that
%  $\Rf _i^2(M)_i\not \subset \NA (\varphi _i^M)(E)$.
%  Then Lemma~\ref{le:MnotinE} tells that
%  \begin{align} \label{eq:SBEinSOK}
%    S_{\NA (\Rf _i^2(M))}(\NA (\varphi _i^M)(E))\subset K_i^{\Rf _i^2(M)}
%    =S_{\NA (\Rf _i^2(M))}\Omega _i^{\Rf _i(M)}(K_i^{\Rf _i(M)}),
%  \end{align}
%  where the last equation holds by Eq.~\eqref{eq:OmK}. The relation
%  \eqref{eq:SBEinSOK} gives the claim.
  Since $M$ is $i$-finite
  the map $\Omega _i^M$ is bijective, see Sect.~\ref{sec:refl}.
  Since $\Rf _i(M)_i\not \subset E$, Lemma~\ref{le:MnotinE} yields that
  \begin{align} \label{eq:SBEinSOK}
    S_{\NA (\Rf _i(M))}(E)\subset K_i^{\Rf _i(M)}
    =S_{\NA (\Rf _i(M))}\Omega _i^M(K_i^M),
  \end{align}
  where the last equation holds by Eq.~\eqref{eq:OmK}. The relation
  \eqref{eq:SBEinSOK} gives the claim.
\end{proof}

\begin{lemma}
  \label{le:notIfin}
  Let $M\in \ffdi $ and $i\in \Ib $. Suppose that $M$ is not $i$-finite.
  Then there exist infinitely many $\ndN _0 ^\theta $-graded
  right coideal subalgebras of $\NA (M)$ in $\ydH $ which do not contain
  any $M_j$ with $j\in \Ib \setminus \{i\}$.
\end{lemma}

\begin{proof}
  Let $k\in \Ib \setminus \{i\}$ such that $\dim (\ad \,\NA (M_i))(M_k)=\infty
  $. For all $n\in \ndN $ let
  $\rcs _n$ be the subalgebra of $\NA (M)$ generated by
  $(\ad \,\NA ^n(M_i))(M_k)$ and $M_i$. By assumption, $\rcs _n\not=0$ for all
  $n\in \ndN $. By construction, $E_n\subset E_m$ for all $m\le n$ and
  \begin{align}
    (\rcs _n)_{\al _k+m\al _i}=
    \begin{cases}
      0 & \text{if $m<n$,}\\
      (\ad \,\NA ^n(M_i))(M_k) & \text{if $m=n$.}
    \end{cases}
    \label{eq:degrscm}
  \end{align}
  Hence $E_1\supset E_2\supset \cdots $ is a strictly decreasing sequence of
  nontrivial $\ndN _0 ^\theta $-graded subalgebras of $\NA (M)$ in $\ydH $
  with $E_n\cap M=M_i$ for all $n\in \ndN $.
  It remains to prove that each $E_n$ is a right
  coideal of $\NA (M)$. But this is true since $\copr _{\NA (M)}(x)\in
  E_n\ot \NA (M)$ for each generator $x$ of $E_n$ by Eq.~\eqref{eq:coprL}.
\end{proof}

Recall that $K_i^M$ is a Hopf algebra in the braided category
${}_{\NA (M_i)\#H}^{\NA (M_i)\#H}\yd $. Its comultiplication is denoted by
$\Delta_{K_i^M}$.
Assume that $M$ is $i$-finite. Regard $K_i^M \# \NA (M_i^*)$
as a Hopf algebra in $\ydH$,
such that the algebra map
$\Omega_i^M : K_i^M \# \NA (M_i^*) \to \NA(\Rf _i(M))$ is an isomorphism of
Hopf algebras in $\ydH $.

\begin{lemma}\label{le:F}
  Let $M\in \ffdi $ and $i\in \Ib $.
  Assume that $M$ is $i$-finite. Let $F \subset K_i^M$ be a subalgebra in
  $\ydH$. Then the following are equivalent:
  \begin{enumerate}
  \item $F \NA (M_i)$ is a right coideal subalgebra of $\NA (M)$ in $\ydH $.
  \item $F$ is a right coideal subalgebra of $K_i^M$ in
    ${}_{\NA (M_i)\#H}^{\NA (M_i)\#H}\yd $, and $(\ad \,M_i)(F)\subset F$.
  \item $F$ is a right coideal subalgebra of $K_i^M \# \NA (M_i^*)$ in $\ydH $.
  \end{enumerate}
\end{lemma}

\begin{proof}
  Assume (1) and let $E= F \NA (M_i)$. Then $F= E \cap K_i^M$, and
  $(\ad \,M_i)(F)\subset F$ since $M_i \subset E$. Let $\pi :\NA (M)\to \NA
  (M_i)$ be the canonical map.  Since
  \[ \copr _{K_i^M}(x)=x\^1S_{\NA (M_i)}\pi (x\^2)\ot x\^3 \quad
  \text{for all $x\in K_i^M$,} \]
  we obtain that $\copr _{K_i^M}(E\cap K_i^M)\subset E\ot K_i^M$. This proves
  (2). Similarly, (2) implies (1).

  Assume (2). By definition and Thm.~\ref{th:Omega}
  the restriction of the comultiplication of $K_i^M \# \NA (M_i^*)$ to $K_i^M$
  is given by the map
  \[ K_i^M \xrightarrow{\copr _{K_i^M}} {K_i^M}\ot {K_i^M}
     \overset{\Rmat }{\longrightarrow }
     K_i^M\ot ( K_i^M\#\NA (M_i^*)). \]
  Hence $F$ is a right coideal subalgebra of $K_i^M\#\NA (M_i^*)$ in $\ydH $.

  Conversely, assume (3). Let $\pi' = \cou \otimes
  \id_{\NA(M_i^*)}:K_i^M \# \NA (M_i^*) \to \NA (M_i^*)$.  Then $\pi'$ is an
  algebra map by Eq.~\eqref{eq:KB(V*)}. Let $x \in F$. Since $F$ is a right
  coideal subalgebra of $K_i^M \# \NA (M_i^*)$ and $(\id \ot \pi ')\copr
  _{K_i^M}(x)=x\ot 1$ for all $x\in K_i^M$, it follows that
  \[
    F \otimes \NA (M_i^*)\ni
    (\id_{K_i^M} \otimes \pi') \Rmat \copr _{K_i^M}(x) =
     \sum_{\al } S_{\NA(M_i)\#H}^{-1}(b_{\al }) \cdot x \otimes b^{\al }.
  \]
  Since $F$ is also an $H$-module,
  $F$ is stable under the adjoint action $\cdot$ of $\NA (M_i)$.
  Since $\Rmat \copr _{K_i^M}(F)\subset F\ot (K_i^M\#\NA (M_i^*))$
  by assumption,
  it follows from \eqref{eq:RmatbRmat} that $F$ is a right coideal subalgebra
  of $K_i^M$ in  ${}_{\NA (M_i)\#H}^{\NA (M_i)\#H}\yd $.
\end{proof}

Recall from Eq.~\eqref{eq:varphi} that for all $i\in \Ib $ and $M\in \ffdi $,
$\varphi _i^M:M\to \Rf _i^2(M)$ is a
family of isomorphisms of objects in $\ydH $. Let $\NA (\varphi _i^M):\NA
(M)\to \NA (\Rf _i^2(M))$ be the induced isomorphism of braided Hopf algebras
in $\ydH $.
The following theorem is the key result in the proof of
Thms.~\ref{th:bijective},
\ref{th:finite}. When $M$ is $i$-finite we will use the isomorphisms
\begin{align*}
 K_i^M \# \NA(M_i^*) &\xrightarrow{\Omega_i^M} \NA(\Rf_i(M)),\\
K_i^{\Rf_i(M)} \# \NA(\Rf_i(M)_i^*) &\xrightarrow{\Omega_i^{\Rf_i(M)}} \NA(\Rf_i^2(M)) \xleftarrow{\NA(\varphi_i^M)} \NA(M)
\end{align*}
to define bijections between the right coideal subalgebras of $\NA(M)$ and of $\NA(\Rf_i(M))$.

\begin{theor}\label{th:sigmai}
  Let $M\in \ffdi $ and $i\in \Ib $. Assume that $M$ is
  $i$-finite. Then the maps
  $\sigma ^M_i:\cK M \to \cK {\Rf _i(M)}$
  defined for all $E \in \cK M$ by
  \begin{align*}
    \sigma ^M_i(E)=&
    \begin{cases}
      \Omega ^M_i(E\cap K^M_i) & \text{if $M_i\subset E$,}\\
      (\Omega _i^{\Rf _i(M)})^{-1}( \NA (\varphi _i^M)(E))\NA (\Rf _i(M)_i) &
      \text{if $M_i\not \subset E$,}
    \end{cases}\\
 \end{align*}
 and  $\bar{\sigma }_i^{\Rf_i(M)}:\cK {\Rf _i(M)}\to \cK M $
  defined for all $E \in \cK {\Rf_i(M)}$ by
  \begin{align*}
  \bar{\sigma }_i^{\Rf_i(M)}(E)=&
    \begin{cases}
      \NA (\varphi _i^M)^{-1}(\Omega ^{\Rf _i(M)}_i(E\cap K^{\Rf _i(M)}_i)) &
      \text{if $\Rf _i(M)_i\subset E$,}\\
      (\Omega _i^M)^{-1}(E)\NA (M_i) &
      \text{if $\Rf _i(M)_i\not \subset E$,}
    \end{cases}
  \end{align*}
  are bijective. More precisely, the following hold.

  (1) For all $E\in \cK M$,
  $M_i\subset E$ if and only if $\Rf _i(M)_i\not \subset \sigma _i^M(E)$.
  For all $E\in \cK {\Rf _i(M)}$,
  $\Rf _i(M)_i\subset E$ if and only if
  $M_i\not \subset \bar{\sigma }_i^{\Rf_i(M)}(E)$.

  (2) $\bar{\sigma }_i^{\Rf_i(M)}\sigma ^M_i=\id _{\cK M}$,
  $\sigma ^M_i\bar{\sigma }_i^{\Rf_i(M)}=\id _{\cK {\Rf _i(M)}}$.
\end{theor}

\begin{proof}
  By Cor.~\ref{co:MnotinE} the maps $\sigma ^M_i$ and $\bar{\sigma }_i^{\Rf_i(M)}$
  are well-defined in the sense that $\sigma ^M_i(E)\subset \NA (\Rf _i(M))$
  for all $E\in \cK M$ and
  $\bar{\sigma }_i^{\Rf_i(M)}(E)\subset \NA (M)$ for all $E\in \cK {\Rf _i(M)}$.
  It remains to prove (1) and that
  $\sigma ^M_i$ maps $\cK M $ to $\cK {\Rf _i(M)}$ and
  $\bar{\sigma }_i^{\Rf_i(M)}$ maps $\cK {\Rf _i(M)}$ to $\cK M $.
  Then the equations in (2) follow from Lemma~\ref{le:EcapR}.

  We prove that $\sigma ^M_i(E)\in \cK {\Rf _i(M)}$ for all $E\in \cK M$, and
  that the part of (1) regarding $\sigma _i^M$ holds. The
  analogous claims for $\bar{\sigma }_i^{\Rf_i(M)}$ can be shown similarly.

  Let $E\in \cK M$. Assume first that $M_i\subset E$,
  and let $F=E\cap K_i^M$.
  Since $K^M_i$ is an $\ndN _0^\theta $-graded algebra
  in $\ydH $, $F$ is an $\ndN _0^\theta $-graded
  subalgebra of $\NA (M)$ in $\ydH $. Further,
  $E= F \NA(M_i)$ by Lemma \ref{le:EcapR}.
  By Lemma~\ref{le:F} (1)$\Rightarrow $(3), and
  since $\Omega ^M_i: K^M_i \# \NA (M_i^*)\to \NA (\Rf _i(M))$ is
  an isomorphism of $\ndN _0^\theta $-graded Hopf algebras in $\ydH $,
  we conclude that $\Omega ^M_i(F) \in \cK {\Rf _i(M)}$.
  Further, $\Rf _i(M)_i\not \subset \Omega ^M_i(F)$
  by \eqref{eq:Omgrad} and
  since $(E\cap K^M_i)_{-\al _i}=0$.
%  Therefore $\sigma _i^{\Rf _i(M)}(\Omega ^M_i(F))=E$.

  Assume now that $M_i\not \subset E$.
  Since $\NA (\varphi _i^N):\NA (N)\to \NA (\Rf _i^2(N))$ is an
  isomorphism of $\ndN _0^\theta $-graded braided Hopf algebras,
  we conclude that
  $\Rf _i^2(M)_i\not\subset E$ and
  $\NA (\varphi _i^M)(E)\in \cK {\Rf _i^2(M)}$.
  Further, $\Rf _i^2(M)$ is $i$-finite.
  Let
  $F=(\Omega _i^{\Rf _i(M)})^{-1}( \NA (\varphi _i^M)(E))$.
  Then $F\subset K_i^{\Rf _i(M)}$ by Cor.~\ref{co:MnotinE}. Further, $F$ is an
  $\ndN _0^\theta $-graded subalgebra of $K_i^{\Rf _i(M)}$ and a right coideal
  subalgebra of $K_i^{\Rf _i(M)} \# \NA(\Rf _i(M)_i^*)$
  in $\ydH $ by the definition of the
  braided Hopf algebra structure of $K_i^{\Rf _i(M)} \# \NA(\Rf _i(M)_i^*)$.
  By Lemma~\ref{le:F} (3)$\Rightarrow $(1),
  $\sigma _i^M(E) = F \NA(\Rf _i(M)_i) \in \cK{\Rf _i(M)}$.
  Clearly, $\Rf _i(M)_i\subset \sigma _i^M(E)$, and hence we are done.
\end{proof}

\begin{corol}\label{co:HE=HE}
  Let $M\in \ffdi$, $i\in \Ib $, and $E_1,E_2\in \cK M$.
  Assume that $M$ is $i$-finite and that $\Hilb _{E_1}=\Hilb _{E_2}$.
  Then $\Hilb _{\sigma _i^M(E_1)}=\Hilb _{\sigma _i^M(E_2)}$.
\end{corol}

\begin{proof}
  Note that $M_i=\NA (M)_{\al _i}$ is irreducible in $\ydH $.
  Hence $M_i\subset E_1$ if and only if $M_i\subset E_2$,
  since $\Hilb _{E_1}=\Hilb _{E_2}$.
  Assume first that $M_i\subset E_1$. By Lemma~\ref{le:EcapR},
  $E_l\simeq (E_l\cap K_i^M)\ot \NA
  (M_i)$ as $\ndN _0^\theta $-graded objects in $\ydH $ for $l\in \{1,2\}$.
  Hence the claim follows from Thm.~\ref{th:sigmai} and \eqref{eq:Omgrad}.
  The case $M_i\not \subset E_1$ is treated similarly.
\end{proof}

\begin{corol}\label{co:finK}
  Let $M\in \ffdi$ and let $i\in \Ib $. Then $\cK M$ is finite if and
  only if $\cK{\Rf _i(M)}$ is finite. In this case $M$ is $i$-finite, and
  $\cK M$ and $\cK{\Rf _i(M)}$ have the same cardinality.
\end{corol}

\begin{proof}
  If $M$ is not $i$-finite, then $\Rf _i(M)=M$, and hence
  $\cK M=\cK{\Rf _i(M)}$ is infinite by Lemma~\ref{le:notIfin}.
  If $M$ is $i$-finite, then
  the claim follows from the bijectivity of $\sigma _i^M$ in
  Thm.~\ref{th:sigmai}.
\end{proof}

\section{Construction of right coideal subalgebras}
\label{sec:Constr}

Let $\theta \in \ndN $ and $M=(M_1,\dots, M_\theta )\in \ffdi$.
Let
\begin{align*}
\ffdi(M)&=\{\Rf_{i_1}\cdots \Rf_{i_n}(M)\,|\,n\in \ndN _0,i_1,\dots,i_n\in \Ib \},\\
\fiso(M)&= \{r_{i_1}\cdots r_{i_n}([M])\,|\,n\in \ndN _0,i_1,\dots,i_n\in \Ib \}
\end{align*}
where $\Rf_i$ and $r_i$, $i\in \Ib $, are defined in Sect.~\ref{sec:refl}.
We say that $M$ \textit{\Fin}, if $N$ is $i$-finite for all $N\in \ffdi (M)$ and
$i\in \Ib $.
This is for example the case if $(M_1\oplus \cdots \oplus M_\theta )^{\ot m}$
is semisimple in $\ydH $ for all $m\ge 1$ and the Gelfand-Kirillov
dimension of $\NA (M)$ is finite, see \cite[Thms.\,I,\,III]{p-HeckSchn08a}.
Also, Cor.~\ref{co:finK} and the definition of $\ffdi(M)$ yield the following.
\begin{propo}\label{criterion}
  Let $M\in \ffdi $. Assume that $\cK M$ is finite.
  Then $M$ \Fin.
\end{propo}

Recall from \cite{p-AHS08} the following crucial result.

\begin{theor} \cite{p-AHS08}, \cite[Thm.~6.10]{p-HeckSchn08a}
  Let $M\in \ffdi$.
  If $M$ \Fin, then $\cC (M)=(\Ib ,\fiso(M),(r_i|_{\fiso(M)})_{i\in \Ib },
  (A^X)_{X\in \fiso(M)})$ is a Cartan scheme.
\end{theor}

Therefore, if $M\in \ffdi(M)$ and $M$ \Fin, then we may attach
the Weyl groupoid $\Wg (M):=\Wg (\cC (M))$ to $M$. Later on, for brevity we will write
$r_i$ instead of $r_i|_{\fiso(M)}$.

In this section we associate a right coideal subalgebra
$E^N(w)$ of $\NA(N)$ to any $N\in \ffdi(M)$ and $w\in \Homsto {[N]}$.
%We put special emphasis on the case
%when $\rsyst M$ is finite.

Recall that $\fie 1 \in \cK N$ for all $N\in \ffdi $.
By Thm.~\ref{th:sigmai}, $\bar{\sigma }_i^{\Rf _i(N)}:
\cK{\Rf_i(N)}\to \cK N$ is a
bijection for all $N\in \ffdi$ and $i\in \Ib $, where $N$ is $i$-finite.

\begin{defin}\label{de:E}
  Let $M\in \ffdi$. Assume that $M$ \Fin.
  For all $N\in \ffdi(M)$,
  $m\in \ndN _0$, $i_1,\dots,i_m\in \Ib $, let
  $E^N()=\fie 1$ and
  \begin{align*}
    E^N(i_1,\dots,i_m)= \bar{\sigma} _{i_1}^{\Rf_{i_1}(N)}
    \bar{\sigma} _{i_2}^{\Rf_{i_2}\Rf_{i_1}(N)}\cdots
    \bar{\sigma} _{i_m}^{\Rf_{i_m}\cdots \Rf_{i_1}(N)}(\fie 1) \in \cK N.
  \end{align*}
  \end{defin}

\begin{lemma}\label{le:Ewunique}
  Let $M\in \ffdi$. Assume that $M$ \Fin.
  Let $N\in \ffdi(M)$,
  $m\in \ndN _0$, and $i_1,\dots,i_m \in \Ib $.
%  , and $w=1_Ns_{i_1}\cdots s_{i_m}\in \Homsto N$.
  Then
  $E^N(i_1,\dots,i_m)$ is the unique element $E\in \cK N$
  with $\Hilb _E=\Hilb _{E^N(i_1,\dots,i_m)}$.
\end{lemma}

\begin{proof}
%  Thm.~\ref{th:Ewdec} gives that
%  $\Hilb _{E^M(i_1,\dots,i_m)}(t)=
%  \prod _{\lambda \in \Lambda ^M_+(w)}
%  \Hilb _{\NA (M_\lambda )}(t^\lambda )$.
%  \begin{align}\label{eq:HilbE}
%  \end{align}
  By Thm.~\ref{th:sigmai},
  \begin{align*}
    &{\sigma} _{i_m}^{\Rf_{i_{m-1}}\cdots \Rf_{i_1}(N)}\cdots {\sigma} _{i_1}^N
    (E^N(i_1,\dots,i_m))\\
    &\quad =
    {\sigma} _{i_m}^{\Rf_{i_{m-1}}\cdots \Rf_{i_1}(N)}\cdots \sigma _{i_1}^N
    \bar{\sigma} _{i_1}^{\Rf_{i_1}(N)}\cdots \bar{\sigma} _{i_m}^{\Rf_{i_m}\cdots \Rf_{i_1}(N)}
    (\fie 1)=\fie 1.
  \end{align*}
  Let $E\in \cK N$ with $\Hilb _E=\Hilb _{E^N(i_1,\dots,i_m)}$, and let
  $E'=\sigma _{i_m}^{\Rf_{i_{m-1}}\cdots \Rf_{i_1}(N)}\cdots \sigma _{i_1}^N(E)$.
%  Let $M'=r_{i_m}\cdots r_{i_2}r_{i_1}(M)$ and
%  \[ E'=\sigma _{i_m}^{r_{i_m-1}\cdots r_{i_1}(M)}
%  \cdots \sigma _{i_2}^{r_{i_1}(M)}\sigma _{i_1}^M(E).
%  \]
  By Cor.~\ref{co:HE=HE} the Hilbert series of $E'$ and of
  $\sigma _{i_m}^{\Rf_{i_{m-1}}\cdots \Rf_{i_1}(N)}\cdots \sigma _{i_1}^N
  (E^N(i_1,\dots,i_m))$ coincide. Hence $\Hilb _{E'}=1$, and therefore
  $E'=\fie 1$.
  Thus
  \[ E=\bar{\sigma} _{i_1}^{\Rf_{i_1}(N)}\cdots \bar{\sigma} _{i_m}^{\Rf_{i_m}\cdots
  \Rf_{i_1}(N)}(E')=
    \bar{\sigma} _{i_1}^{\Rf_{i_1}(N)}\cdots \bar{\sigma} _{i_m}^{\Rf_{i_m}\cdots
  \Rf_{i_1}(N)}(\fie 1)=E^N(i_1,\dots ,i_m).
  \]
  This proves the lemma.
\end{proof}

\begin{defin}\label{de:admissible}
Let $M\in \ffdi$. Assume that $M$ \Fin .
Let $N\in \ffdi (M)$, $m\in \ndN _0$ and $i_1,\dots,i_m,j\in \Ib $.
Let $T_j^{\Rf_j(N)} : \NA (\Rf_j(N)) \to \NA (N)$ be the composition of the linear maps
$$\NA (\Rf_j(N)) \xrightarrow{(\Omega_j^N)^{-1}} K_j^N \# \NA (N_j^*)
\xrightarrow{\id \otimes \varepsilon} K_j^N \subset \NA (N).$$
For all $1 \leq l \leq m$ define $\beta_l^{[N]}(i_1, \dots ,i_m)  \in \rersys
{[N]}$ and $N_l(i_1, \dots ,i_m) \in \ydH$ by
\begin{align*}
  \beta_l^{[N]}(i_1, \dots ,i_m)&= s_{i_1}^{r_{i_1}([N])}
  s_{i_2}^{r_{i_2}r_{i_1}([N])}
  \cdots s_{i_{l-1}}^{r_{i_{l-1}} \cdots r_{i_1}([N])}(\al _{i_l}),\\
  N_l(i_1, \dots ,i_m)& = T_{i_1}^{\Rf_{i_1}(N)}T_{i_2}^{\Rf_{i_2}\Rf_{i_1}(N)}
  \cdots T_{i_{l-1}}^{\Rf_{i_{l-1}} \cdots \Rf_{i_2}\Rf_{i_1}(N)} (\Rf_{i_{l-1}}
  \cdots \Rf_{i_2}\Rf_{i_1}(N)_{i_l})
\end{align*}
(where $\beta_1^N(i_1,\dots,i_m) = \al _{i_1}$, and $N_1(i_1,\dots,i_m) = N_{i_1}$).

We say that $(i_1, \dots,i_m)$ is $N$-\textit{admissible} if
%the elements
%$\beta ^N_l(i_1,\dots,i_m)$, where $1\le l\le m$, are pairwise distinct.
%Equivalently, $(i_1, \dots,i_m)$ is $N$-admissible if and only if
for all $1 \leq k \leq m-1$ and $1 \leq l \leq m-k$,
$$\al _{i_k} \neq \beta_l^{r_{i_k} \cdots r_{i_2}r_{i_1}([N])}(i_{k+1},
\dots,i_m).$$
Equivalently, $(i_1,\dots,i_m)$ is $N$-admissible if and only if
\begin{align}
  \beta _l^{[N]}(i_1,\dots,i_m)\not=-\beta_k^{[N]}(i_1,\dots,i_m)
  \quad \text{for all $1\le k<l\le m$.}
  \label{eq:admiss}
\end{align}
\end{defin}

\begin{lemma}
  \label{admissible1}
  Let $M\in \ffdi $, and assume that $M$ \Fin .
  Let $m\in \ndN _0$, $i_1,\dots,i_m\in \Ib $ and $N\in \ffdi(M)$.
  Assume that $(i_1,\dots,i_m)$ is $N$-admissible. For all $1 \leq l \leq m$ let
  $\beta_l = \beta_l^{[N]}(i_1,\dots,i_m)$, and
  $N_{\beta_l} = N_l(i_1, \dots ,i_m).$
  Then
  \begin{enumerate}
  \item $\beta_1, \dots,\beta_m$ are pairwise distinct elements in $\ndN _0^{\theta}$.
  \item For all $1 \leq l \leq m$, $N_{\beta_l} \subset E^N(i_1,\dots,i_m)
    %\subset \NA (N)
    $ is a finite-dimensional irreducible 
    subobject in $\ydH$
    of degree $\beta_l$, and $N_{\beta _l}\simeq 
    \Rf_{i_{l-1}} \cdots \Rf_{i_2}\Rf_{i_1}(N)_{i_l}$ in $\ydH $.
  \item  For all $1 \leq l \leq m$, the subalgebra $\fie \langle N_{\beta_l} \rangle$ of $\NA (N)$ generated by $N_{\beta_l}$ is isomorphic to $\NA (N_{\beta_l})$ as an algebra and as an $\ndN _0^{\theta}$-graded object in $\ydH$, where $N_{\beta_l}$ has degree $\beta_l$.
  \item The multiplication map   $\fie \langle N_{\beta _m}\rangle \otimes
      \cdots \otimes \fie \langle N_{\beta _1}\rangle \to E^N(i_1,\dots,i_m)$
      is an isomorphism of $\ndN _0^\theta $-graded objects in $\ydH $.
  \end{enumerate}
\end{lemma}

\begin{proof}
  The cases $m=0,1$ are
  clear since $N_{\al_{i_1}} = N_{i_1}$, and $E^N(i_1) = \NA(N_{i_1})$ by
  Thm.~\ref{th:sigmai}. Let $m > 1$ and assume that $(i_1,\dots,i_m)$ is
  $N$-admissible. Then $(i_2.\dots,i_m)$ is $\Rf_{i_1}(N)$-admissible. To prove
  the Lemma for $(i_1,\dots,i_m)$ we may assume by induction that (1)--(4)
  hold for $(i_2,\dots,i_m)$, that is, if we define
  \begin{align*}
  \gamma _l &= s_{i_2}^{r_{i_2}r_{i_1}([N])} s_{i_3}^{r_{i_3}r_{i_2}r_{i_1}([N])} \cdots s_{i_l}^{r_{i_l}\cdots r_{i_2}r_{i_1}([N])}(\al _{i_{l+1}}),\\
  \Rf_{i_1}(N)_{\gamma_l}&=T_{i_2}^{\Rf_{i_2}\Rf_{i_1}(N)}T_{i_3}^{\Rf_{i_3}\Rf_{i_2}\Rf_{i_1}(N)} \cdots T_{i_{l}}^{\Rf_{i_{l}} \cdots \Rf_{i_2}\Rf_{i_1}(N)} (\Rf_{i_{l}} \cdots \Rf_{i_2}\Rf_{i_1}(N)_{i_{l+1}})
  \end{align*}
  for all $1 \leq l \leq m-1$, then
  \begin{enumerate}
  \item [(a)] $\gamma_1, \dots,\gamma_{m-1}$ are pairwise distinct elements in $\ndN _0^{\theta}$.
  \item [(b)] For all $1 \leq l \leq m-1$, $\Rf_{i_1}(N)_{\gamma_l} \subset E^{\Rf_{i_1}(N)}(i_2,\dots,i_m) \subset \NA (\Rf_{i_1}(N))$ is an irreducible finite-dimensional subobject in $\ydH$ of degree $\gamma_l$.
  \item  [(c)] For all $1 \leq l \leq m-1$, the subalgebra $\fie \langle \Rf_{i_1}(N)_{\gamma_l} \rangle$ of $\NA (\Rf_{i_1}(N))$ generated by $\Rf_{i_1}(N)_{\gamma_l}$ is isomorphic to $\NA (\Rf_{i_1}(N)_{\gamma_l})$ as an algebra and as an $\ndN _0^{\theta}$-graded object in $\ydH$, where $\Rf_{i_1}(N)_{\gamma_l}$ has degree $\gamma_l$.
  \item [(d)] The multiplication map   $$\fie \langle \Rf_{i_1}(N)_{\gamma_{m-1}}\rangle \otimes
      \cdots \otimes \fie \langle \Rf_{i_1}(N)_{\gamma_{1}}\rangle \to E^{\Rf_{i_1}(N)}(i_2,\dots,i_m)$$
      is an isomorphism of $\ndN _0^\theta $-graded objects in $\ydH $.
  \end{enumerate}
  By Definition \ref{de:E}
  $$\bar{\sigma}_{i_1}^{\Rf_{i_1}(N)}(E^{\Rf_{i_1}(N)}(i_2,\dots,i_m)) =
  E^N(i_1,\dots,i_m).$$
  By assumption, $\al _{i_1} \neq \gamma_l$ for all $1 \leq l \leq m-1$.
  Hence by degree reasons it follows from (d) that
  $\Rf_{i_1}(N)_{i_1} \not\subset
  E^{\Rf_{i_1}(N)}(i_2\dots,i_m)$, since by (b) $\Rf_{i_1}(N)_{\gamma_l}$ has
  degree $\gamma_l$ for all $1 \leq l \leq m-1$, and $\gamma _1, \dots ,\gamma
  _{m-1}\in \ndN _0^\theta $ by (a).
  Then $$\bar{\sigma}_{i_1}^{\Rf_{i_1}(N)}(E^{\Rf_{i_1}(N)}(i_2,\dots,i_m)) =
  (\Omega_{i_1}^N)^{-1}(E^{\Rf_{i_1}(N)}(i_2,\dots,i_m)) \NA (N_{i_1}),$$ and
  $(\Omega_{i_1}^N)^{-1}(E^{\Rf_{i_1}(N)}(i_2,\dots,i_m)) \subset K_{i_1}^N$ by
  Thm.~\ref{th:sigmai}.
  Thus the multiplication map
  \begin{equation}\label{iso}
  (\Omega_{i_1}^N)^{-1}(E^{\Rf_{i_1}(N)}(i_2,\dots,i_m))\otimes \NA (N_{i_1}) \to E^N(i_1,\dots,i_m)
  \end{equation}
  is bijective. Moreover the restriction of the map $T_{i_1}^{\Rf_{i_1}(N)}$ to
  $E^{\Rf_{i_1}(N)}(i_2,\dots,i_m)$ is the restriction of the algebra
  isomorphism $(\Omega_{i_1}^N)^{-1}$. Therefore we obtain from (d) that the
  multiplication map
  \begin{align*}
  \fie \langle T_{i_1}^{\Rf_{i_1}(N)}(\Rf_{i_1}(N)_{\gamma_{m-1}})\rangle \otimes
      \cdots \otimes \fie \langle
      T_{i_1}^{\Rf_{i_1}(N)}(&\Rf_{i_1}(N)_{\gamma_{1}})\rangle \to\\
      &T_{i_1}^{\Rf_{i_1}(N)}(E^{\Rf_{i_1}(N)}(i_2,\dots,i_m))
 \end{align*}
      is bijective. Since
  $T_{i_1}^{\Rf_{i_1}(N)}(\Rf_{i_1}(N)_{\gamma_l}) = N_{\beta_{l+1}}$
  for all $1 \leq l \leq m-1$,
  (4) follows from the bijectivity of the map in Eq.~\eqref{iso}.

  Note that $s_{i_1}^{r_{i_1}([N])}(\gamma_l) = \beta_{l+1}$ for all $1 \leq l
  \leq m-1$. By \eqref{eq:Omgrad} the subspace
  $T_{i_1}^{\Rf_{i_1}(N)}(\Rf_{i_1}(N)_{\gamma_l}) =
  (\Omega_{i_1}^N)^{-1}(\Rf_{i_1}(N)_{\gamma_l}) = N_{\beta_{l+1}}$ of $\NA (N)$
  has degree
  $\beta_{l+1}$ for all $1 \leq l \leq m-1$. By definition $N_{\beta_1} =
  N_{i_1}$ has degree $\beta_1 = \al _{i_1}$.
  It now follows from (4) that $\beta_l
  \in \ndN _0^{\theta}$ for all $ 1 \leq l \leq m$. Thus (1) holds by
  the characterization of $N$-admissibility via Eq.~\eqref{eq:admiss}.
  Finally,
  (2) and (3) follow from (b) and (c) since $(\Omega_{i_1}^N)^{-1}$ is an
  algebra isomorphism in $\ydH$, and the change of grading is
  given by \eqref{eq:Omgrad}.
\end{proof}

\begin{lemma}\label{admissible2}
Let $M\in \ffdi$, and assume that $M$ \Fin .
Let $m\in \ndN _0$, $i_1,\dots,i_m\in \Ib $ and $N\in \ffdi(M)$. For all
$1\leq l \leq m$ let
$\beta_l = \beta_l^{[N]}(i_1,\dots,i_m)$.
\begin{enumerate}
\item Let $j \in \Ib$ and assume that $\al _j \notin \{\beta_1,\dots,\beta_m\}$ and that $(i_1,\dots,i_m)$ is $N$-admissible. Then $(j,i_1,\dots,i_m)$ is $\Rf_j(N)$-admissible.
\item Assume that $\al _j \in \{\beta_1,\dots,\beta_m\}$ for all $j \in \Ib$ and that $(i_1,\dots,i_m)$ is $N$-admissible. Then $E^N(i_1,\dots,i_m)= \NA(N)$.
\item Assume that $\cC(M)$ is the Cartan scheme of a root system. Then
  $(i_1,\dots,i_m)$ is $N$-admissible if and only if $\id_{[N]} s_{i_1} \cdots
  s_{i_m}$ is a reduced expression.
%\item  Assume that $\cC(M)$ is the Cartan scheme of a root system.  Let $m \in \mathbb{N}_0, j_1,\dots,j_m \in \Ib$ and assume that $\id_{[N]} s_{i_1} \cdots s_{i_m}= %\id_{[N]} s_{j_1} \cdots s_{j_m}$ are reduced expressions of the same morphism in the Weyl groupoid of $M$. Then $E^N(i_1,\dots,i_m)=E^N(j_1,\dots,j_m)$.
\end{enumerate}
\end{lemma}

\begin{proof} (1) holds by definition, and (2) follows from Lemma
  \ref{admissible1} (2) since $N_j \subset E^N(i_1,\dots,i_m)$ for all $j \in
  \Ib$ implies that $E^N(i_1,\dots,i_m)= \NA(N)$.

Suppose in (3) that $(i_1,\dots,i_m)$ is $N$-admissible. Then $\id_{[N]} s_{i_1}
\cdots s_{i_m}$ is a reduced expression
by Lemma \ref{admissible1} (1) and Prop.~\ref{pr:Lambda2}. Conversely,
suppose that $\id_{[N]} s_{i_1} \cdots s_{i_m}$ is a reduced expression. Then
$\beta_1,\dots,\beta_m$ are  pairwise distinct elements in
$\mathbb{N}_0^{\theta}$ by Prop.~\ref{pr:posroots}.
% Since for all $1 \leq l \leq m$ the subexpressions
% $\id_{[N]} s_{i_1} \cdots s_{i_l}$ are reduced, the same argument shows that
Hence $(i_1,\dots,i_m)$ is $N$-admissible.
\end{proof}

We recall a notion from \cite{p-HeckSchn08a}.

\begin{defin} \label{de:decomp}
Let $N \in \ffdi$. Then the
Nichols algebra $\NA(N)$ of $N$ is called {\em decomposable}
if there exist a totally ordered index set $(L,\le)$ and a family
$(W_l)_{l\in L}$ of \fd{} irreducible $\ndN _0^\theta $-graded objects in
$\ydH $
such that
\begin{equation}\label{decomposable}
  \NA (N)\simeq
  \ot _{l\in L}\NA (W_l)
\end{equation}
as $\ndN _0^\theta $-graded objects in $\ydH $, where $\deg N_i=\al _i$
for $1\le i\le \theta$.
\end{defin}

In such a decomposition the isomorphism classes of the
Yetter-Drinfeld modules $W_l$ and
their degrees in $\mathbb{N}_0 ^{\theta} $ are uniquely determined  by
\cite[Lemma 4.7]{p-HeckSchn08a}, and we define the positive roots
${\boldsymbol{\Delta}}_+^{[N]}$
and the roots ${\boldsymbol{\Delta}}^{[N]}$ of $[N]$ by
\begin{align*}
{\boldsymbol{\Delta}}_+^{[N]} &= \{ \deg(W_l) \mid l \in L \},\\
{\boldsymbol{\Delta}}^{[N]} &={\boldsymbol{\Delta}}_+^{[N]} \cup -{\boldsymbol{\Delta}}_+^{[N]}.
\end{align*}
In \cite{p-HeckSchn08a} we showed

\begin{theor}\cite[Thm.\,6.11]{p-HeckSchn08a}\label{cite}
Let $M \in \ffdi$ and assume that $M$ admits all reflections and that
$\NA(M)$ is decomposable. Then
$\NA(N)$ is decomposable for all $N \in \ffdi(M)$, and
$\Rwg (M)=(\cC (M), (\rsys X)_{X\in \fiso (M)})$
is a root system of type $\cC (M)$.
\end{theor}

\begin{corol}\label{co:decomposable2}
Let $M \in \ffdi$ and assume that $M$ admits all reflections and that
$\NA(M)$ is decomposable.
Let $N \in \ffdi(M)$ and  $\lambda \in \rersys {[N]}_+$.
\begin{enumerate}
\item There is exactly one $l(\lambda) \in L$ with $\lambda = \deg W_{l(\lambda)}$ in \eqref{decomposable}.
\item
 Let
 $P \in \ffdi(M)$,
 $w = \id_{[N]} s_{i_1}\cdots s_{i_m} \in \Hom ([P],[N])$
 be a reduced expression and $i\in \Ib $
 such that $w(\al _i) = \lambda$. Let $N_{\lambda } = N_{m+1}(i_1,\dots,i_m,i)
 \subset \NA(N)$. Then $\deg N_{\lambda } = \lambda$, $W_{l(\lambda )}\simeq 
 P_i\simeq N_{\lambda }$ in $\ydH$, and $\fie \langle N_{\lambda} \rangle \cong
 \NA(W_{l(\lambda )})$ as algebras and $\ndN _0^{\theta}$-graded objects in
 $\ydH$.
 \item Let $j\in \Ib $, $\mu =s_j^N(\lambda )$, $Q=\Rf _j(N)$,
   and assume that $\lambda \not=\al _j$. Similarly to
   $N_\lambda $ in (2), define $Q_\mu $ using a
   reduced expression of an element $w'\in \Homsto{[Q]}$. Then $N_\lambda
   \simeq Q_\mu $ in $\ydH $.
\end{enumerate}
\end{corol}

\begin{proof}
(1) is shown in \cite[Lemma 7.1 (1)]{p-HeckSchn08a}, and in (2),
$W_{l(\lambda)}\simeq P_i$ by \cite[Lemma 7.1 (2)]{p-HeckSchn08a}. Since
$\id_{[N]} s_{i_1}\cdots s_{i_m}$ is a reduced expression, $(i_1,\dots,i_m)$ is
$N$-admissible by Lemma \ref{admissible2} (3). Then
$(i_1,\dots,i_m,i)$ is $N$-admissible: Indeed,
$\lambda =\beta ^{[N]}_{m+1}(i_1,\dots,i_m,i)\in \ndN _0^\theta $, and hence it
differs from all $-\beta ^{[N]}_l(i_1,\dots,i_m,i)$ with $1\le l\le m$ by
Lemma~\ref{admissible1} (1).
The remaining part of (2)
follows from Lemma \ref{admissible1} since
$P\simeq \Rf _{i_m}\cdots \Rf _{i_1}(N)$.

Now we prove (3).
Since $\lambda \not=\al _j$, $\mu \in \rersys {[P]}_+$, and
hence $Q_\mu $ can be defined.
By (2), $Q_\mu $ is independent of the choice of $w'$. Hence we may choose
$w'=s_jw$. Then (2) yields that
$N_\lambda \simeq P_i$ and $Q_\mu \simeq P_i$ in $\ydH $, which proves
the claim.
\end{proof}

In the next lemma, which will be needed for Thm.~\ref{th:bijective},
we follow the notation in Cor.~\ref{co:decomposable2}.
For each $\lambda \in \rersys {[N]}_+$ we choose $w_\lambda 
\in \Homsto {[N]}$,
a reduced expression $\id_{[N]} s_{j_1}\cdots s_{j_n}$ of
$w_\lambda $,
and $i\in \Ib $ such that $w_\lambda (\al _i) = \lambda $. Then we define
\begin{equation}
  N_{\lambda } = N_{n+1}(j_1,\dots,j_n,i) \subset \NA(N).
\end{equation}
By Cor.~\ref{co:decomposable2}, the isomorphism class of $N_\lambda \in \ydH $
does not depend on $w_\lambda $ and $i$.

\begin{lemma}\label{le:Ewdec}
  Let $M\in \ffdi$ and $N\in \ffdi (M)$. Assume that $M$ admits all
  reflections and that $\NA(M)$ is decomposable.
  Let $m\in \ndN _0$, $i_1,\dots ,i_m\in \Ib $.
  Then there exists an isomorphism
  \[ E^N(i_1,\dots,i_m)\simeq
  \mathop{\otimes }_{\lambda \in \Lambda _+^{[N]}(i_1,\dots,i_m)}
  \NA (N_\lambda )
  \]
  of $\ndN _0^\theta $-graded objects in $\ydH $.
  \end{lemma}

\begin{proof}
  We proceed by induction on $m$. Since $E^N()=\fie 1$, the claim holds for
  $m=0$.

  Assume now that $m>0$. Let $P=\Rf_{i_1}(N)$, and assume that
  \begin{align}\label{eq:EPdec}
    E:=E^P(i_2,\dots,i_m)\simeq
    \mathop{\otimes }_{\lambda \in \Lambda _+^{[P]}(i_2,\dots,i_m)}
    \NA (P_\lambda ).
  \end{align}
  \textit{Case 1: $\al _{i_1}\notin \Lambda _+^{[P]}(i_2,\dots,i_m)$.}
  Then $P_{i_1}\not \subset E$ by degree reasons, and $(\Omega
  _{i_1}^P)^{-1}(E)\subset K_{i_1}^N$ by Cor.~\ref{co:MnotinE}. Hence
  \begin{align*}
    E^N(i_1,\dots,i_n)=\bar{\sigma }_{i_1}^P(E)
    =&(\Omega _{i_1}^N)^{-1}(E)\NA (N_{i_1}) 
    \simeq  \Big( \mathop{\otimes }_{\lambda \in
    \Lambda _+^{[P]}(i_2,\dots,i_m)}
    \NA (P_\lambda )\Big)\ot \NA (N_{i_1})
  \end{align*}
  in $\ydH $. Here the first two equations follow by definition,
  and the isomorphism is obtained from Lemma~\ref{le:EcapR} and since
  $(\Omega _{i_1}^N)^{-1}$ is a morphism in $\ydH $.
  Now, $\deg \,P_\lambda =s_{i_1}^P(\lambda )$ for all $\lambda \in 
  \Lambda ^{[P]}_+(i_2,\dots,i_m)$ by \eqref{eq:Omgrad}.
  Further, $N_{s_{i_1}^P(\lambda )}\simeq P_\lambda $ in $\ydH $ by
  Cor.~\ref{co:decomposable2} (3).
  Thus the claim follows from Lemma~\ref{le:Lambda1}.

  \textit{Case 2: $\al _{i_1}\in \Lambda _+^{[P]}(i_2,\dots,i_m)$.}
  Then $E\simeq (E\cap K_{i_1}^P)\ot \NA (P_{i_1})$ as $\ndN _0^\theta
  $-graded objects in $\ydH $ by Lemma~\ref{le:EcapR},
  and hence
  \begin{align}\label{eq:EKdec}
    E\cap K_{i_1}^P\simeq
    \mathop{\otimes }_{\lambda \in \Lambda _+^{[P]}(i_2,\dots,i_m)\setminus
    \{\al _{i_1}\}} \NA (P_\lambda )
  \end{align}
  as $\ndN _0^\theta $-graded objects in $\ydH $
  by Eq.~\eqref{eq:EPdec} and \cite[Lemma~4.8]{p-HeckSchn08a}.
  Therefore
  \begin{align*}
    E^N(i_1,\dots,i_n)=&\bar{\sigma }_{i_1}^P(E)
    =\NA (\varphi _{i_1}^N)^{-1}\big(\Omega _{i_1}^P(E\cap K_{i_1}^P)\big)
    \simeq 
    \mathop{\otimes }_{\lambda \in \Lambda _+^{[P]}(i_2,\dots,i_m)\setminus
    \{\al _{i_1}\}} \NA (P_\lambda )
  \end{align*}
  as $\ndN _0^\theta $-graded objects in $\ydH $,
  where $\deg \,P_\lambda =s_{i_1}^P(\lambda )$, see
  \eqref{eq:Omgrad}. Indeed, the first two
  equations hold by definition, and the isomorphism follows from
  Eq.~\eqref{eq:EKdec} and since $\Omega _{i_1}^P$ and $\NA (\varphi
  _{i_1}^N)$ are morphisms in $\ydH $.
  By Cor.~\ref{co:decomposable2} (3) we may replace $P_\lambda $ by
  $N_{s_{i_1}^P(\lambda )}$, and then the claim follows from
  Lemma~\ref{le:Lambda1}.
\end{proof}

\begin{theor}\label{th:bijective}
Let $M\in \ffdi$, and assume that $M$ admits all reflections and that
$\NA(M)$ is decomposable. Let $N \in \ffdi(M)$.
Then for all $w \in \Homsto {[N]}$ the right coideal subalgebra
$$E^N(w) = E^N(i_1,\dots,i_m) \subset \NA (N),$$
where $m\ge 0$ and $1\leq i_1,\dots,i_m \leq \theta $ such that
$w = \id_{[N]} s_{i_1} \cdots s_{i_m}$,
is independent of the choice of
$i_1,\dots,i_m$. The map
$$\varkappa ^N : \Hom(\Wg(M),[N]) \to \cK N, \; w \mapsto E^N(w),$$
is injective, order preserving, and order reflecting, where the set of
morphisms $\Homsto {[N]}$ is ordered by the right Duflo
  order and right coideal subalgebras are ordered with respect to inclusion.
\end{theor}

\begin{proof} To prove that $\varkappa ^N$ is a well-defined map, assume that
  $w = \id_{[N]} s_{i_1} \cdots s_{i_m} =\id_{[N]} s_{j_1} \cdots s_{j_n}$ in
  $\Hom(\Wg (M),[N])$, where $1 \leq i_1,\dots,i_m, j_1, \dots j_n \leq \theta
  $, $m,n\geq 0$. By Prop.~\ref{pr:Lambda2}, $\Lambda _+^{[N]}(i_1,\dots,i_m)=
  \Lambda _+^{[N]}(j_1,\dots,j_n)$. Hence by Lemma~\ref{le:Ewdec} the Hilbert
  series of $E^N(i_1,\dots,i_m)$ and of $E^N(j_1,\dots,j_n)$ coincide, and by
  Lemma \ref{le:Ewunique} $E^N(i_1,\dots,i_m) =E^N(j_1,\dots,j_n)$.

Let  $w,w'\in \Homsto {[N]}$ with $E^N(w)=E^N(w')$. Then
$\Lambda ^{[N]}_+(w)=\Lambda ^{[N]}_+(w')$ by Lemma~\ref{le:Ewdec}
  and \cite[Lemma 4.7]{p-HeckSchn08a}. Therefore $w=w'$
  by Prop.~\ref{pr:Lambda2}. Thus $\varkappa ^N$ is injective.

 By Thm.~\ref{th:Dorder} $\varkappa ^N$ is order preserving and order
 reflecting if and only if the following are equivalent for all $w_1,w_2\in
 \Homsto{[N]}$.
\begin{enumerate}
    \item $E^N(w_1)\subset E^N(w_2)$,
    \item $\Lambda ^{[N]}_+(w_1)\subset \Lambda ^{[N]}_+(w_2)$.
    \end{enumerate}
 To prove the equivalence of (1) and (2) we proceed by induction on $\ell
 (w_1)$. If $w_1=\id _{[N]}$, then $E^N(w_1)=\fie 1$,
 $\Lambda ^{[N]}_+(w_1)=\emptyset $ and hence
  (1) and (2) are both true. If $\ell (w_1)=1$, then $w_1=s_i^{r_i([N])}$ for
  some $i\in \Ib $. Then $\Lambda ^{[N]}_+(w_1)=\al _i$ and $E^N(w_1)=\NA (N_i)$.
  Hence (2) is equivalent to (1)
  by Lemma~\ref{le:Ewdec}, since if $N_i\subset E^N(w_2)$,
  then $E^N(w_1)=\NA (N_i)\subset E^N(w_2)$.
  Assume now that $\ell (w_1)>1$. Let $i\in \Ib $
  with $\ell (w_1)=\ell (w)+1$ for $w=s_i^Nw_1$.
  Then
  \begin{align}
    \al _i\in \Lambda ^{[N]}_+(w_1)
    \label{eq:aliinLw1}
  \end{align}
  by Cor.~\ref{co:aliinL}. Therefore
  Lemma~\ref{le:Lambda1} implies that
  (2) holds if and only if
  \begin{align} \label{eq:istep2}
    \al _i\in \Lambda ^{[N]}_+(w_2) \quad \text{and} \quad
    \Lambda ^{r_i([N])}_+(s_i^Nw_1)\subset \Lambda ^{r_i([N])}_+(s_i^Nw_2).
  \end{align}
  Since $\al _i=\Lambda ^{[N]}_+(s_i^{r_i([N])})$,
  the induction hypothesis gives that the relations in \eqref{eq:istep2}
  are equivalent to
  $N_i\subset E^N(w_2)$, $E^{\Rf_i(N)}(s_i^Nw_1)\subset E^{\Rf_i(N)}(s_i^Nw_2)$.
  Since $N_i\subset E^N(w_1)$ by Lemma~\ref{le:Ewdec} and
  by \eqref{eq:aliinLw1}, the latter is equivalent to (1) by
  Thm.~\ref{th:sigmai}.
\end{proof}

\begin{corol}\label{details}
Let $M\in \ffdi $, $N \in \ffdi(M)$, and assume that $M$ \Fin\
and that $\NA(M)$ is decomposable. Let $w_1,w_2 \in \Homsto {[N]}$
with $E^N(w_1) \subset E^N(w_2)$.
Then there are $m,n \in \ndN _0$, $m \leq n$, and
$i_1,\dots,i_n  \in \Ib $  such that
$w_1 = \id_{[N]}s_{i_1} \cdots s_{i_m}$ and $w_2 = \id_{[N]} s_{i_1} \cdots s_{i_n}$
are reduced expressions. Let
$$\beta_l = \beta_l^{[N]}(i_1,\dots,i_n),\;
  N_{\beta_l} = N_l(i_1, \dots ,i_n)$$ for all $1 \leq l \leq n$. Then
  \begin{enumerate}
  \item For all $1 \leq l \leq n$, $N_{\beta_l} \subset \NA(N)$ is an irreducible finite-dimensional subobject in $\ydH$ of degree $\beta_l$, and $\beta_k \neq \beta_l$ for all $k \neq l$.
  \item  For all $1 \leq l \leq n$, the subalgebra $\fie \langle N_{\beta_l} \rangle$ of $\NA (N)$ generated by $N_{\beta_l}$ is isomorphic to $\NA (N_{\beta_l})$ as an algebra and as an $\ndN _0^{\theta}$-graded object in $\ydH$, where $N_{\beta_l}$ has degree $\beta_l$.
 \item  The multiplication maps
 \begin{align*}
 \fie \langle N_{\beta_n} \rangle &\otimes \cdots \otimes \fie \langle N_{\beta_1}  \rangle \to E^N (w_2),\\
 \fie \langle N_{\beta_m} \rangle &\otimes \cdots \otimes \fie \langle N_{\beta_1} \rangle \to E^N(w_1)
 \end{align*}
\end{enumerate}
are bijective. In particular, $E^N(w_2)$ is a free right module over $E^N(w_1)$.
\end{corol}

\begin{proof}
By Thm.~\ref{th:bijective} $w_1 \leq_D w_2$. Hence by definition of the
Duflo order, any reduced presentation $w_1 = \id_{[N]}s_{i_1} \cdots s_{i_m}$ of
$w_1$ can be extended to a reduced presentation $w_2 = \id_{[N]} s_{i_1} \cdots
s_{i_m} \cdots s_{i_n}$  of $w_2$. Then (1),(2) and (3) follow from Lemma
\ref{admissible1} and Lemma \ref{admissible2} (3).
\end{proof}

The following results generalize properties of commutators and coproducts
of PBW generators of quantized enveloping algebras.

\begin{theor}\label{th:comm}
  Let $M\in \ffdi $, $N \in \ffdi(M)$, and assume that $M$ \Fin\
  and that $\NA (M)$ is decomposable. Let $n\in \ndN _0$,
  $i_1,\dots,i_n\in \Ib $, and $w=\id _{[N]}s_{i_1}\cdots s_{i_n}\in
  \Homsto {[N]}$
  such that $\ell (w)=n$. For all $1\le l\le n$
  let $\beta _l\in \rersys{[N]}$
  and $N_{\beta _l}\subset \NA (N)$ as in Lemma~\ref{admissible1}. Then
  in $E^N(w)$
  \begin{align}
    xy-(x\_{-1}\cdot y)x\_0\in \fie \langle N_{\beta _{l-1}}\rangle
    \fie \langle N_{\beta _{l-2}}\rangle \cdots
    \fie \langle N_{\beta _{k+1}}\rangle
    \label{eq:comm}
  \end{align}
  for all $1\le k<l\le n$, $x\in N_{\beta _k}$, $y\in N_{\beta _l}$, and
  \begin{align}
    \copr _{\NA (N)}(x)-x\ot 1\in
    \fie \langle N_{\beta _{l-1}}\rangle
    \fie \langle N_{\beta _{l-2}}\rangle \cdots
    \fie \langle N_{\beta _1}\rangle \ot \NA (N)
    \label{eq:coprN}
  \end{align}
  for all $1\le l\le n$, $x\in N_{\beta _l}$.
\end{theor}

\begin{proof}
  By Cor.~\ref{details} and the definition of the $N_{\beta _l}$,
  for the proof of Eq.~\eqref{eq:comm} it is enough to consider the case
  $k=1$, $l=n$. In that case
  $$xy-(x\_{-1}\cdot y)x\_0=(\ad x)(y)\in K_{i_1}^N \cap E^N(w)=
    \fie \langle N_{\beta _n}\rangle
    \fie \langle N_{\beta _{n-1}}\rangle \cdots
    \fie \langle N_{\beta _2}\rangle , $$
  since $y\in K_{i_1}^N$ and $x\in N_{i_1}$. Further, $\deg \,(\ad
  \,x)(y)=\beta _1+\beta _n$. But $\beta _m\not=\beta _1=\al _{i_1}$ for all
  $2\le  m\le n$, and hence $(\ad \,x)(y)$ has no summand with a factor in
  $\fie \langle N_{\beta _n}\rangle $.

  Now we prove Eq.~\eqref{eq:coprN}.
  Since $E^N(w')\in \cK N$ for all $w'\in \Homsto{[N]}$, by Cor.~\ref{details}
  it suffices to consider the case $l=n$. Since $\copr _{\NA (N)}$ is $\ndN
  _0^\theta $-graded and $\NA (N)$ is a connected coalgebra, (that is $z\in \NA
  (N)$, $\deg z=0$ implies that $z\in \fie $,) the claim follows by degree
  reasons.
\end{proof}

\begin{theor}\label{th:finite}
Let $M \in \ffdi$. Then the following are equivalent.
\begin{enumerate}
\item $\cK M$ is finite.
\item $M$ admits all reflections and the length of $N$-admissible sequences, where $N \in \ffdi(M)$, is bounded.
\item $M$ admits all reflections and $\rersys {[M]}$ is finite.
\end{enumerate}
Assume the equivalent conditions (1) -- (3). Then $\NA(M)$ is decomposable,
$\Rwg (M)=(\cC (M), (\rersys X)_{X\in \fiso(M)})$
is a finite root system of type $\cC (M)$, and for all $N \in \ffdi(M)$, the
map $$\varkappa ^N : \Hom(\Wg(M),[N]) \to \cK N$$
is bijective.
\end{theor}
\begin{proof}
Assume (2), and let $t \in \mathbb{N}$ such that $t \geq m$ for all $N \in
\ffdi(M)$ and all $N$-admissible sequences $(i_1,\dots,i_m)$. We prove (1),
(3) and the second half of the theorem.

Suppose an $m$-tuple $(i_1,\dots,i_m)$ of elements in $\Ib$ is $P$-admissible
for some $P \in \ffdi(M)$. If there exists an element $j \in \Ib$ such that
$\al _j\neq \beta_l^{[P]}(i_1,\dots,i_m)$ for all $1 \leq l \leq m$, then
$(j,i_1,\dots,i_m)$ is $\Rf_j(P)$-admissible by definition,
and $t \geq m+1$.

Let $N \in \ffdi(M)$. By the previous paragraph,
there is a largest integer $m \geq 1$ such that there is a
$P$-admissible sequence $(i_m,\dots,i_1)$ with $P =\Rf_{i_m }\cdots
\Rf_{i_1}(N)$. Hence $E^P(i_m,\dots,i_1) = \NA(P)$ by Lemma \ref{admissible2}
(2), and by Lemma \ref{admissible1} there is an isomorphism of $\ndN
_0^{\theta}$-graded objects in $\ydH $
\begin{equation}\label{maximal}
\NA(P_{\gamma_m}) \otimes \cdots \otimes \NA(P_{\gamma_1}) \cong \NA(P),
\end{equation}
where $\gamma_1,\dots,\gamma_m$ are pairwise distinct elements in
$\mathbb{N}_0^{\theta}$.
This means that the Nichols algebra of $P$ is decomposable. Hence $\NA (M)$ is
decomposable by \cite[Lemma 6.8]{p-HeckSchn08a},
and the root system $\Rwg (M)$
exists by Thm.~\ref{cite}. Moreover $\Rwg (M)$ is finite, and  for all
objects $X \in \fiso(M)$, $2m = | \rersys X|$ and $m = | \rersys X_+|$. This
proves (3).

We note that $\id_{[P]} s_{i_m} \cdots s_{i_1}$ is a reduced expression by
Lemma \ref{admissible2} (3). Therefore the inverse $\id_{[N]} s_{i_1} \cdots
s_{i_m}$ is a reduced expression. It  cannot be extended to a reduced
expression $\id_{[\Rf_j(N)]} s_js_{i_1} \cdots s_{i_m}$, $1 \leq j \leq
\theta$, by Lemma \ref{admissible1} (1) since $m = | \rersys {\Rf_j(N)}_+|$.
Thus by Lemma \ref{admissible2} (1),(2),
$E^N(w_0) = \NA(N)$, where $w_0 = \id_{[N]} s_{i_1} \cdots s_{i_m}$.

By Thm.~\ref{th:bijective}, $\varkappa ^N$ is injective.
To prove surjectivity of  $\varkappa ^N$, let $E\in \cK N$. Let $w\in \Homsto
{[N]}$ be a shortest element such that
  $E\subset E^N(w)$. Such a $w$ exists, since $\Rwg (M)$ is finite, hence
  $\Homsto {[N]}$ is finite by Lemma \ref{le:finrs} and
  $E\subset E^N(w_0)=\NA (N)$.
  We prove by induction on $\ell (w)$ that $E=E^N(w)$.

  Assume first that $\ell (w)=0$. Then $\fie 1\subset E\subset E^N(\id _{[N]})
  =\fie 1$
  and hence $E=E^N(\id _{[N]})$.

  Assume now that $\ell (w)>0$. Then $E\not=\fie 1$ by the minimality of $w$.
  By Lemma~\ref{le:rcshasMi} there exists $i\in \Ib $ such that
  $N_i\subset E$. Then $N_i\subset E^N(w)$, and hence
  $w=s_i^{r_i([N])}w'$ by Cor.~\ref{details} with $w_1=s_i^{r_i([N])}$ and
  $w_2=w$, where $w'=s_i^Nw$ and $\ell (w)=\ell (w')+1$. Further,
  $\sigma _i^N(E)\subset
  \sigma _i^N(E^N(w)) = E^{\Rf_i(N)}(w')$ by Thm.~\ref{th:sigmai}.
  Hence
  \[ \sigma _i^N(E)=E^{\Rf_i(N)}(w'') \]
  for some $w''\in \Homsto{r_i([N])}$ by induction hypothesis. Thus
  \[ E=\bar{\sigma }_i^{\Rf_i(N)}(\sigma _i^N(E))=\bar{\sigma }_i
  ^{\Rf_i(N)}(E^{\Rf_i(N)}(w'')) =E^N(s_iw''). \]
  Hence $\varkappa ^N$ is bijective.
In particular, $\varkappa ^M$ is bijective. Therefore (1) holds since
$\Homsto {[M]}$ is finite by Lemma \ref{le:finrs}.

Finally we prove (1) $\Rightarrow$ (2) and (3) $\Rightarrow$ (2). 

Assume (1) and let $t = \#(\cK M)$.
First, $M$ admits all reflections by
Prop.~\ref{criterion}. Hence $\# (\cK M) = \#(\cK N)=t$ for all $N \in
\ffdi(M)$ by Cor.~\ref{co:finK}. It follows from Lemma
\ref{admissible1} (4) that for all $N \in \ffdi(M)$ the length of
$N$-admissible sequences $(i_1,\dots,i_m)$ is bounded by $t$ since
$E^N(i_1,\dots,i_k) \neq E^N(i_1,\dots,i_l)$ for all $k,l$ with
$1 \leq k<l \leq m$.

Assume (3). Let $t = \#(\rersys M)$. Since the Weyl groupoid is connected it
follows that
$\#(\rersys M) = \#(\rersys N) = t$ for all $N \in \ffdi(M)$.
By Lemma \ref{admissible1} the length of admissible sequences is bounded by
$t$.
\end{proof}

\begin{corol} \label{co:PBWfin}
  Let $M\in \ffdi $. Assume that $M$ \Fin\ and that $\rersys{[M]}$ is
  finite. Then $\NA(M)$ is decomposable and
  $\Rwg (M)=(\cC (M), (\rersys X)_{X\in \fiso(M)})$
  is a finite root system of type $\cC (M)$.
  Let $w\in \Homsto{[M]}$ be a longest element. Let $m=\ell (w)$ and
  let $w=\id _{[M]}s_{i_1}\cdots s_{i_m}$ be a reduced decomposition.
  For each $l\in \{1,\dots,m\}$ let $\beta _l\in \rsys{[M]}_+$ and $N_{\beta
  _l}\subset \NA (M)$ as in Lemma~\ref{admissible1}.
  Then for each $l\in \{1,\dots,m\}$ the identity on
  $N_{\beta _l}$ induces an isomorphism
  $\fie \langle N_{\beta _l}\rangle \simeq \NA (N_{\beta _l})$
  of $\ndN _0^\theta $-graded objects in $\ydH $, where $N_{\beta _l}$ has
  degree $\beta _l$. Further, the multiplication map
  \[ \fie \langle N_{\beta _m}\rangle \ot \cdots
  \ot \fie \langle N_{\beta _2}\rangle \ot \fie \langle N_{\beta _1}\rangle 
  \to \NA (M) \]
  is an isomorphism of $\ndN _0^\theta $-graded objects in $\ydH $.
\end{corol}

\begin{proof}
  The first claim is proven in Thm.~\ref{th:finite}.
  The rest follows from
  Lemma~\ref{admissible1} and Lemma~\ref{admissible2} (3).
\end{proof}

\begin{corol} \label{co:WKfin}
  Let $M\in \ffdi $. Assume that $M$ \Fin\ and that $\rersys{[M]}$ is
  finite.
  Then there exist order preserving bijections between
  \begin{enumerate}
    \item \label{it:EinBH}
      the set of $\ndN _0 ^\theta $-graded right coideal subalgebras of $\NA
      (M)\#H$ containing $H$,
    \item \label{it:EinB}
      the set of $\ndN _0 ^\theta $-graded right coideal subalgebras of $\NA
      (M)$ in $\ydH $,
    \item \label{it:winHomM}
      $\Homsto{[M]}$,
  \end{enumerate}
  where right coideal subalgebras are ordered with respect to inclusion and
  the set
  $\Homsto{[M]}$ is ordered by the right Duflo
  order.
\end{corol}

\begin{proof}
  See Thm.~\ref{th:finite} for the bijection between (2) and (3)
  and Prop.~\ref{pr:rcs} for the bijection between (1) and (2).
\end{proof}

\begin{remar} \label{re:standard}
  Assume that $\Wg (M)$ is standard, that is,
  for each $N\in \ffdi (M)$ we have
  $a^N_{ij}=a^M_{ij}$ for all $i,j\in \Ib $. Then $\Homsto{[M]}$ can be
  identified with the Weyl group $W$ of $\gfrak $, see
  \cite[Thm.\,3.3(1)]{a-CH09a}.
\end{remar}

\section{Right coideal subalgebras of $U^{\ge 0}$}

In this section we are going to establish a close relationship between the maps
$T_j^{\Rf _j(M)}$, see Def.~\ref{de:admissible}, and Lusztig's automorphisms
$T_\alpha $ of quantized enveloping algebras.
Let $\gfrak $ be a finite-dimensional complex semisimple Lie algebra and let
$\Pi $ be a basis of the root system with respect to a fixed Cartan subalgebra.
Let $W$ be the Weyl group of $\gfrak $ and let $(\cdot ,\cdot )$ be
the invariant scalar product on the real vector space generated by $\Pi $
such that $(\al ,\al )=2$ for all short roots in each component. For
each $\al \in \Pi $ let $d_\al =(\al ,\al )/2$.
Let $U=U_q(\gfrak )$ be the quantized enveloping algebra of $\gfrak $
in the sense of \cite[Ch.\,4]{b-Jantzen96}.
More precisely, let
$\fie $ be a field with $\cha (\fie )\not=2$, and if $\gfrak $ has a component
of type $G_2$, then assume additionally that $\cha (\fie )\not=3$. Let $q\in
\fie $ with $q\not=0$ and $q^n\not=1$ for all $n\in \ndN $.
As a unital associative algebra, $U$ is defined over $\fie $
with generators $K_\al , K_\al ^{-1}, E_\al , F_\al $,
where $\al \in \Pi $, and relations given in \cite[4.3]{b-Jantzen96}.
%\begin{align}
%  K_\al K_\al ^{-1}=1=&K_{\al }^{-1}K_\al ,\quad
%  K_\al K_\beta =K_\beta K_\al ,\\
%  K_\al E_\beta K_\al ^{-1}=&q^{(\al ,\beta )}E_\beta ,\\
%  K_\al F_\beta K_\al ^{-1}=&q^{-(\al ,\beta )}F_\beta ,\\
%  E_\al F_\beta -F_\beta E_\al =&\delta _{\al \beta }\frac{K_\al 
%  -K_\al ^{-1}}{q_\al -q_\al ^{-1}}&
%  \label{eq:Uqgrels}
%\end{align}
%for all $\al ,\beta \in \Pi $, where $q_\al =q^{d_\al }$,
%and
%\begin{align}
%  \sum _{i=0}^{1-a_{\al \beta }}(-1)^i \qbin{1-a_{\al \beta }}{i}{\al 
%  }
%  E_\al ^{1-a_{\al \beta }-i}E_\beta E_\al ^i=&0,
%  \label{eq:Serre1}\\
%  \sum _{i=0}^{1-a_{\al \beta }}(-1)^i \qbin{1-a_{\al \beta }}{i}{\al 
%  }
%  F_\al ^{1-a_{\al \beta }-i}F_\beta F_\al ^i=&0
%  \label{eq:Serre2}
%\end{align}
%for all $\al ,\beta \in \Pi $ with $\al \not=\beta $, where $a_{\al 
%\beta }=2(\al ,\beta )/(\al ,\al )$ and
%$$
%\qbin{m}{n}{\al }=\frac{\qnum{m}{\al }\qnum{m-1}{\al }\cdots 
%\qnum{m-n+1}{\al }}{\qnum{1}{\al }\qnum{2}{\al }
%\cdots \qnum{n}{\al }},
%\quad
%\qnum{m}{\al }=\frac{q_\al ^m-q_\al ^{-m}}{q_\al -q_\al ^{-1}}
%$$
%for all $m,n\in \ndZ $, $n\ge 0$, $\al \in \Pi $.
By \cite[Prop.\,4.11]{b-Jantzen96}
there is a unique Hopf algebra structure on $U$ such that
\begin{align}
  \copr (E_\al )=&E_\al \ot 1+K_\al \ot E_\al ,&
  \cou (E_\al )=&0,\\
  \copr (F_\al )=&F_\al \ot K_\al ^{-1}+1 \ot F_\al ,&
  \cou (F_\al )=&0,\\
  \copr (K_\al )=&K_\al \ot K_\al ,& \cou (K_\al )=&1.
  \label{eq:UHopf}
\end{align}
For all $m\in \ndN $, $\al \in \Pi $ let
$q_\al =q^{d_\al }$,
$\qnum m \al =(q_\al ^m-q_\al ^{-m})/(q_\al -q_\al ^{-1})$,
$\qfact m \al =\prod _{i=1}^m
\qnum i \al $ and %(whenever $\qfact m \al \not=0$)
$E_\al \^m=E_\al ^m/\qfact{m}{\al }$,
$F_\al \^m=F_\al ^m/\qfact{m}{\al }$.
By \cite[8.14]{b-Jantzen96} there exist unique algebra automorphisms $T_\al
$, $\al \in \Pi $ of $U$ such that
\begin{align}
  T_\al (K_\al )=&K_\al ^{-1},&
  T_\al (K_\beta )=&K_\beta K_\al ^{-a_{\al \beta }},\\
  T_\al (E_\al )=&-F_\al K_\al ,&
  T_\al (F_\al )=&-K_\al ^{-1}E_\al ,\\
%  T_\al (E_\beta )=&\sum _{i=0}^{-a_{\al \beta }}(-q_\al )^{-i}E_\al \^{r-i}
%  E_\beta E_\al \^i ,&
  T_\al (E_\beta )=&
  \ad (E_\al \^{-a_{\al \beta }})E_\beta ,&
  T_\al (F_\beta )=&\sum _{i=0}^{-a_{\al \beta }}(-q_\al )^iF_\al \^i
  F_\beta F_\al \^{-a_{\al \beta }-i},
  \label{eq:Ti}
\end{align}
$\al \not=\beta $,
where $\ad $ denotes the usual left adjoint action of $U$ on itself.

As in \cite[4.6,\,4.22]{b-Jantzen96}, let $U^+$ and $U^{\ge 0}$
be the subalgebras of $U$
generated by the sets $\{E_\al \,|\,\al \in \Pi \}$ and $\{ K_\al ,K_\al
^{-1},E_\al \,|\,\al \in \Pi \}$, respectively.
%Let $\ad $ and $\twad $
%denote the usual and the twisted left adjoint action of $U$ on itself,
%see \cite[4.18,\,8.11]{b-Jantzen96}. In particular,
%\begin{align}
%  \ad \,(E_\al )u=E_\al u-K_\al uK_\al ^{-1}E_\al ,\quad
%  \twad \,(E_\al )u=uE_\al -E_\al K_\al uK_\al ^{-1}
%  \label{eq:ad}
%\end{align}
%for all $u\in U$, $\al \in \Pi $.
%Then
%\begin{align}
%  T_\al (\twad\,(E_\al \^m) E_\beta )=
%  \ad (E_\al \^{-a_{\al \beta }-m})E_\beta
%  \label{eq:TiEj}
%\end{align}
%for all $\al ,\beta \in \Pi $, $\al \not=\beta $, $0\le m\le -a_{\al \beta }$,
%see \cite[8.14 (6)]{b-Jantzen96}.

Recall that $U^+\in \ydU $ via the left action $\ad |_{U^0}$ and left
coaction
$$\coa (E_{\al _1}\cdots E_{\al _k})=K_{\al _1}\cdots K_{\al _k}\ot
E_{\al _1}\cdots E_{\al _k},\quad k\in \ndN _0,\, \al _1,\dots,\al
_k\in \Pi .$$

Identify now $\Pi $ with $\Ib =\{1,\dots,\theta \}$, where $\theta =\# \Pi $
is the rank of $\gfrak $. Let $\al \in \Pi $.
Following the notation in Sect.~\ref{sec:refl} we obtain that
$M=(\fie \,E_\beta )_{\beta \in \Pi }\in \ffdi $ and that $\Rf _\al (M)=(\Rf _\al
(M)_\beta )_{\beta \in \Pi }\in \ffdi $, where
\begin{align}
  \Rf _\al (M)_\beta =
  \begin{cases}
    \fie \,\ad (E_\al \^{-a_{\al \beta }})E_\beta & \text{if $\beta \not=\al
    $,}\\
    (\fie \,E_\al )^* & \text{if $\beta =\al $.}
  \end{cases}
  \label{eq:RfalN}
\end{align}
Let
$\vartheta _\al :M_1\oplus \cdots \oplus M_\theta \to 
\Rf _\al (M)_1\oplus \cdots \oplus \Rf _\al (M)_\theta $,
\begin{gather}
  \vartheta _\al (E_\beta )=
  \begin{cases}
    \ad (E_\al \^{-a_{\al \beta }})E_\beta & \text{if $\beta \not=\al $,}\\
    (q_\al ^{-3}-q_\al ^{-1})^{-1}E_\al ^* & \text{if $\beta =\al $}
  \end{cases}
\end{gather}
for all $\beta \in \Pi $,
where $E_\al ^*\in (\fie E_\al )^*$ such that $E_\al ^*(E_\al )=1$.
Note that
\begin{align}
  \coa (\vartheta _\al (E_\beta ))=K_\beta K_\al ^{-a_{\al \beta }}\ot \vartheta
  _\al
  (E_\beta ) \quad \text{for all $\al ,\beta \in \Pi $.}
  \label{eq:coaRN}
\end{align}
In particular, $[M]\not=[\Rf _\al (M)]$ in $\fiso $.
Nevertheless $\vartheta _\al $ is an isomorphism
of braided vector spaces. Indeed, the braiding $c$ satisfies
\begin{align*}
  c(E_\beta \ot E_\gamma )=
  &\ad (K_\beta )E_\gamma \ot E_\beta =q^{(\beta ,\gamma
  )}E_\gamma \ot E_\beta ,\\
  c(\vartheta _\al (E_\beta )\ot \vartheta _\al (E_\gamma ))=&
  \ad (K_\beta K_\al ^{-a_{\al
  \beta }})\vartheta _\al (E_\gamma )\ot \vartheta _\al (E_\beta )\\
  =&q^{(\beta -a_{\al \beta }\al ,\gamma -a_{\al \gamma }\al )}\vartheta _\al
  (E_\gamma )\ot \vartheta _\al (E_\beta )
  =q^{(\beta ,\gamma )}\vartheta _\al (E_\gamma )\ot \vartheta _\al (E_\beta )
\end{align*}
for all $\beta ,\gamma \in \Pi $ because of the $W$-invariance of $(\cdot
,\cdot )$. Hence $\NA (\vartheta _\al ):\NA (M)\to \NA (\Rf _\al (M))$ is an
isomorphism of $\ndN _0^\theta $-graded algebras and coalgebras.

\begin{propo} \label{pr:JantzenT}
  Let $\al \in \Pi $. Let
  $\iota _\al :K_\al ^M\#\NA (M_\al ^*)\to U$ be the linear map with $\iota _\al
  (x\# (E_\al ^*)^m)=(q_\al ^{-1}-q_\al ^{-3})^m x(F_\al K_\al )^m$
  for all $x\in K_\al ^M$, $m\in \ndN _0$.
  Then $\iota _\al $ is an injective algebra map, and
  the following diagram is commutative.
  \begin{align}
    \begin{CD}
      \NA (M)=U^+ @>T_\al >> U \\
      @V\NA (\vartheta _\al )VV  @AA\iota _\al A\\
      \NA (\Rf _\al (M)) @>>(\Omega _\al ^M)^{-1}> K_\al ^M\# \NA (M_\al ^*)
    \end{CD}
    \label{eq:T}
  \end{align}
\end{propo}

\begin{proof}
  We first prove that $\iota _\al $ is an algebra map. By definition, $\iota
  _\al |_{K_\al ^M\#1}$ and $\iota _\al |_{1\#\NA (M_\al ^*)}$ are algebra maps.
  By Prop.~\ref{pr:Klocfin} (i), the algebra $K_\al ^M$
  is generated by the elements $\ad
  (E_\al \^n)E_\beta $, $\beta \in \Pi \setminus \{\al \} $, $0\le n\le -a_{\al
  \beta }$, and the algebra $\NA (M_\al ^*)$ is generated by $E_\al ^*$.
  Further,
  \begin{align*}
    \partial ^L_{E_\al ^*}(\ad (E_\al \^n)E_\beta ) =&
    E_\al ^* \ad (E_\al \^n)E_\beta -(\ad (E_\al \^n)E_\beta )(K_\al
    ^{-n}K_\beta ^{-1}\cdot E_\al ^* )\\
    =& E_\al ^* \ad (E_\al \^n)E_\beta -q^{(n\al +\beta ,\al )}
    (\ad (E_\al \^n)E_\beta )E_\al ^* 
  \end{align*}
  for all $\beta \in \Pi \setminus \{\al \}$ and $0\le n\le -a_{\al \beta }$
  by Eq.~\eqref{eq:fx}, where $\partial ^L_{E_\al ^*}(x)=\langle E_\al
  ^*,x\^1\rangle x\^2$ for all $x\in K_\al ^M$.
  By \cite[8A.5(2)]{b-Jantzen96} we obtain that
  $\partial ^L_{E_\al ^*}(E_\beta ) =0$ and
  \begin{align*}
    \partial ^L_{E_\al ^*}(\ad (E_\al \^n)E_\beta ) =&
    q_\al ^{n-1}(1-q_\al ^{-2(-a_{\al \beta }-n+1)})\ad (E_\al \^{n-1})E_\beta 
  \end{align*}
  for all $\beta \in \Pi \setminus \{ \al \}$ and $1\le n\le -a_{\al \beta }$.
  Since $F_\al K_\al E_\beta =q^{(\al, \beta )}E_\beta F_\al K_\al $ for all
  $\beta \in \Pi \setminus \{\al \}$,
  it suffices to prove that the following relations hold in $U$.
  \begin{align*}
    &q_\al ^{n-1}(1-q_\al ^{-2(-a_{\al \beta }-n+1)})\ad (E_\al \^{n-1})E_\beta
    \\
    &\qquad = (q_\al ^{-1}-q_\al ^{-3})(F_\al K_\al \ad (E_\al \^n)E_\beta
    -q^{(n\al +\beta ,\al )} (\ad (E_\al \^n)E_\beta )F_\al K_\al ),
  \end{align*}
  where
  $\beta \in \Pi \setminus \{\al \} $ and $1\le n\le -a_{\al \beta }$.
  The latter equations follow from \cite[8.9 (2)]{b-Jantzen96}.

  The injectivity of $\iota _\al $ follows immediately from the triangular
  decomposition of $U$. Since all maps in the diagram~\eqref{eq:T} are algebra
  maps, it is enough to check that the diagram commutes on the algebra
  generators $E_\beta $, $\beta \in \Pi $, of $U^+$.
  This follows directly from the definitions of the maps
  involved.
\end{proof}

\begin{remar} \label{re:PBWgen}
  Prop.~\ref{pr:JantzenT} implies that the PBW basis of $U^+$ given in
  \cite[Thm.\,8.24]{b-Jantzen96} coincides with the PBW basis in
  Cor.~\ref{co:PBWfin}. Let us indicate a proof.

  First observe that $\Wg (M)$ is standard, since $M$ is a Yetter-Drinfeld
  module of Cartan type \cite[Rem.\,3.27]{p-AHS08}.
  This means that for each $N\in \ffdi (M)$ we have
  $a^N_{ij}=a^M_{ij}$ for all $i,j\in \Ib $. Hence $\Homsto{[M]}$ can be
  identified with the Weyl group $W$ of $\gfrak $, see
  \cite[Thm.\,3.3(1)]{a-CH09a}.

  Let $\al ,\beta ,\gamma \in \Pi $ such that $\ell (s_\al s_\beta s_\gamma
  )=3$. There exists a commutative diagram
  \begin{align*}
    \begin{CD}
      \NA (M) \\
      @V\NA (\vartheta _\gamma )VV \\
      \NA (\Rf _\gamma (M)) @>T_\gamma ^{\Rf _\gamma (M)}>> \NA (M)\\
      @VVV @V\NA (\vartheta _\beta )VV\\
      \NA (\Rf _\gamma \Rf _\beta (M))
      @>T_\gamma ^{\Rf _\gamma \Rf _\beta (M)}>>
      \NA (\Rf _\beta (M)) @>T_\beta ^{\Rf _\beta (M)}>> \NA (M)\\
      @VVV @VVV @V\NA (\vartheta _\al )VV\\
      \NA (\Rf _\gamma \Rf _\beta \Rf _\al (M))
      @>T_\gamma ^{\Rf _\gamma \Rf _\beta \Rf _\al (M)}>>
      \NA (\Rf _\beta \Rf _\al (M))
      @>T_\beta ^{\Rf _\beta \Rf _\al (M)}>>
      \NA (\Rf _\al (M)) @>T_\al ^{\Rf _\al (M)}>> \NA (M)\\
    \end{CD}
  \end{align*}
  such that the unlabelled vertical arrows are isomorphisms of $\ndN
  _0^\theta $-graded algebras and coalgebras. The existence of such maps can
  be concluded by considering $\ffdi $ as a category, where morphisms between
  $M,N\in \ffdi $ are bijective maps
  $f:M_1\oplus \cdots \oplus M_\theta \to N_1\oplus
  \cdots \oplus N_\theta $ preserving the braiding and satisfying
  $f(M_i)\subset N_i$ for each $i\in \Ib $. Then $\Rf _i:\ffdi \to \ffdi $
  becomes a functor, and for example the vertical arrow left to
  $\NA (\vartheta _\beta )$ is just
  $\NA (\Rf _\gamma (\vartheta _\beta ))$.
%  Then $\NA (f):\NA (M)\to \NA (N)$ is a graded isomorphism of $\ndN _0^\theta
%  $-graded algebras and coalgebras, and
%  \begin{align}
%    \NA (f)( (\ad \,x)y)=(\ad \,f(x))(\NA (f)(y))\quad
%    \text{for all $x\in M$, $y\in \NA (M)$,}
%    \label{eq:Bfad}
%  \end{align}
%  since $\NA (f)$ commutes with the braiding.
%  Let $i\in \Ib $. Define
%  \begin{align}
%    \Rf _i(f):\Rf _i(M)_1\oplus \cdots \oplus \Rf _i(M)_\theta \to
%    \Rf _i(N)_1\oplus \cdots \oplus \Rf _i(N)_\theta , \notag \\
%    \Rf _i(f)(x)=
%    \begin{cases}
%      \NA (f)(x) & \text{if $x\in \Rf _i(M)_j$, $j\not=i$,}\\
%      (f^{-1})^*(x) & \text{if $x\in \Rf _i(M)_i=M_i^*$.}
%    \end{cases}
%    \label{eq:Rif}
%  \end{align}
%  Note that $\Rf _i(f)(x)\in \Rf _i(N)_j$ for all $x\in \Rf _i(M)_j$, $j\in
%  \Ib $: Indeed, if $x\in M_i^*$, then $(f^{-1})^*(x)\in N_i^*=\Rf _i(N)_i$,
%  and if $x\in (\ad \,M_i)^{-a^M_{ij}}(M_j)$, where $j\in \Ib \setminus
%  \{i\}$, then
%  $$\NA (f)(x)\in (\ad \,N_i)^{-a^M_{ij}}(N_j)=
%  (\ad \,N_i)^{-a^N_{ij}}(N_j)=\Rf _i(N)_j$$
%  by Eq.~\eqref{eq:Bfad}. Further, $\Rf _i(f)$ is bijective,
  The PBW generators of $\NA (M)$ constructed in Cor.~\ref{co:PBWfin}
  arise as images (at the lower right corner)
  of appropriate generators of the Nichols algebras in the
  lower line.
  Similarly, the PBW generators of $\NA (M)$ arise by
  applying the maps $T_\al $, $T_\beta $, \dots appropriately
  to the algebras $\NA (M)$ at the diagonal.
  Then Prop.~\ref{pr:JantzenT} gives that the images obtained this way
  coincide.
\end{remar}

Let $w$ be an element of the Weyl group $W$,
let $m=\ell (w)$, and let $s_{\al _1}\cdots s_{\al _m}$ be a reduced
decomposition of $w$. Recall from \cite[8.24]{b-Jantzen96}
that $U^+[w]\subset U^+$ is the linear span of the products
\begin{align} \label{eq:PBWgenJ}
  E_{\beta _m}^{a_m}\cdots E_{\beta _2}^{a_2}E_{\beta _1}^{a_1},\quad
  \text{$a_1,\dots,a_m\in \ndN _0$,}
\end{align} 
where $E_{\beta _l}=T_{\al _1}\cdots T_{\al _{l-1}}(E_{\al _l})$ for all $1\le
l\le m$.

\begin{theor} \label{th:rcsU}
  The map $\varkappa $ from $W$ to the set of right coideal subalgebras of
  $U^{\ge 0}$ containing $U^0$, given by $\varkappa (w)=U^+[w]U^0$,
  is an order preserving bijection.
\end{theor}

\begin{proof}
  Subalgebras of $U^{\ge 0}$ containing $U^0$ are $\ndN _0^\theta $-graded
  by the non-degeneracy of $(\cdot ,\cdot )$ and since $q$ is not a root of
  $1$.
  Thus the claim is a special case of Cor.~\ref{co:WKfin}, see also
  Rem.~\ref{re:standard} for the interpretation of $W$ and
  Rem.~\ref{re:PBWgen} for the equality of the PBW generators in
  \eqref{eq:PBWgenJ} and in Cor.~\ref{co:PBWfin}.
\end{proof}

\begin{remar}\label{re:small}
  In view of Cor.~\ref{co:WKfin}, the claim of Thm.~\ref{th:rcsU}
  holds also for multiparameter deformations of $\gfrak $ if $q_{\al }$ is not
  a root of $1$ for all $\al \in \Pi $.
  Similarly, if $q_\al $ is a root of $1$ for all $\al \in \Pi $, then 
  the claim of Thm.~\ref{th:rcsU} holds for the (multiparameter version of)
  small quantum groups, if we restrict ourselves to $\ndN _0^\theta $-graded
  right coideal subalgebras.
\end{remar}

%\bibliography{quantum}
%\bibliographystyle{amsalpha}

\providecommand{\bysame}{\leavevmode\hbox to3em{\hrulefill}\thinspace}
\providecommand{\MR}{\relax\ifhmode\unskip\space\fi MR }
\providecommand{\MRhref}[2]{%
  \href{http://www.ams.org/mathscinet-getitem?mr=#1}{#2}
}
\providecommand{\href}[2]{#2}

\end{document}